\documentclass[11pt,reqno]{amsart}
\usepackage{amsmath}
\usepackage{amsfonts}
\usepackage{dsfont}
\usepackage{mathrsfs}
\usepackage{amssymb}
\usepackage{amsthm}
\usepackage{enumitem}
\usepackage{bbm}
\usepackage{times}
\usepackage{tabularx}
\usepackage{amsbsy}
\usepackage{mathtools}
\usepackage{tikz}
\usetikzlibrary{arrows,decorations.markings}
\usetikzlibrary{matrix}
\usetikzlibrary{graphs}
\usetikzlibrary{backgrounds}
\usepackage{setspace}     \spacing{1}
\usepackage{color}

\usepackage[colorlinks=true,linkcolor=blue,citecolor=blue]{hyperref}

\usepackage{amssymb,latexsym}
\usepackage{mathrsfs,hyperref}
\usepackage{amsmath,latexsym}
\usepackage{amscd,latexsym} 
\usepackage{graphicx}  
\usepackage[all]{xy}
\usepackage{float}
\usepackage{indentfirst}
\usepackage{amsfonts, latexsym}

\renewcommand{\theequation}{\thesection.\arabic{equation}}

\newcommand{\R}{{\mathbb R}}

\newcommand{\Z}{{\mathbb Z}}

\def\al{\alpha}
\def\eps{\varepsilon}
\def\Z{\mathbb{Z}}
\def\T{\mathbb{T}}

\newtheorem{theorem}{Theorem}[section]

\newtheorem{corollary}[theorem]{Corollary}

\newtheorem{definition}[theorem]{Definition}
\newtheorem{example}[theorem]{Example}
\newtheorem{remark}[theorem]{Remark}

\newtheorem{lemma}[theorem]{Lemma}

\newtheorem{proposition}[theorem]{Proposition}

\numberwithin{equation}{section}

\numberwithin{equation}{section}

\title[Floer homology in the cotangent bundle of a closed Finsler manifold]{Floer homology in the cotangent bundle of a closed Finsler manifold and noncontractible periodic orbits}

\author{Wenmin Gong }

\address{Laboratory of Mathematics and Complex Systems (Ministry of Education), School of Mathematical Sciences, Beijing Normal University,
    Beijing, 100875, China}

\email{ wmgong@bnu.edu.cn}
\author{Jinxin Xue}
\address{Jingzhai 310, Department of Mathematics  Yau Mathematical Sciences Center, Tsinghua University, Beijing, 100084, China}

\email{ jxue@tsinghua.edu.cn}
\begin{document}

\maketitle


\begin{abstract}
We show that the existence of noncontractible periodic orbits for compactly sup- ported time-dependent Hamiltonian on the disk cotangent bundle of a Finsler manifold provided that the Hamiltonian is sufficiently large over the zero section. We generalize the BPS capacities and earlier constructions of Weber (2006 \textit{Duke Math. J.} \textbf{133} 527--568) and other authors Biran et al (2003 \textsl{Duke Math. J.} \textbf{119} 65--118) to the Finsler setting. 

We then obtain a number of applications including: (1) generalizing the main theorem of Xue (2017 \textit{J. Symplectic Geom.} \textbf{15} 905--936) to the Lie group setting, (2) preservation of minimal Finsler length of closed geodesics in any given free homotopy class by symplectomorphisms, (3) existence of periodic orbits for Hamiltonian systems separating two Lagrangian submanifolds, (4) existence of periodic orbits for Hamiltonians on noncompact domains, (5) existence of periodic orbits for Lorentzian Hamiltonian in higher dimensional case, (6) partial solution to a conjecture of Kawasaki (2016 Heavy subsets and non-contractible trajectories (arXiv:1606.01964)), (7) results on squeezing/nonsqueezing theorem on torus cotangent bundles, etc.
\end{abstract}


\renewcommand{\theequation}{\thesection.\arabic{equation}}
\renewcommand{\thefigure}{\thesection.\arabic{figure}}

\tableofcontents

\renewcommand\contentsname{Index}

\section{Introduction}
\setcounter{equation}{0}
The main theme of this article is the existence of noncontractible periodic orbits for compactly supported time-dependent Hamiltonian systems on the unit disk cotangent bundle of a closed Finsler manifold $M$.
Closely related earlier results on finding noncontractible Hamiltonian periodic orbits on cotangent bundles go back to the papers by Gatien and Lalonde~\cite{GL} and by Biran, Polterovich and Salamon~\cite{BPS} etc.
Weber \cite{We0} extended the result in~\cite{BPS}, relaxing the condition that $M$ is either the Euclidean torus $\mathbb{T}^n$ or of negative sectional curvature, to the case that $M$ is just a closed connected Riemannian manifold. Applications of the results include the existence of periodic orbits for certain magnetic Hamiltonian systems on cotangent bundles \cite{Go,SaW}.

On the other hand, there are some other results \cite{Xu,K} on the existence of periodic orbits on cotangent bundle which do not fit into the general Riemannian framework of \cite{We0}. In particular, these results have important applications including the preservation of marked length spectrum under symplectomorphisms, the existence of noncontractible periodic orbits for Hamiltonian Lorentzian systems, etc.

The main result in this paper is a generalization of ~\cite{We0} to the Finsler setting, which provides a unified framework hence enables us to recover many results in the literature including \cite{BPS,We0,Xu} etc. Such a Finslerian generalization is by no means obvious. There are many Riemannian results that do not admit a Finslerian version, for instance, Katok's example~\cite{Ka} of a Finsler metric with only two closed geodesics for $\mathbb S^2$.

In the Riemannian case the key ingredient used by Weber to prove his main result is to construct an isomorphism between filtered Floer homology and singular homology of the level set of free loop space, which was already carried out by Salamon and Weber using heat flow method in~\cite{SW}. It seems very difficult to extend Salamon and Weber's construction to the Finsler setting, for in the later case we cannot rewrite naturally the Floer equation as a heat flow equation and estimate as in~\cite{SW}. The isomorphism between Floer homology and singular homology was firstly introduced by Viterbo~\cite{Vi} using the generating function method. 
In this article we apply a different approach to compute the filtered Floer homology of convex radial Hamiltonians (see Theorem~\ref{thm:convexradial}), in which the Legendre transform plays an essential role. One of our main observations is that on the cotangent bundle of a closed Finsler manifold $(M,F)$ the Hamiltonian $H_{F^*}={F^*}^2/2$ is naturally uniformly convex and grows quadratically in the fibers. This allows us to establish the necessary compactness result to define Floer homology (see Section~\ref{SLiouville}). However, the non-smoothness of $H_{F^*}$ on all of $T^*M$ causes the main obstruction to apply \cite{AS1}. To overcome this problem we apply the perturbation method by Lu~\cite{Lu} to produce a family of quadratic modifications which smoothen $F^2$ near the zero section maintaining the fiberwise convexity (see Section~\ref{convexq}). Another basic observation in this paper is that Hamiltonian $1$-periodic orbits $z$ of radial Hamiltonians $H$ on $T^*M$ with respect to a co-Finsler metric correspond to $1$-periodic Finslerian geodesics by the Legendre transform, and that the action $\mathscr{A}_{H} (z)$ of $z$ can be computed explicitly (see Lemma~\ref{lem:radialHamsyst}). This enables us to deform radial Hamiltonians elaborately (see the proof of Theorem~\ref{thm:symhomology}) in a given action interval without changing filtered Floer homology, and finally to compute it in terms of singular homology of the level set of free loop space.

Our result allows us to find periodic orbits for Hamiltonian systems supported in fiberwise convex domains of the cotangent bundle (Theorem \ref{thm:mainthm}) and compute the BPS capacity on such domains (Theorem \ref{ThmBPSCap}). The setting is so flexible that it leads to a number of applications. These include:
\begin{enumerate}
\item generalizing Theorem 2 of \cite{Xu} to the Lie group setting (Theorem~\ref{thm:Xue-Lie}),
\item preservation of minimal Finsler length of closed geodesics in any given free homotopy class by symplectomorphisms (Theorem \ref{coro:inv.F-length}),
\item existence of periodic orbits for Hamiltonian systems separating two Lagrangian submanifolds  (Theorem \ref{ThmLagrangian}),
\item existence of periodic orbits for Hamiltonians on noncompact domains (Theorem \ref{thm:Ncompact}),
\item existence of periodic orbits for Lorentzian Hamiltonian in higher dimensional case (Theorem \ref{ThmLorentz}),
\item partial solution to a conjecture of Kawasaki in \cite{K} (Theorem \ref{ThmKawasaki}),
\item results on squeezing/nonsqueezing theorem on torus cotangent bundles (Theorem \ref{thm:nonsqueezing}, \ref{thm:nonsqueezing'}, \ref{ThmCBPSY}).
\end{enumerate}

We expect that the list is not exhaustive and there are a lot more applications to come. Moreover, we expect that the method of this paper may have broader interests in further studying geometry, topology and dynamics on Finsler manifold using Floer theory. Finally, we refer the reader to, e.g., \cite{Ba,GG, Gu, KaO, Or} and the references therein for existence results concerning non-contractible Hamiltonian periodic orbits.

To state our result, let us first give a brief introduction to  Finsler geometry. For more comprehensive materials on Finsler geometry, we refer to the books~\cite{Ru,Sh}.

\subsection{Finsler geometry}\label{subsec:Finslergeo}
Let $M$ be a smooth manifold of dimension $n$. A \emph{Finsler metric} on $M$ is a continuous function $F:TM\to [0,\infty)$ satisfying the following properties:
\begin{enumerate}
	\item[(i)] $F$ is $C^\infty$ on $TM\setminus M\times\{0\}$.
	\item[(ii)] $F(x,\lambda v)=\lambda F(x,v)$ for every $\lambda>0$ and $(x,v)\in TM$.
	\item[(iii)] For any $(x,y)\in TM\setminus \{0\}$, the symmetric bilinear form
	$$g^F(x,y):T_xM\times T_xM\rightarrow \mathbb{R}, \quad (u,v)\to \frac{1}{2}\frac{\partial^2 }{\partial s\partial t}F^2(x,y+su+tv)\bigg|_{s=t=0}$$
	is positive definite.
\end{enumerate}
A smooth manifold $M$ endowed with a Finsler metric $F$ is called a \textit{Finsler manifold}. Let $\gamma:(a,b)\to M$ be a smooth curve. We define the $F$-length of $\gamma$ as
$${\rm len}_F(\gamma)=\int^b_aF(\gamma(t),\dot{\gamma}(t))dt.$$
Define the \emph{co-Finsler metric} $F^*$ on the cotangent bundle $T^*M$ as
$$F^*(x,\cdot):T^*_xM\to \mathbb{R},\quad F^*(x,p):=\max\limits_{F(x,v)\leq 1}\langle p,v\rangle,\quad \forall\ \ x\in M.$$
Denote by $D^FT^*M$ the unit open disk cotangent bundle, and by
$S^FT^*M$ its boundary:
$$D^FT^*M:=\big\{(x,p)\in T^*M\ \big|\ F^*(x,p)< 1\big\},$$
$$S^FT^*M:=\big\{(x,p)\in T^*M\ \big|\ F^*(x,p)=1\big\}.$$
It is obvious from (iii) that for every $x\in M$, the \textit{disk cotangent fiber}
$$(D^FT^*M)_x:=\big\{p\in T^*_xM\ \big|\ F^*(x,p)< 1\big\}$$
is a strictly convex set. Note that, on all $TM$, $L_0:=F^2$ is of class $C^{1,1}$, and is of class $C^2$ if and only if $F$ is the square of the norm of a Riemannian metric (cf.~\cite{Me}). In the latter case, every $(D^FT^*M)_x$ is an ellipsoid in $T_xT^*M$.

Conversely, given a fiberwise strictly convex domain $U\subset T^*M$ with smooth boundary and containing $M$ in its interior, a Finsler metric $F$ can be associated to it to realize $U$ as the unit disk cotangent bundle as follows
$$F(x,v):=\sup_{p\in U\cap T^*_xM}\langle p,v\rangle,\quad \forall\ v\in TM.$$
\subsection{Main result}
Let $(M,F)$ be a closed connected Finsler manifold,  and let $\pi:T^*M\to M$ denote the natural projection. Define $\lambda_0\in\Omega^1(T^*M)$ as
$$\lambda_0(\xi):=\langle p, d\pi(\xi)\rangle,\quad\forall\ \;x\in M,\;p\in T_x^*M,\; \xi\in T_{(x,p)} T^*M.$$
The cotangent bundle $T^*M$ admits a canonical symplectic form $\omega_0:=d\lambda_0$. For any $H\in C^\infty(S^1\times T^*M)$, denote $H_t:=H(t,\cdot)$. The time-dependent \textit{Hamiltonian vector field} $X_{H_t}$ is defined as $\omega_0(X_{H_t},\cdot)=-dH_t$. Set $S^1:=\mathbb{R}/\mathbb{Z}$.
Let $[S^1,M]$  denote the set of homotopy classes of free loops on $M$. Since $M$ is compact, for any non-trivial homotopy class $\alpha\in [S^1,M]$, we have
\begin{equation}\label{EqLength}l_\alpha^F:=\inf \big\{{\rm len}_F(\gamma)\ \big|\ \gamma\in C^\infty(S^1,M),\; [\gamma]=\alpha\big\}>0.\end{equation}

The set of Finslerian lengths of all periodic Finslerian geodesics in a Finsler manifold $(M,F)$ representing $\alpha$ is said to be the \emph{marked length spectrum} $\Lambda_\alpha$ of the Finsler manifold. It is closed and nowhere dense by Lemma~\ref{lem:mls}, and hence
$$l^F_\alpha=\inf\Lambda_\alpha\in \Lambda_\alpha.$$

The main result of this paper is the following.  It was conjectured by Kei Irie \cite{I}.

\begin{theorem}\label{thm:mainthm}
	Let $(M,F)$ be a closed connected smooth Finsler manifold. Let $\alpha\in [S^1,M]$ be a nontrivial free homotopy class. Then, for any $H\in C^\infty(S^1\times T^*M)$ which is compactly supported in $D^FT^*M$ and satisfies
	$$H(t,x,0)\geq l^F_\alpha\quad\forall\   t\in S^1,\; \forall\  x\in M$$
	there exists $z:S^1\to T^*M$ with $\dot{z}(t)=X_H(t,z(t))$ and $[z]=-\alpha$.
\end{theorem}

\begin{remark}
{\rm
When $(M,g)$ is a Riemannian manifold,  Theorem~\ref{thm:mainthm} was proved by Weber (Theorem~A in~\cite{We0}).
}
	
\end{remark}


\vskip 0.2in

Following \cite{BPS}, we define the BPS capacity.  For $c>0$, an open set $W\subset T^*M$ and a compact set $A\subset W$,  denote
$$\mathscr{K}_c(W,A):=\big\{H\in C_0^\infty(S^1\times W)\ \big|\sup_{S^1\times A}H\leq -c\big\}.$$
Define the BPS capacity for the pair $(W,A)$ and a nontrivial free homotopy class $\alpha$
$$C_{\rm BPS}(W,A;\alpha):=\inf\big\{c>0\ \big|\ \forall\  H \in \mathscr{K}_c(W,A), \mathscr{P}_{\alpha}(H)\neq\emptyset\big\}.$$

Our second result computes the BPS capacity for fiberwise convex sets, see also Theorem~\ref{thm:BPS capacities} (compared to Theorem~3.2.1 in \cite{BPS} and Theorem~4.3 in \cite{We0}).
\begin{theorem}\label{ThmBPSCap}
Let $M$ be a closed connected manifold. Let $U$ be a fiberwise strictly convex and compact set of $T^*M$ with smooth boundary and containing $M$ in its interior, and let $F$ be the associated Finsler metric to $U$. Then for all nontrivial free homotopy class $\alpha$, we have  $$C_{\rm BPS}(U,M; \alpha)=l_\alpha^F.$$
\end{theorem}
We remark that when $U$ is noncompact, for properly chosen $\alpha$, the capacity may also be estimated as in the proof of Theorem \ref{thm:Ncompact}.

The technical result that we obtain in this paper is the following generalization of the results of \cite[Theorem~2.9]{We0} to the Finsler setting.

\begin{theorem}[Floer homology of convex radial Hamiltonians]\label{thm:convexradial}
Let $(M,F)$ be a closed connected smooth Finsler manifold, and let $\alpha\in [S^1,M]$.
Let $f:[0,\infty)\to\mathbb{R}$ be a smooth function such that
 $f'\geq 0$, $f''\geq 0$, and $f= f(0)$ on $[0,\epsilon_f)$ for some constant $\epsilon_f>0$. Assume that
 $\lambda\in(0,\infty)\setminus\Lambda_\alpha$ and $f'(r)=\lambda$ for some $r>\epsilon_f$. Let $c_{f,\lambda}:=rf'(r)-f(r)$. Then the following holds.

\noindent(i) There exists a natural isomorphism
$$\Psi_{f}^{\lambda}: {\rm HF}^{(-\infty,c_{f,\lambda})}_*(f\circ F^*;\alpha)\longrightarrow {\rm H}_*(\Lambda_\alpha^{\lambda^2/2} M)$$
where $\Lambda_\alpha^{\lambda^2/2} M:=\{x\in C^\infty(S^1,M)|\int_{S^1}F^2(x(t),\dot{x}(t))dt<\lambda^2,\; [x]=\alpha\}$.

\noindent(ii) If $\mu\in(0,\lambda]\setminus\Lambda_\alpha$ is another slope of $f$, then the following diagram commutes:

\begin{eqnarray}
\begin{CD}\label{DC:diag0.1}
{\rm HF}^{(-\infty,c_{f,\mu})}_*(f\circ F^*;\alpha)  @>{[i^F]}>> {\rm HF}^{(-\infty,c_{f,\lambda})}_*(f\circ F^*;\alpha)\\
@V{\Psi_{f}^{\mu}}V\cong V  @V\cong V{\Psi_{f}^{\lambda}}V \\
{\rm H}_*(\Lambda_\alpha^{\mu^2/2} M) @>{[I]}>> {\rm H}_*(\Lambda_\alpha^{\lambda^2/2} M)
\end{CD}.
\end{eqnarray}

\noindent(iii) Let $g$ be another such function, then there exists an isomorphism $\Psi^\lambda_{gf}$ such that the following diagram commutes:

\begin{equation}\label{trangle:diag0.2}
\xymatrix{{\rm HF}^{(-\infty,c_{f,\lambda})}_*(f\circ F^*;\alpha)
\ar[rr]^{\Psi^\lambda_{gf}}_{\cong}\ar[dr]^{\cong}_{\Psi^\lambda_{f}}& & {\rm HF}^{(-\infty,c_{g,\lambda})}_*(g\circ F^*;\alpha)\ar[dl]_{\cong}^{\Psi^\lambda_{g}}\\ & {\rm H}_*(\Lambda_\alpha^{\lambda^2/2} M) & }.
\end{equation}

\end{theorem}
The proof is given in Section~\ref{sec:FHconvexradial}.

Our next result shows that a symplectomorphism in the identity component ${\rm Symp}_0(T^*M)$ of the symplectomorphic  group ${\rm Symp}(T^*M)$ does not change the minimal length of the Finsler closed geodesic,  though the Finsler metric is deformed.

\begin{theorem}\label{coro:inv.F-length}
Let $F_1,F_2$ be two Finsler metrics on a closed connected manifold $M$ with nontrivial free homotopy group. If $\psi\in {\rm Symp}_0(T^*M)$ is a symplectomorphsim from $D^{F_1}T^*M$ to $D^{F_2}T^*M$ such that $\psi(M\times \{0\})$ is the graph of an exact one-form on $M$, then for every non-trivial $\alpha\in [S^1,M]$, we have $l_\alpha^{F_1}=l_\alpha^{F_2}$. 
\end{theorem}
\begin{remark}\label{rmk:inv.F-length}
{\rm
The assumption that $\psi(M\times \{0\})$ is a graph of an exact one-form on $M$ is only technical and it can be replaced by some other assumptions, for instance, a stronger condition that $\psi$ is a Hamiltomorphism with $\psi(D^{F_1}T^*M)=D^{F_2}T^*M$, see~\cite{MS,Vi1}. Also, inspired by Theorem \ref{coro:inv.F-length}, Jun Zhang obtained the generalization which states that the Finsler lengths of all closed geodesics are preserved by Liouville diffeomorphisms. Though his assumption Liouville diffeomorphism is in principle stronger than our assumption here, his conclusion is clearly much stronger. The proof is to combine Theorem~\ref{thm:symhomology} with the tool of persistent homology~\cite{Zh}.}
\end{remark}

Inspired by the above result, it is interesting to ask the following question:

{\bf Problem 1:} {\it  Let $\psi\in {\rm Symp}_0(T^*M)$ be a symplectomorphsim between the unit co-disk bundles of two Finsler metrics $F_1$ and $F_2$ on a closed connected manifold $M$ with nontrivial free homotopy group. Is it true that the manifolds $(M,F_1)$ and $(M,F_2)$ are isometric?}

Here we remark that if $F_1,F_2$ are any two flat Finsler metrics on $M=\T^n$, by Remark~\ref{rmk:inv.F-length} the answer is obviously yes. This is also implied by a rigidity theorem of Benci and Sikorav~\cite{Si}.

\vskip 0.2in

The paper is organized as follows.

In Subsection \ref{SSApp}, we give a number  of applications of our main result. In Section \ref{SLiouville}, we introduce Floer theory for the Liouville domain. In Section \ref{sec:Finslergeometry}, we introduce radial Hamiltonian systems associated to a Finsler metric and compute the action of periodic orbits. In Section \ref{SFilter}, we introduce filtered Floer homology groups and study their properties such as homotopical invariance, direct and inverse limits, etc following \cite{BPS}. In Section \ref{SAS}, we show how to define Floer homology groups in the Finsler setting and solve the main technical problem--the nonsmoothness of the Finsler metric on the zero section. In Section \ref{sec:FHconvexradial}, we give the proof of Theorem \ref{thm:convexradial}. In Section \ref{SCap}, we compute the BPS capacities. In Section \ref{SProofs}, we prove Theorem \ref{thm:mainthm} and its applications in Section \ref{SSApp}. In Section \ref{STechnical}, we present the proof of all the main technical results.

\subsection{Applications}\label{SSApp}
In this section, we give a number of applications of our main result.
\subsubsection{A generalization of \cite{Xu} to the Lie group setting}

Let $G$ be an $n$-dimensional compact Lie group with Lie algebra $\mathfrak{g}$. Recall that a subgroup $T$ of $G$ is a \emph{torus} if it is   isomorphic to $(S^1)^k$ for some $k\in\mathbb{N}$, and it is called a \emph{maximal torus} if it is not properly contained in any other torus in $G$. The Lie algebra $\mathfrak{t}$ of $T$ is a maximal commutative subalgebra $\mathfrak{g}$. Let $\Gamma$ be the \emph{kernel of the exponential map} of $\mathfrak{t}$ defined by
$$\Gamma:=\{h\in \mathfrak{t}\ |\ \exp h=I\}.$$
Denote by $\mathfrak{C}$ the \emph{coroot lattice}, namely, the set of all integer linear combination of coroots (cf.~\cite{Ha}). Let
$$[\mathfrak{g},\mathfrak{g}]^0:=\{\xi\in \mathfrak{g}^*\ |\ \langle\xi,X\rangle=0,\quad \forall X\in [\mathfrak{g},\mathfrak{g}]\}$$
be the annihilator of the \emph{commutator ideal} $[\mathfrak{g},\mathfrak{g}]$,
that is, the space of elements $Z$ in $\mathfrak{g}$ that can be expressed as
$$Z=\sum_{i=1}^n c_i[X_i,Y_j]$$
for some constants $c_i$ and vectors $X_i,Y_j\in \mathfrak{g}$.

To state our result, we need the following well known fact.
\begin{lemma}[{\cite[Corollary~13.18]{Ha}}]
The fundamental group of $G$ is commutative and isomorphic to the quotient group $\Gamma/ \mathfrak{C}$,  i.e., $\pi_1(G)\cong \Gamma/ \mathfrak{C}$.
\end{lemma}

Now let us fix a basis $\{e_i\}_{i=1}^n$ in $\mathfrak{g}$. Let $\mathcal{C}$ be a cone in $\mathfrak{g}$ given by
$$\mathcal{C}=\bigg\{\sum_{i=1}^n v_ie_i\ \bigg|\ v_i\geq 0,\quad i=1,2,\ldots, n \bigg\}.$$

Denote its dual cone $\mathcal{C}^*$ in $\mathfrak{g}^*$ (the dual Lie algebra of $\mathfrak{g}$) by
$$\mathcal{C}^*=\big\{w\in \mathfrak{g}^*\ \big|\ \langle w, v\rangle \geq 0\quad \forall v\in \mathcal{C} \big\}.$$

We may identify the cotangent bundle $T^*G$ with $G\times \mathfrak{g}^*$ by means of the diffeomorphism that sends $(g,\xi)\in G\times \mathfrak{g}^*$ to $dL_{g^{-1}}^*\xi\in T^*_gG$, where $L_g:G\to G$ denotes the left-translation $L_g(x)=gx$. Under this trivialization of $T^*G$, the canonical symplectic form $\omega$ on $T^*G$ is given by:
$$\omega_{(g,\xi)}((X_1,Y_1^*),(X_2,Y_2^*))=Y_2^*(X_1)
-Y_1^*(X_2)+\xi\big([X_1,X_2]\big).$$

By taking an appropriate left-invariant Finsler (non-Riemannian)  metric on $G$,
Theorem~\ref{thm:mainthm} implies the following.
\begin{theorem}\label{thm:Xue-Lie}
Let $G$ be a compact Lie group, and $T^*G$ its cotangent bundle with the canonical symplectic form $\omega$. Let $\mathcal{C}$ be the closed cone as above and $\mathcal{C}^*$ its dual cone. Let $\alpha$ be a non-trivial homotopy class of free loops on $G$. Suppose that there exists a point $p^*\in {\rm int}~\mathcal{C}^*\cap [\mathfrak{g},\mathfrak{g}]^0$, a number $c>0$ and a vector $X\in \Gamma\cap \mathcal{C}$ such that $\alpha=[X]\in \Gamma/ \mathfrak{C}$ and
$$\langle p^*, X\rangle\leq c.$$
Then for every $H\in C_0^\infty(S^1\times T^*G)$ with compact support in $S^1\times G\times {\rm int}~\mathcal{C}^*$ satisfying
$$H(t,x,p^*)\geq c\quad \forall\  t\in S^1,\; \forall\  x\in G,$$
	 there exists
	 $z:S^1\to T^*G$ with $\dot{z}(t)=X_H(t,z(t))$ and $[z]=\alpha$.
\end{theorem}
Whenever $G=\mathbb{T}^n$, it is obvious that $\Gamma=\mathbb{Z}^n$, $[\mathfrak{g},\mathfrak{g}]=0$ (and thus $[\mathfrak{g},\mathfrak{g}]^0=\mathfrak{g}^*$).
Therefore Theorem~\ref{thm:Xue-Lie} implies the following known result.

\begin{theorem}[{\cite[Theorem~2]{Xu}}]\label{thm:Xue'theorem}
	 Let $\mathcal{C}$ be a closed cone in $\mathbb{R}^n$, and $\mathcal{C}^*$ denotes its dual cone.
	 Let $0\neq\alpha\in [S^1,T^*\mathbb{T}^n]\cong \mathbb{Z}^n\subseteq\mathbb{R}^n$. Suppose that $p^*$ belongs to the interior of the cone $\mathcal{C}^*$ (denoted by ${\rm int}~ \mathcal{C}^*$), $\alpha \in \mathcal{C}$ and that $c>0$ satisfies
	 $$\langle p^*,\alpha\rangle\leq c. $$
	 Then, for any $H\in C^\infty_0(S^1\times T^*\mathbb{T}^n)$ which is supported in $S^1\times \mathbb{T}^n\times {\rm int}~\mathcal{C}^*$ and satisfies
	 $$H(t,x,p^*)\geq c\quad \forall\  t\in S^1,\; \forall\  x\in \mathbb{T}^n,$$
	 there exists
	 $z:S^1\to T^*\mathbb{T}^n$ with $\dot{z}(t)=X_H(t,z(t))$ and $[z]=\alpha$.
\end{theorem}

\begin{remark}
{\rm The proof of Theorem \ref{thm:Xue-Lie} is based on an argument proving Theorem~\ref{thm:Xue'theorem} from Theorem~\ref{thm:mainthm} communicated to us by Irie \cite{I}.
Note that the method in the proof of \cite[Theorem~2]{Xu} based on Pozniak's Theorem~\cite{Poz} can not be used to prove Theorem~\ref{thm:Xue-Lie}. This is because the sets of critical points of the Floer action functionals of the profile functions constructed in~\cite{Xu}, in general, are not Morse-Bott for the case that $G$ is a general compact Lie group .
}
\end{remark}

\subsubsection{Periodic orbits for Hamiltonians separating two Lagrangian submanifolds}


The problem of existence of periodic orbits for Hamiltonian systems separating two Lagrangian submanifolds were studied in \cite{GL,Lee, Xu}, etc. As an application of our main theorem, we have the following result. 

\begin{theorem}\label{ThmLagrangian}
Let $M$ be a closed connected smooth manifold. Let $\sigma$ be a smooth closed one-form on $M$ whose graph does not intersect the zero section. Suppose that there exists a compact set $U\subset T^*M$ with $C^\infty$-boundary, containing the graph~$\sigma=\{(x,\sigma(x))|x\in M,\;\sigma(x)\in T^*_xM\}$ in its interior, not containing the zero section and satisfies that $U\cap T_x^*M$ is strictly convex for all $x\in M$. Then for any nontrivial free homotopy class $\alpha$, there exists a number $c(U,\alpha)>0$ such that for any Hamiltonian $H:\ S^1\times T^*M\to \R$ compactly supported in $U$ and satisfying $\min_{t,x} H(t, x,\sigma(x))\geq c(U,\alpha)$, there exists a 1-periodic orbit with homotopy class $-\alpha$.
\end{theorem}
We introduce the symplectic map $\Phi: (x,p)\mapsto (x, p-\sigma(x))$. Then apply the main theorem to the Hamiltonian $H\circ \Phi^{-1}$ and the domain $\Phi (U)$. The constant $c(U,\alpha)$ equals to $l_\alpha^F$ where $F$ is the Finsler metric associated to the  fiberwise convex set $\Phi U$.

In case when $\sigma$ is non closed, it represents a special class of magnetic Hamiltonian systems in the following way. Given a manifold $M$, we endow its cotangent bundle with the twisted symplectic form $\omega_\sigma=\omega_0+d\pi^*\sigma$, where $\omega_0$ is the standard symplectic form and $\sigma$ is a $C^\infty$ non closed 1-form.  Given a Hamiltonian $H:\ T^*M\to \R$, the Hamiltonian flow determined by the twisted symplectic form is the same as the Hamiltonian flow of the Hamiltonian $H(x, p-\sigma(x)):\ T^*M\to \R$ with the standard symplectic form. Existences of periodic orbits for this kind of magnetic systems are studied in~\cite{FS}.

Theorem~\ref{ThmLagrangian} partially generalizes~\cite[Theorem~2]{Xu}. However, there is a special feature in Theorem 2 of \cite{Xu} which is not generalized by Theorem \ref{ThmLagrangian}. Indeed, the domain $\mathcal C^*$ in Theorem \ref{thm:Xue'theorem} is noncompact while the BPS capacity $\langle \alpha,p^*\rangle$ is bounded. To recover this feature, we have the following result, which is inspired by Irie \cite{I}.
\begin{theorem}\label{thm:Ncompact}
Let $M$ be a closed connected $C^\infty$-manifold. Let $K_0\subset K_1\subset\ldots$ be a sequence of compact fiberwise strictly convex sets with $C^\infty$-boundaries in $T^*M$ and $F_0,F_1,\ldots$ be the associated Finsler metrics. Let $\alpha\in [S^1,M]$ be a nontrivial free homotopy class. Suppose that $\{K_i\}$, $\{F_i\}$ and $\alpha$ satisfy the following property:
\begin{enumerate}
\item $K_0$ contains a neighborhood of the zero section;
\item There exists a compact set $A\subset T^*M$ such that the Legendre transform of the lift of length minimizing closed $F_i$-geodesics in class $\alpha$ to $TM$ lies in $A$ for all $i=0,1,\ldots$
\end{enumerate}
Then there exists a constant $c>0$ such that for any $C^\infty$-Hamiltonian $H: S^1\times T^*M\to \R$ with compact support in $K=\lim K_i$, satisfying $\min_{t,x}H(t,x,0)\geq c$, there exists a 1-periodic orbit representing $-\alpha$.
\end{theorem}


\subsubsection{Periodic orbits for higher dimensional Minkowskian systems}
We next generalize Theorem 5 of \cite{Xu} to higher dimensional cases.

Consider the Lorentzian Hamiltonian system $H:\ (T^*\T^n,\omega_0)\to \R$, $n\geq 2$, via $$H(q,p)=\frac{1}{2}(p_1^2-(p_2^2+\ldots+p_n^2))+V(q),$$
where $V\in C^\infty(\T^n,\R)$. We normalize $V$ such that $\max V=0$.

We introduce the cone $\mathcal C^*:=\{-p^2_1+p_2^2+\ldots+p_n^2<0,\quad  p_1>0\}\subset \R^n$ and its dual cone
$$\mathcal C:=\{v\in \R^n\ |\ \langle v,p\rangle>0,\quad \forall\ \ p\in \mathcal C^*\}.$$

\begin{theorem}\label{ThmLorentz}
Let $H$ and $\mathcal C$ be as above. Then for any homology class $\al\in \mathcal C\cap H_1(\T^n,\Z)$, there exists a dense subset $S_\al\subset(0,\infty)$ such that on each energy level in $ S_\al$, the Hamiltonian system $H$ admits a periodic orbit in the homology class $\al$.
\end{theorem}
\begin{remark}
{\rm
 We consider here only Minkowski type kinetic energy with signature $(1,-1,\ldots,-1)$. The same result holds for the signature $(-1,1,\ldots,1)$ by setting $H\mapsto -H$. The main reason is that the cone $\mathcal C^*$ is convex. For other signatures, we do not have this property. The argument can be generalized to allow all possible $(1,-1,\ldots,-1)$ type Lorentzian metrics as the kinetic energy part $K(p)$. Moreover, using the argument of \cite{SaW}, we expect that the result holds for almost all positive energy levels. We do not pursue this generality here. We learned from Stefan Suhr that the result can be proved by first converting to Maupertuis Lorentzian metric $K(p)/(E-V(q))$, then apply the global hyperbolicity technique in Lorentz geometry (see \cite{T}). Moreover, the latter approach gives existence of periodic orbits for every positive energy levels.
}
\end{remark}

\subsubsection{A conjecture of Kawasaki}
The following result partially confirms Conjecture 1.3 of \cite{K}.
\begin{theorem}\label{ThmKawasaki}
Given a homology class $\al=(\al_1,\ldots,\al_n)\in H_1(\T^n,\Z)\setminus\{0\}$, for any Hamiltonian function $H:S^1\times  T^*\T^n\to \R$ with support in $S^1\times \T^n\times (\prod_{i=1}^n(-R_i,R_i))$ and satisfying that $$\min_{q,t}H(t,q,0)\geq \sum_{i=1}^n R_i|\al_i|,$$ there exists a 1-periodic orbit in the homology class $\al$.
\end{theorem}

\subsubsection{Symplectic squeezing vs. nonsqueezing on $T^*\mathbb{T}^n$}

Symplectic nonsqueezing problems on $T^*\T^n$ were studied by many authors, c.f. \cite{Si,MMT} etc. In this section, we show that BPS capacity provides obstructions for symplectic embeddings on $T^*\T^n$. We have the flexibility to choose the fiber convex domain as well as the free homotopy type to compute the BPS capacity. We first state two nonsqueezing type results explaining the way of applying BPS capacity to the nonsqueezing problem. We do not claim the originality of the results since they can be also proved using Sikarov's rigidity theorem \cite{Si}. After that, we show that the BPS capacity may be infinity simultaneously for all free homotopy classes in which case it does not provide any obstruction to the embedding problem and squeezing can indeed occur.

\vskip 0.2in

Let $\Delta^n(r)$ denote the interior of the $n$-dimensional simplex with the $n+1$ vertices $(0,\ldots,0)$, $(r,0,\ldots,0)$, $\dots$, $(0,\ldots,0,r)$. Denote open subsets $B^n,Z^n\subset \mathbb{R}^n$ and $P^{2n},Y^{2n}\subset (T^*\mathbb{T}^n,\omega_0)$:
\begin{equation}
\begin{aligned}
B^n(r):&=\{(x_1,\ldots,x_n)\in \mathbb{R}^n\ |\ x_1^2+x_2^2+\ldots+x_n^2\leq r\},\\
Z^n(r):&=\{(x_1,\ldots,x_n)\in \mathbb{R}^n\ |\ x_1^2+x_2^2\leq r\},\\
P^{2n}(r):&=\mathbb{T}^n\times \Delta^n(r),\\
Y^{2n}(r):&=\mathbb{T}^n\times(0,r)\times (\mathbb{R}^+)^{n-1},
\end{aligned}
\end{equation}
where $r\in\mathbb{R}^+:=(0,\infty)$ and $T^*\mathbb{T}^n$ is naturally symplectically identified to $\mathbb{T}^n \times \mathbb{R}^n$.

Let $U,V$ be two open subset of $T^*M$. Recall that a symplectic embedding $\psi:U\to V$ is called \textit{$\tilde{\pi}_1(M)$-trivial} if $\psi_*\alpha=\alpha$ for any $\alpha\in[S^1,M]$.

\begin{theorem}\label{thm:nonsqueezing}
There is a $\tilde{\pi}_1(M)$-trivial symplectic embedding $\phi:P^{2n}(s)\to Y^{2n}(r)$ such that for every $u\in \Delta^n(r)$, $\phi(\mathbb{T}^n\times \{u\})$ is a $C^\infty$ section in $T^*\mathbb{T}^n$ if and only if $s\leq r$.
\end{theorem}

Under a weaker topological condition, namely, the induced map $\phi_*:{\rm H}_*(P^{2n}(s);\mathbb{Z}) \to {\rm H}_*(Y^{2n}(r);\mathbb{Z})$ is an isomorphism,
Maley, Mastrangeli and Traynor~\cite{MMT} proved the above nonsqueezing theorem for any symplectic embedding $\phi:P^{2n}(s)\to Y^{2n}(r)$ and used it to study symplectic packing problems.

\begin{theorem}\label{thm:nonsqueezing'}
There is a $\tilde{\pi}_1(M)$-trivial symplectic embedding $\phi:\T^n \times B^{n}(s)\to \T^n \times Z^{n}(r)$ such that for a sequence $\{ u_i\}\subseteq B^{n}(s)$ with $\lim u_i=(s,0\ldots,0)$ every $\phi(\mathbb{T}^n\times \{u_i\})$ is a $C^\infty$ section in $T^*\mathbb{T}^n$ if and only if $s\leq r$.
\end{theorem}

In general we have infinitely many choices for the BPS capacity by varying the homotopy type. However, they may be infinity simultaneously and provide no obstruction for the symplectic embedding.

Let $v$ be a unit vector in $\R^n$, we define the following tilted cylinder
$$Y^{2n}(r, v):=\T^n\times (-r,r)v\times v^\perp.$$

We have the following theorem.
\begin{theorem}\label{ThmCBPSY}
Let $v$ be a unit vector in $\R^n$. Then it holds the following.
\begin{enumerate}
\item
If $v$  is the scalar multiple of an integer vector $\al\in \Z^n\setminus\{0\}$.
Then $$C_{\rm BPS}(Y^{2n}(r, v),\T^n,\pm\alpha)=r\|\alpha\|.$$
And for all $\beta\in \Z^n\setminus \mathrm{span}\{\alpha\}$ and all $r>0$, we have
$$C_{\rm BPS}(Y^{2n}(r, v),\T^n,\beta)=\infty.$$
\item
If $v$ is not a scalar multiple of any integer vector.
Then for all $\alpha\in H_1(\T^n,\Z)\setminus\{0\}$ and all $r>0$, we have
$$C_{\rm BPS}(Y^{2n}(r, v),\T^n,\alpha)=\infty.$$
\end{enumerate}
\end{theorem}

We remark that similar statements can be formulated for $Y$ of the form $\T^n\times D^{k}\times (D^{k})^\perp$, where $D^k$ is a disk of dimension $k$ lying in a $k$-dimensional plane in $\R^n$.  This theorem also explains the necessity of assumption 2 of Theorem \ref{thm:Ncompact}.

In the case of infinite capacity for all classes, squeezing may occur. We illustrate the squeezing by the following example.

{\bf Example: }

Let $A:=\left[\begin{array}{cc}
2&1\\
1&1\end{array}\right]$ and the symplectic map $\Phi_A:\ T^*\T^2\to T^*\T^2$ via $(x,y)\mapsto (Ax, A^{-1}y)$. We next denote by $v$ the eigenvector associated to the smaller eigenvalue of $A$.
Then introduce the cylinder $Y^4(r, v)$. Let $U$ be any bounded domain in $\R^2$, then the image of $U$ under $A^{-n} $ is expanded along $v$ and contracted along $v^\perp$. So for any $r>0$, there exists $N$ such that $\Phi^n(U)\subset Y^4(r,v)$ for all $n>N$.

\vskip 0.2in

Note that embeddings produced in this way can only be induced by a matrix $A\in \mathrm{PSL}_2(\Z)$ with tr$(A)>2$. Its eigenvalue has to be an algebraic number solving $x^2-tr(A)x+1=0$.
The vector $v$ has to be an eigenvector. 

We are naturally led to the following question.

{\bf Problem 2:} {\it Let $v$ be an irrational vector but not an eigenvector. Is it true that for all $r>0$, there exists a symplectic map $\Phi$ on $T^*\T^2$ such that $\Phi(P^4(1))\subset Y^4(r,v)?$ }

Since we have the capacity $C_{\rm BPS}(Y^4(r,v),\T^2,\alpha)=\infty$ for all $\alpha\in \Z^2\setminus\{0\}$, there is no obstruction provided by the BPS capacity.

\section*{Acknowledgments}
We greatly appreciate the communications \cite{I} with Professor Kei Irie, which is the main motivation of the work. We would like to thank Professor Stefan Suhr for the diskussion on Lorentz geometry. We would like to thank Professor Guangcun Lu for precious comments, and for pointing out his work about closed Finsler geodesics in~\cite{Lu}. We also thank Dr. Jun Zhang for helpful diskussions on the first draft of this paper and for providing us the proof of a generalized version of Theorem~\ref{coro:inv.F-length}.

W. G. is supported by  the grant NSFC 11701313 and the Fundamental Research Funds for the Central Universities grant 2018NTST18. J. X. is supported by NSFC in China (Significant project No.11790273) and Beijing Natural Science Foundation (Z180003).

\section{Floer homology on a Liouville domain $(X,\lambda)$}\label{SLiouville}
\setcounter{equation}{0}

\subsection{Basic definitions and convexity results}\label{subsect:convexity}

Let $(X,\lambda)$ be a \emph{Liouville domain}, meaning that $X$ is a $2n$-dimensional compact manifold with boundary $\partial X$, $\lambda$ is a $1$-form on $X$ such that $d\lambda$ is a symplectic form on $X$, and $\lambda\wedge  (d\lambda)^{n-1}>0$ on $\partial X$. There exists a vector field on $X$, called the \emph{Liouville vector field} $Z$ of $(X,\lambda)$, which points transversely outward at $\partial X$ and satisfies
$$\mathcal{L}_Zd\lambda=d\lambda.$$
Then $(\partial X,\theta:=\lambda|_{\partial X})$ is a contact manifold. The \emph{Reeb vector field} $R$ is defined by $d\theta(R,\cdot)=0$ and $\theta(R)=1$.
The \emph{action spectrum}
$${\rm Spec}(X,\lambda):=\bigg\{\int_\gamma\lambda\ \bigg|\
\gamma\;\hbox{is a periodic Reeb orbit of}\;R\bigg\},$$
is closed and nowhere dense in $\mathbb{R}$. Moreover, $\lambda\wedge  (d\lambda)^{n-1}>0$ implies ${\rm Spec}(X,\lambda)\subset (0,\infty)$.

The vector field $Z$ gives rise to an embedding
$\phi:(0,1]\times\partial X\to X$ satisfying
$$\phi(1,z)=z\quad\hbox{and}\quad \rho\partial_\rho\phi(\rho,z)=Z(\phi(\rho,z)).$$
It is easy to verify that $\phi^*\lambda=\rho\theta$, and thus $\phi^*d\lambda=d(\rho\theta)$. So a neighborhood of $\partial X$ in $X$
can be symplectically identified with the symplectic manifold $((0,1]\times\partial X, d(\rho\theta))$. By attaching a cylindrical end to
$(0,1]\times\partial X$ we obtain a \emph{completion} of $(X,\lambda)$ defined by
$$(\widehat{X},\widehat{\lambda}):=(X,\lambda)\cup_{\partial X}
([1,\infty)\times \partial X,\rho\theta).$$
Obviously, $(\widehat{X},d\widehat{\lambda})$ is an open symplectic manifold.
Let $H_t$ be a smooth time-dependent Hamiltonian on $\widehat{X}$. The \emph{Hamiltonian vector field} $X_H$ associated to $H$ is defined by $-dH_t=d\widehat{\lambda}(X_H,\cdot)$. Denote by $\phi_H^t$ the flow of $X_H$.
Let $\widehat{Z}$ be the extended Liouville vector field of $Z$ satisfying $\widehat{Z}=Z$ on $X$ and $\widehat{Z}=\rho\partial\rho$ on $(0,\infty]\times\partial X$. $\widehat{Z}$ is complete since its flow exists for all times.

Let $J_t$ be the $t$-dependent $1$-periodic smooth \emph{$d\widehat{\lambda}$-compatible almost complex structure} on $\widehat{X}$, that is, $\langle\cdot, \cdot\rangle:=d\widehat{\lambda}(J_t\cdot,\cdot)$ is a loop of Riemanian metrics on $\widehat{X}$ and $J^2=-I$. The corresponding norm is defined as $|\cdot|_{J_t}:=\sqrt{\langle\cdot, \cdot\rangle}$.   The set consisting of all such $J_t$ is denoted by $\mathcal{J}$. We call an almost complex structure $J_t\in \mathcal{J}$ \emph{of contact type} on $[\rho_0,\infty)\times\partial X$ for some $\rho_0>0$ if
$$d\rho\circ J_t=\widehat{\lambda}\quad\hbox{on}\;[\rho_0,\infty)\times\partial X.$$
Equivalently, $J_t$ preserves the symplectic splitting
$$T_{(\rho,z)}\widehat{X}=\hbox{ker}\lambda(z)\oplus \mathbb{R}R(z)\oplus \mathbb{R}\widehat{Z}(\rho,z)\quad \forall\  (\rho,z)\in [\rho_0,\infty)\times\partial X$$
and
\begin{equation}\label{e:contactJ}
J_tR=\widehat{Z},\quad J_t\widehat{Z}=-R.
\end{equation}
Denote by $\mathcal{J}(\widehat{X},\widehat{\lambda})$ the subset of $\mathcal{J}$ which consists of almost complex structures  of contact type at infinity.

Disk cotangent bundles of a closed Finsler manifold are examples of Liouville domains which we are mainly interested in throughout this paper. For more examples of Liouville domains, we refer readers to the survey article~\cite{Se} by Seidel.

\begin{example}\label{ex:diskcotanb}
{\rm Let $M$ be a closed $C^\infty$ manifold with a Finsler metric $F$, which induces a co-Finsler metric $F^*$ defined on the cotangent bundle $T^*M$ (see Subsection~\ref{subsec:Finslergeo}). Denote by $W:=D^FT^*M$ the unit cotangent disk bundle.
The restriction $\theta=\lambda_0|_{\partial W}$ of the standard Liouville form $\lambda_0:=pdx$ to the boundary $\partial W$ is a contact form on $\partial W$.
So $(W,\lambda_0)$ is a Liouville domain. The standard Liouville vector field
$Z=p\frac{\partial}{\partial p}$ is transverse to $\partial W$. The flow
$\varphi$ of $Z$ induces the diffeomorphism
$$(0,\infty)\times\partial W\rightarrow T^*M\setminus\{0\},\quad (\rho,z)\to \varphi_{\log \rho}(z)$$
which has the inverse
$$T^*M\setminus\{0\}\rightarrow (0,\infty)\times\partial W, \quad z\to \bigg(F^*(z),\frac{z}{F^*(z)}\bigg).$$
Then $T^*M$ is naturally identified with its completion
$$\widehat{W}:=W\cup_{\partial W}
[1,\infty)\times \partial W.$$
Under this identification, on
$T^*M\setminus D^F_rT^*M\cong (r,\infty)\times\partial W$ (where $D^F_rT^*M$ denotes the disk cotangent bundle of radius $r$,  see~(\ref{e:Finslerdisk})), a function $f\in C^\infty (T^*M,\mathbb{R})$ can be written as
$f(\rho,z)$
with respect to the variables $\rho$ and $z$.
}
\end{example}

To define Floer homology on $(\widehat{X},d\widehat{\lambda})$, we will need the $C^0$-bounds for solutions
 $u:\mathbb{R}\times S^1\to \widehat{X}$ to the  \emph{$s$-dependent Floer equation}
\begin{equation}\label{e:sFloereq}
\partial_su+J_{s,t}(u)(\partial_tu-X_{H_{s,t}}(u))=0.
\end{equation}

 Based on the method of the maximum principle, various $C^0$-bounds for the solutions of Floer equations given by the almost complex structures of contact type are established in the literature, see, for instance, \cite{Ir,Se,Vi}.
For our purpose we will show the following convexity results for those Hamiltonians which are constant, linear and/or superlinear (with respect to the $\rho$-variable) at infinity in a unified way.

\begin{lemma}\label{lem:convexity}
Let $(X,\lambda)$ be a Liouville domain.
Let $\{H_{s,t}\}_{s,t\in\mathbb{R}\times S^1}$ be a smooth family of Hamiltonians on $\widehat{X}$, and $\{J_{s,t}\}_{s,t\in\mathbb{R}\times S^1}$
be a smooth family of elements in $\mathcal{J}$. Assume that there exists a constant $\rho_0>0$ and two functions $f,g\in C^\infty(\mathbb{R})$ such that for every $(s,t)\in\mathbb{R}\times S^1$,
\begin{itemize}
  \item $H_{s,t}(\rho,z)=f(s)\rho^\mu+g(s)$ on $ [\rho_0,\infty)\times\partial X$, where $\mu\geq 1$ and $f'(s)\geq 0$, or $\mu=0$ without further restrictions on $f$ and $g$.
  \item  $J_{s,t}$ is of contact type on $[\rho_0,\infty)\times\partial X$.
\end{itemize}
If $u:\mathbb{R}\times S^1\to \widehat{X}$ satisfies (\ref{e:sFloereq}) and
 $u^{-1}([\rho_1,\infty)\times\partial X)$ is bounded for some $\rho_1>\rho_0$, then $u(\mathbb{R}\times S^1)\subset \widehat{X}\setminus ((\rho_1,\infty)\times\partial X)$.
\end{lemma}
This lemma is proved in Section \ref{SSMP}.

\begin{remark}\label{rem:actionspectrum}
{\rm Suppose that the Hamiltonians $H$ are \emph{radial outside a compact set} of $(\widehat{X},d\widehat{\lambda})$, that is,
$$H(\rho,z)=h(\rho),\quad \forall\ (\rho,z)\in[\rho_0,\infty)\times\partial X$$
for some $\rho_0>0$, where $h:[\rho_0,\infty)\to \mathbb{R}$ is a smooth function. The Hamiltonian vector field $X_H$ has the form
$$X_H(\rho,z)=h'(\rho)R(z)\quad \hbox{on}\; [\rho_0,\infty)\times\partial X.$$
Thus a $T$-periodic Reeb orbit of $R$ gives rise to a $1$-periodic orbit of $X_H$ in $\{r_0\}\times \partial X$ if and only if $T=|h'(r_0)|$, and conversely every $1$-periodic orbit of $X_H$ in $\{r_0\}\times \partial X$ for $r_0\geq \rho_0$ is given by a $T$-periodic orbit of $R$ with $T=|h'(r_0)|$.
}
\end{remark}

\subsection{Floer homology of admissible Hamiltonians}\label{subsec:Floerhomology}

For $H\in C^\infty(S^1\times \widehat{X})$, we denote $\mathscr{P}(H)\subseteq \widehat{X}$ as the set of all $1$-periodic orbits of $X_{H_t}$:
$$\mathscr{P}(H)=\big\{z\in C^\infty(S^1,\widehat{X})\big| \dot{z}(t)=X_{H_t}(z(t))\}.$$

Define  the set of \emph{admissible Hamiltonians} $\mathcal{H}_{\rm ad}$ to consist of all smooth $H:S^1\times \widehat{X}$ with the following properties:
\begin{itemize}
 \item $H$ is linear at infinity, meaning that there exist
 $\rho_0>0$, $\mu_H>0$ and $a_H\in\mathbb{R}$ such that
 $$H_t(\rho,z)=\mu_H\rho+a_H\quad\hbox{on}\; [\rho_0,\infty)\times\partial X$$
 for all $t\in S^1$. Here $\mu_H$ is required to satisfy $\mu_H\notin {\rm Spec}(X,\lambda)$.
  \item All elements $z\in \mathscr{P}(H)$ are \emph{non-degenerate}, that is, the linear map
$$d\phi_H^1(z(0)):T_{z(0)}\widehat{X}\rightarrow T_{z(0)}\widehat{X}$$ does not have $1$ as an eigenvalue.
\end{itemize}

Observe that since $\mu_H\notin {\rm Spec}(X,\lambda)$, the linear behavior of $H_t:=H(t,\cdot)$ means that there are no $1$-periodic orbits in $[\rho_0,\infty)\times\partial X$, hence there are only finitely many $1$-periodic orbits in total.

For a free homotopy class $\alpha\in [S^1,\widehat{X}]$, denote $\Lambda_\alpha \widehat{X}:=\{z\in C^\infty(S^1,\widehat{X})|[z]=\alpha\}$ and $\mathscr{P}_\alpha(H):=\mathscr{P}(H)\cap \Lambda_\alpha \widehat{X}$. Note that if $\alpha\neq 0$, there is no canonical way to assign $\mathbb{Z}$-valued Conley-Zehnder index to elements $z\in \mathscr{P}_\alpha(H)$. However, for $(\widehat{X},d\lambda)=(T^*M,dp\wedge dx)$, one can follow \cite{AS0} to define $\mathbb{Z}$-valued Conley-Zehnder index. In this paper we only work on $(X,\lambda)=(D^FT^*M,pdx)$, and so let us assume that the $\mathbb{Z}$-valued Conley-Zehnder index of noncontractible Hamiltonian periodic orbits is well-defined.

The \emph{Floer action functional} $\mathscr{A}_{H}:\Lambda_\alpha \widehat{X}\to \mathbb{R}$ is defined by
\begin{equation}\label{e:AF0}
\mathscr{A}_{H}(z)=\int_{S^1}z^*\lambda-\int^1_0H(t,z)dt.
\end{equation}

It is easy to check that a loop $z\in\mathscr{P}_{\alpha}(H)$ if and only if $z$ is a critical point of $\mathscr{A}_{H}$ on $\Lambda_\alpha \widehat{X}$. The set of values of $\mathscr{A}_H$ on $\mathscr{P}_{\alpha}(H)$ is called the \emph{action spectrum} with respect to $\alpha$, and we denote it by
$$\hbox{Spec}(H;\alpha):=\{\mathscr{A}_{H}(x)\big|
x\in \mathscr{P}_{\alpha}(H)\}.$$

For every $k\in \mathbb{Z}$, we define the \emph{Floer chain group} ${\rm CF}_k(H;\alpha)$ to be the free $\mathbb{Z}_2$-module generated by the elements $z$ of $\mathscr{P}_{\alpha}(H)$ with Conley-Zehnder index $\mu_{CZ}(z)=k$.

Assume that $H\in \mathcal{H}_{\rm ad}$ and $J\in\mathcal{J}(\widehat{X},\widehat{\lambda})$. Given critical points $z_\pm\in \mathscr{P}_{\alpha}(H)$ ,  denote by
$$\widehat{\mathcal{M}}_\alpha(z_-,z_+,H,J)\subseteq C^\infty(\mathbb{R}\times S^1,\widehat{X})$$
the set of smooth maps $u:\mathbb{R}\times S^1\to \widehat{X}$ that satisfy the Floer equation
\begin{equation}\label{e:Feq0}
\partial_su+J(u)(\partial_tu-X_{H}(u))=0
\end{equation}
and the asymptotic conditions
\begin{equation}\label{e:AC0}
\lim\limits_{s\to-\infty}u(s,t)=z_-\quad\hbox{and}\quad \lim\limits_{s\to+\infty}u(s,t)=z_+
\end{equation}
uniformly in $t\in S^1$. Note that each moduli space $\widehat{\mathcal{M}}_\alpha(z_-,z_+,H,J)$ carries a free $\mathbb{R}$-action given by $(\tau\cdot u)(s,t):=u(s-\tau,t)$, denote by $\mathcal{M}_\alpha(z_-,z_+,H,J)$ the quotient space under this action.
For generic $J\in \mathcal{J}(\widehat{X},\widehat{\lambda})$, $\mathcal{M}_\alpha(z_-,z_+,H,J)$ carries a smooth manifold structure with dimension $\mu_{CZ}(z_-)-\mu_{CZ}(z_+)-1$. Such
$J$ is called a \emph{regular} almost complex structure for $H$. 

If $\mu_{CZ}(z_-)=\mu_{CZ}(z_+)+1$,
$\mathcal{M}_\alpha(z_-,z_+,H,J)$ is compact and thus consists of finitely many points. The boundary operator $$\partial_k=\partial_k(H,J):{\rm CF}_k(H;\alpha)\rightarrow {\rm CF}_{k-1}(H;\alpha)$$
is defined by
$$\partial_k(z_-)=\sum\limits_{z_+\in \mathscr{P}_{\alpha}(H)_{k-1}}n(z_-,z_+)z_+,\quad x\in \mathscr{P}_{\alpha}(H)_k$$
with $n(z_-,z_+):=\sharp_2\big(\mathcal{M}_\alpha(z_-,z_+,H,J)\big)$. The standard transversality and gluing arguments (see~ e.g. ~\cite{Sa}), combining with a $C^0$-bounds for solutions of the Floer equation (see Lemma~\ref{lem:convexity}) show that
$$\partial_{k-1}\circ\partial_{k}=0.$$
So $\{{\rm CF}_*(H;\alpha),\partial_*(H,J)\}$ is a chain complex, called the \emph{Floer chain complex}. The \emph{Floer homology} ${\rm HF}_*(H,J;\alpha)$ is defined to be the homology of such a chain complex. Moreover, the homology group of $\{{\rm CF}_*(H;\alpha),\partial_*(H,J)\}$ does not depend on $J$, and we denote it by ${\rm HF}_*(H;\alpha)$.

\section{Radial Hamiltonian systems related to Finsler structures }\label{sec:Finslergeometry}
\setcounter{equation}{0}
Let $(M,F)$ be a smooth $n$-dimensional Finsler manifold.
The Euler theorem implies that for any $(x,y)\in TM$ with $y\neq 0$ one has $F^2(x,y)=g^F(x,y)[y,y]$. Note that $F^2$ is only $C^1$ on the zero section of $TM$. In a local coordinate centered at $x\in M$, the Finsler metric can be represented by
$$g_{ij}(x,y)=\frac{1}{2}\frac{\partial ^2(F^2)}{\partial y^i\partial  y^j}(x,y)\quad \forall\ \;y\in T_xM\setminus \{0\}.$$
For $r\in(0,\infty)$, denote the Finsler cotangent disk $D^F_rT^*M$ with radius $r$ by
\begin{equation}\label{e:Finslerdisk}
D_r^FT^*M:=\big\{(x,p)\in T^*M\big|F^*(x,p)\leq r\big\}.
\end{equation}
Let $(g^{ij})$ be the inverse matrix of the matrix $(g_{ij})$.
$\nabla^F$ denotes the \emph{Chern connection} associated to the Finsler metric on $M$, see~\cite{SZ}.
\begin{definition}
{\rm
A smooth curve $c:(a,b)\to M$ in a Finsler manifold $M$ is called a \emph{$F$-geodesic} if ${\rm len}_F$ is stationary at $c$, that is, for any variation $c_s(t)$ of $c$ it holds that
$$\frac{d}{ds}{\rm len}_F(c_s)\big|_{s=0}=0.$$
Here $c_s(\cdot)=\tau(s,t)$,  $\tau:(-\varepsilon,\varepsilon)\times [a,b]\to M$ is a smooth function such that $\tau_0(t)=c(t)$ for all $t\in(a,b)$, and $c_s(a),c_s(b)$ are constants not depending on $s\in (-\varepsilon,\varepsilon)$.
}
\end{definition}


\begin{lemma}\label{lem:mls}
Let $M$ be a closed connected Finsler manifold, and let $\alpha\in [S^1,M]$ be a free homotopy class in $M$. Then the marked length spectrum $\Lambda_\alpha$ is a closed and nowhere dense subset of $\mathbb{R}$.
\end{lemma}

 The proof of the above lemma is parallel to the Riemannian case, see, e.g. \cite[Lemma~3.3]{We0}.

The geometries $(T^*M,F^*)$ and $(TM,F)$ are related by the Legendre transforms $\ell$.
\begin{definition}\label{def:Ltransform}
{\rm The Legendre transform $\ell_x: T_xM\rightarrow T_x^*M$ is defined as
$$(\ell_x(v))_j=g_{ij}(x,v)v^i\quad \forall\  v\in T_xM\setminus\{0\}\quad \hbox{and}\quad \ell_x(0)=0$$
which has the inverse map $\ell^*_x: T_x^*M\rightarrow T_xM$ given by
$$(\ell_x^*(p))^j=g^{ij}(x,p)p_i\quad \forall\  p\in T^*_xM\setminus\{0\} \quad \hbox{and}\quad \ell^*_x(0)=0.$$

}
\end{definition}

\begin{lemma}\label{lem:radialHamsyst}
Let $h:[0,\infty)\to\mathbb{R}$ be a smooth function which is constant near $0$. Let $H$ be a $C^\infty$-radial function on the cotangent bundle $T^*M$ which is given by
$$H(x,y)=f(F^*(x,y))=h(F^{*2}(x,y)),$$
where $f(x)=h(x^2)$ is a smooth function on $[0,\infty)$.
Assume that $z(t)=(x(t),y(t))$ is a smooth loop on $T^*M$. Then $z$ is a critical point of the action functional $\mathscr{A}_{H}$ if and only if $x(t)$ is a $F$-geodesic loop with constant speed satisfying
\begin{equation}\label{e:F-geo&Hamorbit}
f'(r_z) y(t)=r_z\ell_x(\dot{x})(t),\quad
F(x(t),\dot{x}(t))\equiv\pm f'(r_z)\quad \forall\  t\in S^1
\end{equation}
for some constant $r_z=F^*(z(t))\geq 0$, and in this case the action of $\mathscr{A}_{H}$ at $z(t)$ is given by
$$\mathscr{A}_{H} (z(t))=f'(F^*(z(t)))F^*(z(t))-f(F^*(z(t))).$$
\end{lemma}
The proof is postponed to Section \ref{SSRadial}.
\section{Filtered Floer homology of cotangent bundles}\label{SFilter}
\setcounter{equation}{0}
Let $(M,F)$ be a closed connected Finsler manifold. In this section we introduce a filtration on the Floer chain complex via the action for a class of time-dependent $1$-periodic Hamiltonians defined on the cotangent bundle $T^*M$, which can be viewed as the completion of the Liouville domain $(D^FT^*M,\lambda)$ (see Example~\ref{ex:diskcotanb}).

\subsection{Radial and linear Hamiltonians at infinity}
Fix a free homology class $\alpha\in[S^1,T^*M]$. Let $a,b\in\mathbb{R}\cup\{\pm\infty\}$ with $a<b$.
For $\eta>0$, define
\begin{eqnarray}
\mathscr{H}_{\eta;\alpha}^{a,b}:=&\big\{&H\in C^\infty(S^1\times T^*M)\ \big|\ \exists\ \tau\geq0\;\exists\ c\in\mathbb{R}\;\hbox{such that}\;H_t(x,p)=\tau F^*(x,p)-c\;\hbox{for}\;\notag\\
&&F^*(x,p)\geq \eta,\;\{a,b\}\cap \hbox{Spec}(H;\alpha)=\emptyset,\;\hbox{and}\;\tau\notin\Lambda_\alpha\;
\hbox{or}\;c\notin[a,b]\big\}.\notag
\end{eqnarray}
Let us give some explanations about the above definition:
\begin{enumerate}
\item[(1)]
the condition that $a,b$ are not in the action spectrum allows for small perturbations of the Hamiltonian without changing Floer homology;
\item[(2)]  if $\tau\in\Lambda_\alpha$, by Lemma~\ref{lem:radialHamsyst}, such a Hamiltonian admits degenerate $1$-periodic orbits of action $c$ on all hypersurfaces $\{(x,p)\in T^*M|F^*(x,p)=r\}$ with $r\geq \eta$; if $\tau=0$, every point in the complement of $D_\eta^FT^*M$ is a degenerate $1$-periodic orbit with action $c$. So the condition $c\notin[a,b]$ excludes all these degenerate orbits;
\item[(3)]  if $\tau\notin\Lambda_\alpha$ and $\tau>0$, by Lemma~\ref{lem:radialHamsyst} again, there are no Hamiltonian $1$-periodic orbits on $T^*M\setminus D^F_\eta T^*M$, so
from Remark~\ref{rem:actionspectrum} one can see that $\tau\notin{\rm Spec}(D^FT^*M,\lambda)$.
\end{enumerate}

For each $H\in \mathscr{H}_{\eta;\alpha}^{a,b}$, denote
$$\mathscr{P}_\alpha^{(a,b)}(H):=\{z\in\mathscr{P}_\alpha (H)\ |\ \mathscr{A}_{H}(z)\in (a,b)\}.$$

Observe that one can perturb $H$ along the periodic orbits $z\in \mathscr{P}_\alpha^{(a,b)}(H)$ by a smooth  function $h\in C_0^\infty(S^1\times D^F_\eta T^*M)$ with sufficiently small $\|\cdot\|_{C^2}$-norm  such that $H+h\in\mathscr{H}_{\eta;\alpha}^{a,b}$ and all elements in $\mathscr{P}_\alpha^{(a,b)}(H+h)$ are nondegenerate, see, \cite[Section~9]{Sa} or \cite{We0}.

\subsection{The definition of filtered Floer homology}
Assume that $H\in \mathscr{H}_{\eta;\alpha}^{a,b}$ is nondegenerate.
Consider the $\mathbb{Z}_2$-vector space
$${\rm CF}^{(a,b)}(H,\alpha):=\bigoplus\limits_{x\in\mathscr{P}^{(a,b)}
_{\alpha}(H)}\mathbb{Z}_2x$$
which is graded by the index $\mu_{CZ}$. Here we use the convention that the complex generated by the empty set is zero.

Given $z_{\pm}\in {\rm CF}^{(a,b)}(H;\alpha)$, every solution $u$ of (\ref{e:Feq0}) connecting $z_-$ to $z_+$ satisfies
$$\mathscr{A}_{H}(z_-)-\mathscr{A}_{H}(z_+)=\int_{\mathbb{R}\times S^1}\big|\partial_su(s,t)\big|^2_{J_t}dsdt\geq 0.$$

Similar to Section~\ref{subsec:Floerhomology}, if $\mu_{CZ}(z_-)=\mu_{CZ}(z_+)+1$,  then for generic $J\in \mathcal{J}(T^*M,\lambda_0)$ the moduli space $\mathcal{M}_\alpha(z_-,z_+,H,J)$, in consideration of the fact that $\mathscr{A}_{H}(z_-)\geq\mathscr{A}_{H}(z_+)$ if it is not empty,  gives rise to the boundary operator $$\partial_*:{\rm CF}_*^{(a,b)}(H;\alpha)\rightarrow {\rm CF}_{*-1}^{(a,b)}(H;\alpha)$$
satisfying $\partial_{*-1}\partial_*=0$. So $({\rm CF}_*^{(a,b)}(H;\alpha),\partial_*)$
is a chain complex. The above claims are proved by the usual transversality~\cite{FHS} and gluing arguments~\cite{Sa}, together with the $C^0$-estimate given by Lemma ~\ref{lem:convexity}. Its homology group ${\rm HF}_*^{(a,b)}(H;\alpha)$ is called the \emph{filtered Floer homology} of $H\in \mathscr{H}_{\eta;\alpha}^{a,b}$.

\subsection{Homotopy invariance}\label{ssec:Homotopy invariance}

Suppose that Hamiltonians $H^{\pm}\in \mathscr{H}_{\eta;\alpha}^{a,b}$ are nondegenerate. Then $H^\pm=\tau_\pm F^*(x,p)-c_\pm$ on $T^*M\setminus D^F_\eta T^*M$ for some constants $\tau_+,\tau_-\geq0$ and $c_\pm\in\mathbb{R}$. We assume with loss of generality that $\tau_+\geq \tau_-$.
Let $\beta:\mathbb{R}\to [0,1]$ be a smooth cut-off function such that
$\beta=0$ for $s\leq0$, $\beta(s)=1$ for $s\geq 1$ and $0\leq\dot{\beta}(s)\leq1$.
Let $H_s=\{H_{s,t}\}$ be a smooth homotopy from $\mathbb{R}$ to $\mathscr{H}_{\eta;\alpha}^{a,b}$ defined by
$$H_{s,t}:=(1-\beta(s))H^-_t+\beta(s)H^+_t.$$

Identifying $T^*M\setminus D^F_\eta T^*M$ with $(\eta,\infty)\times D^F T^*M$ yields
$$H_{s,t}(\rho,z):=H_{s,t}(x,p)=\beta(s)(\tau_+-\tau_-)\rho
+\tau_-\rho+\beta(s)(c_--c_+)-c_-$$
for any $(\rho,z)\in (\eta,\infty)\times S^F T^*M$, where $\rho=F^*(x,p)$ and $z=(x,p/F^*(x,p))\in S^F T^*M:=\partial D^F T^*M$.

Given $z_{\pm}\in \mathscr{P}_{\alpha}(H^{\pm})$,
consider the parameter-dependent Floer equation
\begin{equation}\label{e:pFloereq}
\partial_su+J_{s,t}(u)(\partial_tu-X_{H_{s,t}}(u))=0
\end{equation}
which satisfies the asymptotic boundary conditions
\begin{equation}\label{e:asympbc}
\lim\limits_{s\to\pm\infty}u(s,t)=z_{\pm}\quad\hbox{and}\quad
\lim\limits_{s\to\pm\infty}\partial_su(s,t)=0
\end{equation}
uniformly in $t\in S^1$. Here $J_{s,t}:\mathbb{R}\to \mathcal{J}$ is a \emph{regular homotopy} of smooth families of $d\lambda_0$-compactible almost complex structure on $T^*M$, meaning that
\begin{itemize}
  \item $J_{s,t}$ is of contact type on $(\eta,\infty)\times D^F T^*M$.
  \item $J_{s,t}=J^-_t$ is regular for $H^-_t$ for $s\leq0$.
  \item $J_{s,t}=J^+_t$ is regular for $H^+_t$ for $s\geq1$.
  \item The linearized operator for equation~(\ref{e:pFloereq}) is surjective for each finite-energy solution of (\ref{e:pFloereq}) in the homotopy class $\alpha$.
\end{itemize}

By our assumption,
$$\frac{\partial^2H_{s,t}}{\partial s\partial\rho}=\beta(s)(\tau_+-\tau_-)\geq 0.$$

Combining Lemma~\ref{lem:convexity} with the exactness of the canonical symplectic form $d\lambda_0$  implies that the moduli spaces
$\mathcal{M}_\alpha(z_-,z_+,H_{s,t},J_{s,t})$ of smooth solutions of (\ref{e:pFloereq}) satisfying the boundary conditions (\ref{e:asympbc}) are $C_{loc}^\infty$-compact. Observe that every solution $u$ of (\ref{e:pFloereq}) and (\ref{e:asympbc}) satisfies the energy identity
\begin{eqnarray}\label{e:energyid}
E(u):&=&\int^\infty_{-\infty}\int_0^1|\partial_s u|_{J_{s,t}}dsdt\nonumber\\
&=&\mathscr{A}_{H^-}(z_-)-\mathscr{A}_{H^+}(z_+)
-\int^\infty_{-\infty}\int^1_0(\partial_s H_{s,t})(u(s,t))dsdt.
\end{eqnarray}
Therefore, for $|H^+-H^-|$ sufficiently small we can define a chain map
from ${\rm CF}^{(a,b)}_*(H^-;\alpha)$ to ${\rm CF}^{(a,b)}_*(H^+;\alpha)$, (see~\cite[Section~4.4]{BPS}). This defines an homomorphism
$$\sigma_{H^+H^-}:{\rm HF}^{(a,b)}(H^-;\alpha)\to {\rm HF}^{(a,b)}(H^+;\alpha)$$
which is independent of the choice of homotopy by standard arguments, see, e.g., \cite{Sa,SZ}. Actually, we have the following.

\begin{figure}[H]
  \centering
  \includegraphics[scale=0.25]{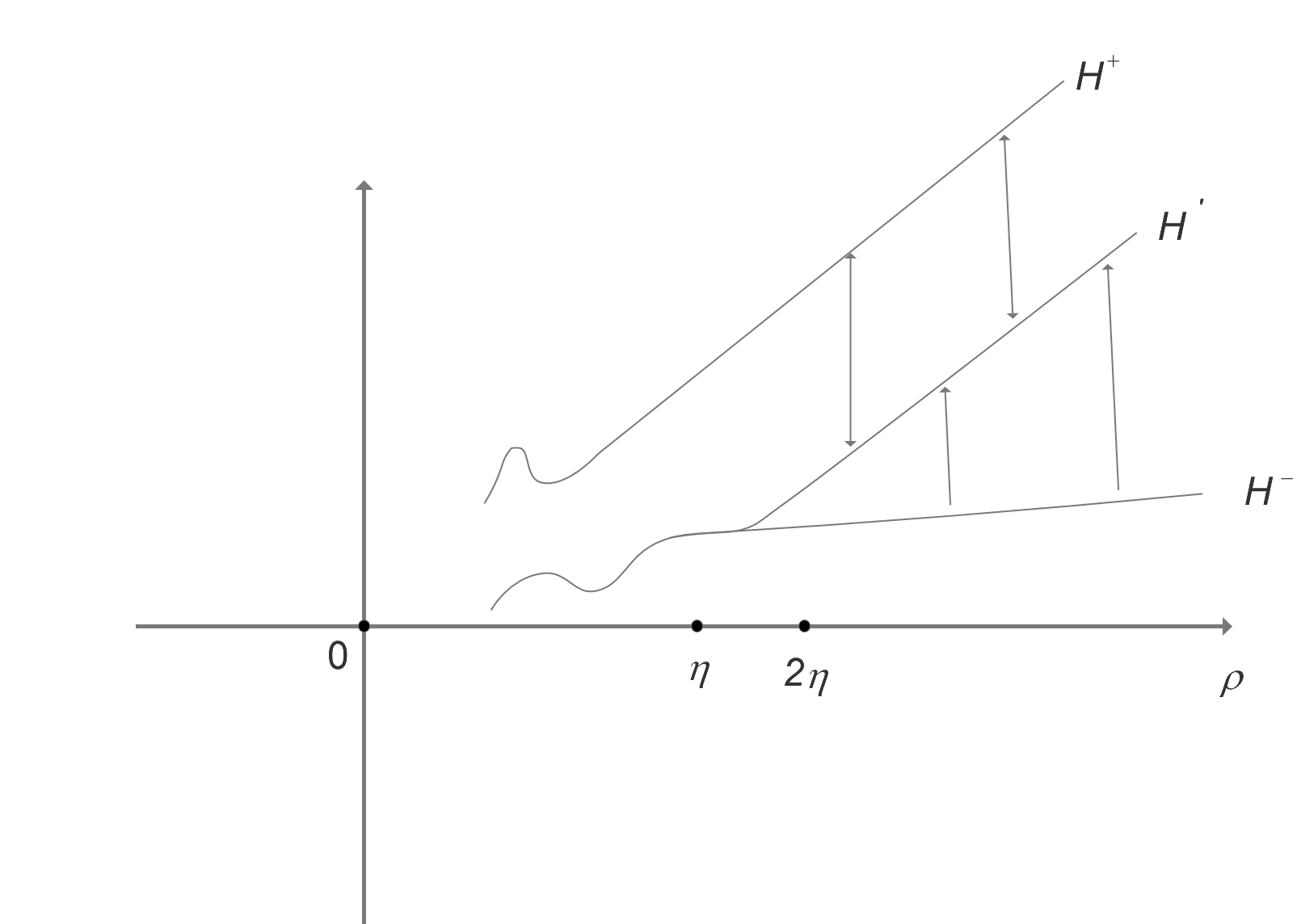}
  \caption{The functions $H^{\pm}$, $H'$ and homotopies} \label{fig:1}
\end{figure}


\begin{lemma}[Local isomorphisms]\label{lem:localiso}
If $H^\pm$ are sufficiently close to $H\in \mathscr{H}_{\eta;\alpha}^{a,b}$,
then the homomorphism $\sigma_{H^+H^-}$ is an isomorphism.
\end{lemma}
The proof of the above lemma, essentially due to Weber (see~\cite[Lemma~2.5]{We0}),  is based on the convexity result in Section~\ref{subsect:convexity} and the properties of the marked length spectrum (see~Lemma~\ref{lem:mls}).
\begin{proof}[Sketch of the proof]
As in~\cite{We0}, consider an appropriate intermediate Hamiltonian $H'$ (see Figure~\ref{fig:1}), and two homotopies  $H^0_{s,t}:=(1-\beta(s))H^-_t+\beta(s)H'_t$ and $H^1_{s,t}:=(1-\beta(s))H'_t+\beta(s)H^+_t$. Here the $C^\infty$-Hamiltonian $H'$ is required to satisfy:
\begin{itemize}
	\item  $H'=H^-$ on $D^F_\eta T^*M$;
	\item $H'=\tau_+F^*-c_0$ for some constant $c_0\in \mathbb{R}$ outside $D^F_{2\eta} T^*M$;
	\item $\frac{\partial H'}{\partial\rho}\geq 0$ for $\rho\in (\eta,2\eta)$. 
\end{itemize}
  Then The standard arguments (see~\cite[Section~3.4]{Sa}) implies that the homomorphisms $\sigma_{H^+H^-}$ and $\sigma_{H^+H'}\circ\sigma_{H'H^-}$ are equal. Hence it suffices to prove that both $\sigma_{H^+H'}$ and $\sigma_{H'H^-}$ are isomorphisms whenever $H^\pm$ are sufficiently close to $H\in \mathscr{H}_{\eta;\alpha}^{a,b}$.

To prove this, let us first note that since $\mathbb{R}\setminus\Lambda_\alpha$ is open by Lemma~\ref{lem:mls},  $(\tau_-,\tau_+)\cap \Lambda_\alpha=\emptyset$ provided that $H^+$ and $H^-$ are  sufficiently close to $H\in \mathscr{H}_{\eta;\alpha}^{a,b}$. From Lemma~\ref{lem:radialHamsyst} we know that no $1$-periodic Hamiltonian orbits of $H^+$ and $H^-$ appear outside $D^F_\eta T^*M$. Then, by lemma~\ref{lem:convexity}, the solution $u$ of Floer equation~(\ref{e:pFloereq}) connecting $z^-\in \mathscr{P}_{\alpha}(H^-)$ to
$z\in \mathscr{P}_{\alpha}(H')$ can not escape from $D^F_\eta T^*M$.
Moreover, we have $H^0_s=H^-$ on $D^F_\eta T^*M$ for all $s\in\mathbb{R}$. Therefore, the homomorphism $\sigma_{H'H^-}$ is the identity map.

Secondly, observe that outside $D^F_{2\eta} T^*M$ it holds that
$\frac{\partial^2  H_{s,t}}{\partial s\partial\rho}=0.$ By lemma~\ref{lem:convexity} again all solutions of Floer equation~(\ref{e:pFloereq}) corresponding to the homotopy $H^1_{s,t}$ remains in $D^F_{2\eta} T^*M$, and this also holds for its inverse homotopy $G_{s,t}:=(1-\beta(s))H^+_t+\beta(s)H'_t.$ So the usual argument of reversing the homotopy achieves $\sigma_{H^+H'}=\sigma_{H'H^+}^{-1}$, and hence $\sigma_{H^+H'}$ is an isomorphism. This completes the proof of Lemma~\ref{lem:localiso}.
\end{proof}

A direct consequence of the above lemma is that one can still define ${\rm HF}^{(a,b)}(H^-;\alpha)$ for every $H\in \mathscr{H}_{\eta;\alpha}^{a,b}$ when $H$ is degenerate, by simply setting
$${\rm HF}^{(a,b)}(H;\alpha):={\rm HF}^{(a,b)}(\widetilde{H};\alpha)$$
for any nondenerate Hamiltonian $\widetilde{H}\in \mathscr{H}_{\eta;\alpha}^{a,b}$ which is sufficiently close to $H$.
Moreover, by composing the above local isomorphisms we deduce that in each component of $\mathscr{H}_{\eta;\alpha}^{a,b}$ all the functions have identical Floer homology groups.

\subsection{Monotone homotopies}\label{subsec:Monotone homotopies}

Let $H,K\in \mathscr{H}_{\eta;\alpha}^{a,b}$ be two functions with $H(t,z)\leq K(t,z)$ for all $(t,z)\in S^1\times T^*M$. Choose a smooth homotopy $s\to H_{s}=\{H_{s,t}\}$ from $H$ to $K$ such that $\partial_s H_{s}\geq0$ everywhere and $\frac{\partial^2H_{s,t}}{\partial s\partial\rho}\geq 0$ for every $\rho\geq \eta$. Here we do not require $H_{s,t}$ to be in $\mathscr{H}_{\eta;\alpha}^{a,b}$ for every $s\in [0,1]$. From the energy identity (\ref{e:energyid}) we deduce that such a homotopy induces a natural homomorphism, which is called \emph{monotone homomorphism}
\begin{equation}\label{e:MH}
\sigma_{KH}:{\rm HF}^{(a,b)}(H;\alpha)\to {\rm HF}^{(a,b)}(K;\alpha).
\end{equation}
(see, e.g., \cite{BPS,FH,CFH,SZ,Vi}):

\begin{lemma}\label{lem:mh}
These monotone homomorphisms are independent of the choice of the monotone homotopy of Hamiltonians and satisfy the following properties
\begin{equation}
\begin{aligned}
\sigma_{HH}&={\rm id}\quad\forall\ \;H\in \mathscr{H}_{\eta;\alpha}^{a,b},\\
\sigma_{KH}\circ\sigma_{HG}&=\sigma_{KG}
\end{aligned}
\end{equation}
whenever $ G,H,K\in \mathscr{H}_{\eta;\alpha}^{a,b}$ satisfy $G\leq H\leq K$.
\end{lemma}
As a corollary we have
\begin{lemma}[{\rm see~\cite{Vi} or \cite[Section~4.5]{BPS}}]\label{lem:misom}
If $K_s$ is a monotone homotopy from $H$ to $K$ such that
$K_s\in \mathscr{H}_{\eta;\alpha}^{a,b}$ for every $s\in[0,1]$, then $\sigma_{KH}$ is an isomorphism.
\end{lemma}

\subsection{Symplectic homology}
Following closely the diskussion in~\cite[Sections~4.6 and ~4.7]{BPS},
we consider the direct and inverse limits of Floer homology groups, see also~\cite[Section~3.1]{We0}.
\subsubsection{Partially ordered set $(\mathscr{H}_\alpha^{a,b},\preceq)$}
Fix a free homotopy class $\alpha\in[S^1,M]$. For $-\infty\leq a<b\leq +\infty$, we define
$$\mathscr{H}_\alpha^{a,b}:=\big\{H\in C^\infty_0(S^1\times D^FT^*M)\ \big|\ a,b\notin \hbox{Spec}(H;\alpha)
\big\}$$
of all Hamiltonians which are compactly supported in $D^FT^*M$ and do not contain $a$ and $b$ in their action spectrum. We introduce the \emph{partial-order relation} $\preceq$ on the set $\mathscr{H}_\alpha^{a,b}$ by
$$H_0\preceq H_1\quad \Longleftrightarrow \quad H_0(t,z)\leq H_1(t,z)\quad \hbox{for all}\; (t,z)\in S^1\times D^FT^*M.$$
\subsubsection{Inverse limits}
Note that $\alpha\neq 0$ implies $\mathscr{H}_\alpha^{a,b}\subseteq \mathscr{H}_{1;\alpha}^{a,b}$. When $\alpha$ is the homotopy class of constant loops, we ask the intervals $[a,b]$ to not contain $0$. In this case, we still have $\mathscr{H}_\alpha^{a,b}\subseteq \mathscr{H}_{1;\alpha}^{a,b}$.
So the monotone homomorphisms $\sigma_{H_1H_0}$ in Section~\ref{subsec:Monotone homotopies} yields the partially ordered system $(\hbox{HF},\sigma)$ of $\mathbb{Z}_2$-vector spaces over $\mathscr{H}_\alpha^{a,b}$, that is, $\hbox{HF}$ assigns to every $H\in \mathscr{H}_\alpha^{a,b}$ the $\mathbb{Z}_2$-vector space $\hbox{HF}^{(a,b)}_*(H;\alpha)$, and $\sigma$ assigns to any two elements $H_0,H_1\in \hbox{HF}^{(a,b)}_*(H;\alpha)$ with $H_0\preceq H_1$ the monotone homomorphism $\sigma_{H_1H_0}$ satisfying (\ref{lem:mh}). The partial order system $(\mathscr{H}_\alpha^{a,b},\preceq)$
is \emph{downward directed}, meaning that for any $H_1,H_2\in \mathscr{H}_\alpha^{a,b}$, there exists $H_0\in \mathscr{H}_\alpha^{a,b}$ such that $H_0\preceq H_1$ and $H_0\preceq H_2$. The functor $(\hbox{HF},\sigma)$
is an \emph{inverse system} of $\mathbb{Z}_2$ vector spaces over $\mathscr{H}_\alpha^{a,b}$.  We call its inverse limit the \emph{symplectic homology} of $D^FT^*M$ in the free homotopy class $\alpha$ with action interval $(a,b)$ and denote it by
\begin{equation}
\begin{aligned}
&\underleftarrow{\hbox{SH}}
^{(a,b)}_*(D^FT^*M;\alpha):=\lim
\limits_{\substack{\longleftarrow\\H\in \mathscr{H}_\alpha^{a,b}}}\hbox{HF}^{(a,b)}_*(H;\alpha)\notag\\
:&=\bigg\{\{e_H\}_{H\in \mathscr{H}_\alpha^{a,b}}\in \prod\limits_{H\in \mathscr{H}_\alpha^{a,b}}\hbox{HF}^{(a,b)}_*(H;\alpha)
\bigg|H_1\preceq H_2\Rightarrow \sigma_{H_2H_1}(e_{H_1})=e_{H_2}
\bigg\}.\notag
\end{aligned}
\end{equation}

For $H\in \mathscr{H}_\alpha^{a,b}$, one can define the natural projection
$$\pi_H:\underleftarrow{\hbox{SH}}
^{(a,b)}_*(D^FT^*M;\alpha)\rightarrow \hbox{HF}^{(a,b)}_*(H;\alpha)$$
which satisfies $\pi_{H_1}=\sigma_{H_1H_0}\circ\pi_{H_0}$ whenever $H_0\preceq H_1$.

\subsubsection{Direct limits}
Fix $c>0$. Next we introduce \emph{relative symplectic homology} of the pair $(D^FT^*M,M)$ at the level $c$ in the homotopy class $\alpha$ for the action interval $(a,b)$. Consider the subset
$$\mathscr{H}_\alpha^{a,b;c}:=\big\{H\in\mathscr{H}_\alpha^{a,b}\big|
\sup\limits_{S^1\times M}H< -c\big\}.$$
This set is \emph{upward directed}. Namely, for any $H_0,H_1\in \mathscr{H}_\alpha^{a,b;c}$, there exists $H_2\in \mathscr{H}_\alpha^{a,b;c}$ such that $H_0\preceq H_2$ and $H_1\preceq H_2$. The functor
 $(\hbox{HF},\sigma)$
can be viewed as an \emph{direct system} of $\mathbb{Z}_2$-vector spaces over $\mathscr{H}_\alpha^{a,b;c}$, whose \emph{direct limit} is defined as
\begin{eqnarray}
\underrightarrow{\hbox{SH}}
^{(a,b);c}_*(D^FT^*M,M;\alpha)&:=&\lim
\limits_{\substack{\longrightarrow\\H\in \mathscr{H}_\alpha^{a,b;c}}}\hbox{HF}^{(a,b)}_*(H;\alpha)\notag\\
&:=&\big\{(H,e_H)
\big|H\in \mathscr{H}_\alpha^{a,b;c},\; e_H\in \hbox{HF}^{(a,b)}_*(H;\alpha)
\big\}/\sim.\notag
\end{eqnarray}
Here the equivalence relation is defined as follows: $(H_0,e_{H_0})\sim (H_1,e_{H_1})$ if and only if there exists $H_2\in \mathscr{H}_\alpha^{a,b;c}$
such that $H_0\preceq H_2$, $H_1\preceq H_2$ and $\sigma_{H_2H_0}(e_{H_0})=\sigma_{H_2H_1}(e_{H_1})$. The direct limit is
a $\mathbb{Z}_2$-vector space with the operations
$$k[H_0,e_{H_0}]:=[H_0,ke_{H_0}], \quad [H_0,e_{H_0}]+[H_1,e_{H_1}]:=[H_2,\sigma_{H_2H_0}(e_{H_0})
+\sigma_{H_2H_1}(e_{H_1})],$$
where $k\in \mathbb{Z}_2$, and $H_2\in \mathscr{H}_\alpha^{a,b;c}$ with
$H_0\preceq H_2$ and $H_1\preceq H_2$. For $H\in \mathscr{H}_\alpha^{a,b;c}$, let
$$\iota_H:\hbox{HF}^{(a,b)}_*(H;\alpha)\longrightarrow \underrightarrow{\hbox{SH}}
^{(a,b);c}_*(D^FT^*M,M;\alpha),\quad e_H\mapsto [H,e_H]$$
be the natural homomorphism. Clearly, we have $\iota_{H_0}=\iota_{H_1}\circ\sigma_{H_1H_0}$.

\subsubsection{Exhausting sequences}
A sequence $\{H_i\in \mathscr{H}_\alpha^{a,b}|i\in \mathbb{N}\}$ is called \emph{downward exhausting} for $(\hbox{HF},\sigma)$ if it holds the following two properties
\begin{itemize}
  \item for every $i\in \mathbb{N}$ we have $H_{i+1}\preceq H_i$ and $\sigma_{H_iH_{i+1}}:\hbox{HF}^{(a,b)}_*(H_{i+1};\alpha)\to \hbox{HF}^{(a,b)}_*(H_i;\alpha)$ is an isomorphism
  \item for every $H\in \mathscr{H}_\alpha^{a,b}$ there exists a $i\in \mathbb{N}$ such that $H_i\preceq H$.
\end{itemize}

Correspondingly, a sequence $\{H_i\in \mathscr{H}_\alpha^{a,b;c}|i\in \mathbb{N}\}$ is called \emph{upward exhausting} for $(\hbox{HF},\sigma)$ if and only if it holds that
\begin{itemize}
  \item for every $i\in \mathbb{N}$ we have $H_i\preceq H_{i+1}$ and $\sigma_{H_{i+1}H_i}:\hbox{HF}^{(a,b)}_*(H_i;\alpha)\to \hbox{HF}^{(a,b)}_*(H_{i+1};\alpha)$ is an isomorphism
  \item for every $H\in \mathscr{H}_\alpha^{a,b;c}$ there exists a $i\in \mathbb{N}$ such that $H\preceq H_i$.
\end{itemize}

\begin{lemma}\label{lem:inv/dirlimit}
Let $({\rm HF},\sigma)$ be the partially ordered system of $\mathbb{Z}_2$-vector spaces over $(\mathscr{H}_\alpha^{a,b},\preceq)$.
\begin{itemize}
  \item [(i)] If $\{H_i\in \mathscr{H}_\alpha^{a,b}|i\in \mathbb{N}\}$ is a downward exhausting sequence for $({\rm HF},\sigma)$, then the homomorphism $\pi_{H_i}:\underleftarrow{{\rm SH}}
^{(a,b)}_*(D^FT^*M;\alpha)\rightarrow {\rm HF}^{(a,b)}_*(H_i;\alpha)$ is an isomorphism for every $i\in \mathbb{N}$.
  \item [(ii)] If $\{H_i\in \mathscr{H}_\alpha^{a,b;c}|i\in \mathbb{N}\}$ is a upward exhausting sequence for $({\rm HF},\sigma)$, then the homomorphism $\iota_{H_i}:\underrightarrow{{\rm SH}}
^{(a,b);c}_*(D^FT^*M,M;\alpha)\rightarrow {\rm HF}^{(a,b)}_*(H_i;\alpha)$ is an isomorphism for every $i\in \mathbb{N}$.
\end{itemize}
\end{lemma}

Before the end of this section, let us state the following proposition, which is an adaptation of {\cite[Proposition~4.8.2]{BPS}}.

\begin{proposition}[{\cite[Proposition~4.8.2]{BPS}}]\label{prop:T-hom}
Let $\alpha\in[S^1,M]$ be a free homotopy class and suppose that $-\infty\leq a<b\leq +\infty$. Then for any $c\in\mathbb{R}$ there exists a unique homomorphism
$$T^{(a,b);c}_\alpha:\underleftarrow{{\rm SH}}
^{(a,b)}_*(D^FT^*M;\alpha)\longrightarrow \underrightarrow{{\rm SH}}
^{(a,b);c}_*(D^FT^*M,M;\alpha)$$
such that for every $H\in \mathscr{H}_\alpha^{a,b;c}$, the following diagram commutes:
\begin{equation}\label{e:factorization}
\xymatrix{\underleftarrow{{\rm SH}}
^{(a,b)}_*(D^FT^*M;\alpha)
\ar[rr]^{T^{(a,b);c}_\alpha}\ar[dr]_{\pi_H}& & \underrightarrow{{\rm SH}}
^{(a,b);c}_*(D^FT^*M,M;\alpha)\\ & {\rm HF}^{(a,b)}_{\alpha}(H;\alpha)\ar[ur]_{\iota_H} & }
\end{equation}

\end{proposition}

\section{The isomorphism between the Morse and the Floer complex}\label{SAS}
\subsection{The  isomorphism }
In order to compute the groups $\underleftarrow{{\rm SH}}
^{(a,b)}_*(D^FT^*M;\alpha)$ and $\underrightarrow{{\rm SH}}
^{(a,b);c}_*(D^FT^*M,M;\alpha)$, we use an isomorphism between the Morse and the Floer complex following \cite{AS1}.

Let $M$ be a smooth compact connected $n$-dimensional manifold without boundary.
Recall that a smooth Lagrangian $$L:S^1\times TM\to\mathbb{R}$$ is \emph{fiberwise convex and quadratic at infinity} if it satisfies:
\begin{enumerate}
\item[(\textbf{L1})] there exists $l_1>0$ such that
$$\partial_{vv}L(t,x,v)\geqslant l_1 \hbox{I}\quad \forall\ \; (t,x,v)\in S^1\times TM;$$
\item[(\textbf{L2})] there exists $l_2>0$ such that
\begin{gather}
\|\partial_{vv}L(t,x,v)\|\leqslant l_2,\quad \|\partial_{xv}L(t,x,v)\|\leqslant l_2(1+|v|), \notag\\ \|\partial_{xx}L(t,x,v)\|\leqslant l_2(1+|v|^2)\quad \forall\ \;(t,x,v)\in S^1\times TM\notag\\
\end{gather}
with respect to some Riemannian metric $g$ on $M$ with $|v|_x^2:=g_x(v,v)$.
\notag
\end{enumerate}
Equivalently, there exists a finite atlas on $M$ and two constants $0<c_0<c_1$ such that in every chart of this atlas the following conditions hold:
\begin{itemize}
  \item [(\textbf{L1})] $\sum_{i,j}\frac{\partial^2 }{\partial v_i\partial v_j}L(t,x,v)u_iu_j\geq c_0|u|^2$ \quad $\forall\  t\in S^1,\;\forall\  u=(u_1,u_2,\ldots,u_n)\in \mathbb{R}^n$.
  \item [(\textbf{L2})]$
\big|\frac{\partial^2 }{\partial v_i\partial v_j}L(t,x,v)\big|\leq c_1,\quad \big|\frac{\partial^2 }{\partial x_i\partial v_j}L(t,x,v)\big|\leq c_1(1+|v|)$ \quad and \\
$\big|\frac{\partial^2 }{\partial x_i\partial x_j}L(t,x,v)\big|\leq c_1(1+|v|^2),\quad \forall\ \;(t,x,v)\in S^1\times TM.
$
\end{itemize}

Denote by $\mathcal{L}(M)$ the space of absolutely continuous curves $\gamma:S^1\rightarrow M$ such that
$$\int_{S^1}g(\dot{\gamma}(t),\dot{\gamma}(t))dt<\infty,$$
and denote by $\mathcal{L}_\alpha(M)$ the connected component whose elements represent the free homotopy class $\alpha$. It is well known that $\mathcal{L}_\alpha(M)$ has Hilbert manifold structures and that its tangent space at $\gamma$ $T_\gamma \mathcal{L}_\alpha(M)$ can be identified with the Hilbert space $W^{1,2}(\gamma^*TM)$ equipped with the inner product
$$\langle\xi,\eta\rangle_{W^{1,2}}:=\int_0^1g(\xi(t),\eta(t))dt
+\int_0^1g(\nabla^g \xi(t),\nabla^g \eta(t))dt.$$
Here $W^{1,2}(\gamma^*TM)$ denotes the space of all $1$-periodic vector fields along $\gamma$ with finite $W^{1,2}$-norm.

\begin{remark}
Since $\mathcal{L}_\alpha(M)$ is homotopy equivalent to $\Lambda_\alpha(M)$, if no ambiguity is possible,
we sometimes just work on $\Lambda_\alpha(M)$ instead of the complete space $\mathcal{L}_\alpha(M)$ (see, e.g., using Lemma~\ref{lem:deformlemma} to show Theorem~\ref{thm:convexradial}).
\end{remark}

The \emph{Lagrangian action functional} $\mathscr{L}:\mathcal{L}_\alpha(M)\to \mathbb{R}$ is defined by
$$\mathscr{L}(x):=\int^1_0L(t,x(t),\dot{x}(t))dt\quad \forall\ \;x\in\mathcal{L}_\alpha(M).$$
Here we need to point out that in general the functional $\mathscr{L}$ is not of class $C^2$. 
To do Morse theory one may need at least $C^2$-regularity since the Morse lemma requires $C^2$-regularity (see~\cite{Cha}). The lack of $C^2$-regularity of $\mathscr{L}$ is the main difficulty that  prevents us from doing infinite dimensional Morse theory for $\mathscr{L}$. We introduce the following non-degeneracy hypothesis (see (\textbf{L0}) below).
\begin{itemize}
  \item [(\textbf{L0})] Every critical point $x$ of $\mathscr{L}$ is non-degenerate, that is, the symmetric bilinear form $d^2\mathscr{L}(x)$ on $T_x\mathcal{L}_\alpha(M)$ is non-degenerate.
\end{itemize}

We next cite the following result from~\cite{AS1}.
\begin{theorem}[\cite{AS1}]\label{Thm:AS-isomorphism}
Let $M$ be a closed manifold. Assume that $H\in C^\infty(S^1\times T^*M,\mathbb{R})$ is the Legendre transform of the Lagrangian $L\in C^\infty(S^1\times TM,\mathbb{R})$ satisfying (\textbf{L0}) -- (\textbf{L2}).
Then there exists a chain isomorphism $\Psi$ from the Floer complex of $H$ to the Morse complex of the Lagrangian action functional $\mathscr{L}$ associated to the Lagrangian $L$ with coefficients in $\mathbb{Z}_2$. Moreover, such an isomorphism preserves the action filtrations, that is, $\Psi$ induces an isomorphism $\Psi_*$ between ${\rm HF}^{<a}_*(H;\alpha)$ and ${\rm HM}^{<a}_*(\mathscr{L};\alpha)$ for every $a\in\mathbb{R}\cup \{\pm \infty\}$.
\end{theorem}

\subsection{Convex quadratic modifications}\label{convexq}
\setcounter{equation}{0}
The main problem of applying Theorem \ref{Thm:AS-isomorphism} to our setting is that $F^2$ is not smooth at the zero section. To solve this problem, we introduce the following quadratic modification which smoothens $F^2$ near the zero section maintaining the fiberwise convexity.
\begin{lemma}\label{lem:modification}
 Let $(M,F)$ be a closed connected Finsler manifold and $L_0=F^2:\ TM\to \R$. Then there exists a family of \emph{convex quadratic $\eta$-modifications} $L_\eta$ of $L_0$ which satisfies $(\mathbf{L1}), (\mathbf{L2})$ and the following on $TM$:
\begin{itemize}
  \item[(a)] $L_0(x,v)-\eta\leq L_\eta(x,v)\leq L_0(x,v)$ for all $(x,v)\in TM$,
  \item[(b)] $L_\eta(x,v)=L_0(x,v)$ for $L_0(x,v)\geq\eta$.
\end{itemize}
\end{lemma}

\begin{remark}
{\rm
Actually, we can construct a family of convex quadratic Lagrangians $\{L_{\eta}\}_{0<\eta<\eta_0}$ which satisfy (a) and (b), and in addition satisfy $L_{\eta_1}\geq L_{\eta_2}$ for $0<\eta_1<\eta_2$.
This can be done by picking a Lagrangian $L_{\eta_0}$ as in Lemma \ref{lem:modification}, and then
rescaling the parameters in the auxiliary functions $\lambda^\mu_{\epsilon,\delta}$  and $\chi^\kappa_{\delta,\rho}$ (see Section~\ref{sec:Pfmodification}) by $\Delta\in(0,1]$, for instance, asking $\kappa\to\kappa/\Delta$, $\mu\to \mu/\Delta$, $\delta\to\Delta\delta$, $\epsilon\to\Delta\epsilon$ and $\eta_0\to\Delta\eta_0$ without changing $\sigma$ and $\rho$.
}
\end{remark}

The strong convexity condition (\textbf{L1}) can be used to define a Hamiltonian on $T^*M$ by means of the Legendre transform
\begin{equation}\label{e:L-transform}
\mathfrak{L}:S^1\times TM\rightarrow S^1\times T^*M,\quad (t,x,v)\to\big(t,x,d_vL(t,x,v)\big),
\end{equation}
which is is a fiber-preserving diffeomorphism, cf.~\cite{Ma}.
The Hamiltonian $H:S^1\times T^*M\to\mathbb{R}$ of the form
$$
H(t,x,p):=\max\limits_{v\in T_xM}(p(v)-L(t,x,v))
$$
is called the \emph{Fenchel dual Hamiltonian} of $L$.

For each $L_\eta$ in Lemma \ref{lem:modification}, its Legendre transform $H_\eta$ coincides with $\frac{1}{2}F^{*2}$ on $T^*M\setminus D_{\sqrt{\eta}}T^*M$. Since $L_\eta\leq L_0$, by definition it holds that $H_\eta\geq \frac{1}{2}F^{*2}$ on $T^*M$. Moreover, since we have
$\partial_{pp}H_\eta(x,p)\partial_{vv}L_\eta(x,v)=I,$
we get that $H_\eta$ is also fiberwise uniformly convex since $L_\eta$ satisfies (\textbf{L1}) by our construction.

For $V\in C^\infty_0(S^1\times TM,\mathbb{R})$,
we define the Lagrangian action functional $\mathscr{L}_V^\eta$ on $\Lambda_\alpha M$ by
$$\mathscr{L}^\eta_V(x):=\int^1_0 \frac{1}{2}L_\eta(x(t),\dot{x}(t))-V(t,x(t),\dot{x}(t))dt.$$

Now following directly from Theorem \ref{Thm:AS-isomorphism} we have the corollary.

\begin{corollary}\label{coro:AS-isomorphism}
	Let $(M,F)$ be a connected closed Finsler manifold, and let $L_\eta$ ($0<\eta\leq\eta_0$) be a  Lagrangian constructed in Section~\ref{convexq}. Fix $r\in (0,\infty)\setminus\Lambda_\alpha$ with $r^2>\eta_0$. Suppose that $L_\eta/2+V$ for $V\in C_0^\infty(S^1\times D^F_rTM)$ satisfies (\textbf{L1}) and (\textbf{L2}), and that every critical point $x$ of $\mathscr{L}^\eta_V$ with $\mathscr{L}^\eta_V(x)\leq r^2/2$ is non-degenerate. Let $H_\eta+U$ be its Fenchel dual Hamiltonian with $U\in C_0^\infty(S^1\times D^F_rT^*M)$.
	\begin{enumerate}
\item[(1)] 	Then there is an isomorphism
	\begin{equation}\label{eq:AS-isomorphism}
	\Psi_*:{\rm HF}^{<r^2/2}_*(H_{\eta}+U;\alpha) \longrightarrow {\rm H}_*\big(\big\{x\in\Lambda_\alpha M\big|\mathscr{L}^{\eta}_{V}(x)<r^2/2\big\}\big)
	\end{equation}
	which preserves the action filtrations, i.e., for every $a\in (0,r^2/2]$, $\Psi_*$ is an isomorphism between ${\rm HF}^{<a}_*(H_{\eta}+U;\alpha)$ and ${\rm H}_*\big(\big\{x\in\Lambda_\alpha M\big|\mathscr{L}^{\eta}_{V}(x)<a\big\}$.
		\item[(2)]
	 For every $a\in (0,r^2/2]$ and $0<\eta_1<\eta_2<\eta_0$, the following diagram commutes:
\begin{eqnarray}
\begin{CD}\label{CD:diag3}
{\rm HF}^{(-\infty,a)}_*(H_{\eta_1}+U_1;\alpha) @>>>{\rm HF}^{(-\infty,a)}_*(H_{\eta_2}+U_2;\alpha) \\
@V{{\cong}}VV  @VV{{\cong}}V \\
{\rm H}_*\big(\{x\in\Lambda_\alpha M\big|\mathscr{L}^{\eta_1}_{V_1}(x)< a\}\big) @>>>
{\rm H}_*\big(\{x\in\Lambda_\alpha M\big|\mathscr{L}^{\eta_2}_{V_2}(x)< a\}\big)
\end{CD}
\end{eqnarray}
Here $U_1,U_2\in C_0^\infty(S^1\times T^*M,\mathbb{R})$ are chosen so that $H_{\eta_1}+U_1$ and $H_{\eta_2}+U_2$ satisfy $H_{\eta_2}+U_2\leq H_{\eta_1}+U_1$ (and hence $\mathscr{L}^{\eta_1}_{V_1}\geq \mathscr{L}^{\eta_2}_{V_2}$) and the non-degeneracy condition
\begin{enumerate}
  \item[(\textbf{H0})] All the elements $z\in \mathscr{P}_{\alpha}^{r^2/2}(H)$ with $H=H_\eta+U_i$, $i=1,2$ are \emph{non-degenerate}, that is to say, the linear map $d\phi_H^1(z(0))\in \mathrm{Sp}(T_{z(0)}T^*M)$ does not have $1$ as an eigenvalue.
\end{enumerate}
	\end{enumerate}
\end{corollary}

\section{The proof of Theorem~\ref{thm:convexradial}}\label{sec:FHconvexradial}
In this section we compute Floer homology of a class of convex radial Hamiltonians and prove Theorem ~\ref{thm:convexradial}. Let $(M,F)$ be a smooth connected closed Finsler manifold.
Let $\{L_\eta\}_{0<\eta<\eta_0}$ be a family of Lagrangians constructed in Section~\ref{convexq}, and $H_\eta$ the Fenchel dual Hamiltonians of $L_\eta/2$. Denote $Q_\eta:=\sqrt{2H_\eta}$. Obviously, it holds that
\begin{itemize}
  \item $Q_\eta=F^*$ on $T^*M\setminus D_{\sqrt{\eta}}T^*M$,
  \item $F^*\leq Q_\eta\leq \sqrt{{F^*}^2+\eta}$ on $T^*M$\quad and
  \item $Q_{\eta_2}\geq Q_{\eta_1}$ for $0<\eta_1\leq\eta_2<\eta_0$.
\end{itemize}
For $c\in \mathbb{R}$, denote ${\rm H}_*(\Lambda_\alpha^{\eta;c} M)$ the singular homology with $\mathbb{Z}_2$-coefficients of the sublevel set
$$\Lambda_\alpha^{\eta;c} M:=\big\{x\in\Lambda_\alpha M\big|\mathscr{L}^\eta_0(x)\leq c \big\}.$$
Hereafter, we abbreviate $\Lambda_\alpha^{0;c} M$ by $\Lambda_\alpha^c M$, and for $a\leq b$ denote by $I^\eta_{a,b}$ the natural inclusion from  $\Lambda_\alpha^{\eta,a}M$ to $\Lambda_\alpha^{\eta,b}M$. For the sake of simplicity, we also denote $[I]:{\rm H}_*(\Lambda_\alpha^{\eta;a} M)\to{\rm H}_*(\Lambda_\alpha^{\eta;b} M)$ the homomorphism induced by $I^\eta_{a,b}$.

In the proof of Theorem~\ref{thm:convexradial}, we use the following two lemmata.

\begin{lemma}(~\cite[Proposition~3.1]{CiD})\label{lem:deformlemma}
Let $\{f_t\}_{t\in [0,1]}$ be a family of $C^1$-functions from a Banach space $X$ to $\mathbb{R}$. Let $a\in\mathbb{R}$ and $\varepsilon>0$. For any $t\in[0,1]$ we denote 
\begin{equation}\notag
\begin{array}{ll}
\Sigma_t:=\{u\in X|f_t(u)\leq a\}.
\end{array}
\end{equation}
Suppose that
\begin{itemize}
  \item[(a)] $\inf\limits_u\{\|f_t'(u)\|_{X^*}:\;a-\varepsilon\leq f_t(u)\leq a\}>0,\quad\forall t\in [0,1];$
  \item[(b)] if $t_k\to t$ in $[0,1]$, then $f_{t_k}\to f_t$ uniformly on $\bigcup_{0\leq \tau\leq 1}\Sigma_\tau$.
\end{itemize}
Then $H_*(\Sigma_0)\cong H_*(\Sigma_1)$.

\end{lemma}

\begin{lemma}\label{lem:loophomology}
Let $\{L_\eta\}_{0<\eta<\eta_0}$ be a family of Lagrangians constructed in Section~\ref{convexq}. Suppose that $\alpha$ is a nontrivial homotopy class of free loops in $M$ and $a\in(0,+\infty)\setminus \Lambda_\alpha$. Then there exists a natural isomorphism
$$\lim
\limits_{\substack{\longleftarrow\\ \eta\in (0,\eta_0]}}
{\rm H}_*\big(\big\{x\in\Lambda_\alpha M\ \big|\ \mathscr{L}^{\eta}_0(x)\leq a^2/2\big\}\big)\longrightarrow {\rm H}_*(\Lambda_\alpha^{a^2/2} M).$$
\end{lemma}

\begin{remark}\label{rem0:deformlemma}
{\rm
Actually, in the proof of Lemma~\ref{lem:loophomology} we have shown that there exists a sufficiently small constant $\bar{\eta}=\bar{\eta}(a)>0$ such that for any $0\leq\eta\leq\bar{\eta}$, the natural inclusion
$$ \Lambda_\alpha^{a^2/2} M\hookrightarrow \{x\in\Lambda_\alpha M\ |\ \mathscr{L}^{\eta}_0(x)\leq a^2/2\big\}$$
induces an isomorphism between the $\mathbb{Z}_2$-singular homologies ${\rm H}_*(\{x\in\Lambda_\alpha M\ |\ \mathscr{L}^{\eta}_0(x)\leq a^2/2\})$ and ${\rm H}_*(\Lambda_\alpha^{a^2/2} M)$. Besides,
one can also see that for $\eta$ sufficiently small, the Lagrangian functional $\mathscr{L}^{\eta}_0$ has no critical points $x\in \Lambda_\alpha M$ such that $\dot{x}(\tau)\in D^F_{\sqrt{\eta}}TM$ for some $\tau\in S^1$. Hence, by the property of Legendre transform, we find that whenever $\eta$ is sufficiently small, say, not larger than $\bar{\eta}>0$, the functional $\mathscr{A}_{H_\eta}$ does not admit critical points on $\Lambda_\alpha T^*M$ intersecting $D^F_{\sqrt{\eta}}T^*M$.
}
\end{remark}

We now complete the proof of Theorem~\ref{thm:convexradial} assuming Lemma~\ref{lem:deformlemma} Lemma~\ref{lem:loophomology}.
\begin{proof}[The proof of Theorem~\ref{thm:convexradial}]
We introduce a smooth function $f^{(\lambda)}$ with slope $\lambda$ at infinity associated to $f$ as following: firstly follow the graph of $f$ until it takes on slope $\lambda$ for the first time, say, at a point $r=r(f,\lambda)$, then continue linearly with slope $\lambda$, finally smoothing near $r$ yields a $C^\infty$-function $f^{(\lambda)}$ which coincides with $f$ outside a small neighborhood of $r$. Define the function $f_0\in C^\infty([0,\infty),\mathbb{R})$ by $f_0(s):=s^2/2$. Pick $\bar{\eta}>0$ as in Remark~\ref{rem0:deformlemma}. Denote $\hat{\eta}:=\min\{\bar{\eta},\lambda^2, \epsilon_f^2/2\}$. Obviously, $f\circ Q_\eta=f\circ F^*$, $r(f_0,\lambda)=\lambda$,  and $H_\eta=f_0\circ Q_\eta$ for all $\eta\in(0,\hat{\eta}]$.

To prove Theorem~\ref{thm:convexradial}~(i) and (ii), we show that the vertical homomorphisms in diagram~(\ref{CD:diag4}) are isomorphisms, and prove that the following diagram commutes:
\begin{equation}\label{CD:diag4}
\begin{split}
\xymatrix{
     {\rm HF}_*^{(-\infty,c_{f,\mu})}
     (f\circ Q_\eta+K_1;\alpha)
     \ar[d]
     \ar[r]^{[i^F]}
    &
     {\rm HF}_*^{(-\infty,c_{f,\lambda})}
     (f\circ Q_\eta+K_1;\alpha)
     \ar[d]
     \ar[d]^{[{\rm id}]}
    \\
    {\rm HF}_*^{(-\infty,\infty)}
     (f^{(\mu)}\circ Q_\eta+K_2;\alpha)
     \ar[r]^{[i^F]}_{[\sigma]}
     \ar[d]
    &
      {\rm HF}_*^{(-\infty,\infty)}
     (f^{(\lambda)}\circ Q_\eta+K_1;\alpha)
     \ar[d]^{[\sigma]}
    \\
     {\rm HF}_*^{(-\infty,\infty)}
     (f_0^{(\mu)}\circ Q_\eta+K_3;\alpha)
     \ar[d]
     \ar[r]_{[\sigma]}^{[i^F]}
    &
     {\rm HF}_*^{(-\infty,\infty)}
     (f_0^{(\lambda)}\circ Q_\eta+K_4;\alpha)
     \ar[d]^{[{\rm id}]}
    \\
     {\rm HF}_*^{(-\infty,\mu^2/2)}
     (f_0\circ Q_\eta+K_4;\alpha)
     \ar[r]^{[i^F]}_{[\sigma]}
     \ar[d]
    &
     {\rm HF}_*^{(-\infty,\lambda^2/2)}
     (f_0\circ Q_\eta+K_4;\alpha)
     \ar[d]^{(\ref{CD:diag3})}
    \\
     {\rm H}_*(\Lambda_\alpha^{\eta;\mu^2/2} M)
     \ar[r]^{[I]}
    &
     {\rm H}_*(\Lambda_\alpha^{\eta;\lambda^2/2} M)
}
\end{split}
\end{equation}
where $K_i\in C^\infty_0(S^1\times T^*M)$, $i=1,2,3,4$, are compactly supported perturbations such that the corresponding Hamiltonian $1$-periodic orbits within the action windows are non-degeneracy.
Once this is completed,  composing of the vertical maps on the right hand side of diagram~(\ref{CD:diag4}) yields the isomorphism
$$\Psi_{f}^{\eta,\lambda}:{\rm HF}_*^{(-\infty,c_{f,\lambda})}
     (f\circ Q_\eta+K_1;\alpha)
     \longrightarrow{\rm H}_*(\Lambda_\alpha^{\eta;\lambda^2/2} M)$$
     
Composing $\Psi_{f}^{\eta,\lambda}$ with the natural isomorphism ${\rm H}_*(\Lambda_\alpha^{a^2/2} M)\to {\rm H}_*(\Lambda_\alpha^{\eta;\lambda^2/2} M)$ we obtain the desired isomorphism~$\Psi_{f}^{\lambda}$.

\noindent \textbf{The proof of (i).}\quad
Observe that for every $\eta\in(0,\hat{\eta}]$,
$f\circ Q_\eta$ is constant on $D_{\sqrt{\eta}}T^*M$ (since $Q_\eta\leq\sqrt{{F^*}^2+\eta}\leq \sqrt{2\eta}\leq \epsilon_f$ and $f\equiv f(0)$ on $[0,\epsilon_f)$), and is radial outside of $D_{\sqrt{\eta}}T^*M$.
Then by Lemma~\ref{lem:radialHamsyst}, for sufficiently small perturbations $K_1$,
the image of every $1$-periodic orbit of $f\circ Q_\eta+K_1$ with action less than $c_{f,\lambda}$ belongs to $D_rT^*M$ with $r=r(f,\lambda)$, and the connecting trajectories between two such $1$-periodic orbits are located in $D_{r}T^*M$. The same is true for $f^{(\lambda)}\circ Q_\eta+K_1$ without any restriction on action. Since in $D_rT^*M$,  Hamiltonians $f\circ Q_\eta+K_1$ and $f^{(\lambda)}\circ Q_\eta+K_1$ are the same, these two chain complexes are identical.

The second vertical homomorphism is induced by Floer's continuation map $\sigma(H_s)$ associated to a homotopy $H_s$ from $f^{(\lambda)}\circ Q_\eta+K_1$ to $f_0^{(\lambda)}\circ Q_\eta+K_4$. By the homotopy invariance (see subsection~\ref{ssec:Homotopy invariance}), $[\sigma(H_s)]$ is indeed an isomorphism because outside some compact set, say, $D^F_{\hat{r}+1}T^*M$ with $\hat{r}:=\max\{r(f_0,\lambda), r(f,\lambda)\}$, $f^{(\lambda)}$ and $f_0^{(\lambda)}$ are of slope $\lambda$ and hence belong to the same component in $\mathscr{H}_{\hat{r}+1;\alpha}^{-\infty,\lambda^2/2}$.

The third vertical homomorphism is induced by the identity map between two identical chain complexes, and hence is an isomorphism. In fact, although $f_0$ is not constant near the origin, by our choice of $\eta$ we know that $H_\eta=f_0\circ Q_\eta$ has no Hamiltonian periodic orbits $z\in\Lambda_\alpha T^*M$ intersecting $D^F_{\sqrt{\eta}}T^*M$. Thus $z\in\Lambda_\alpha T^*M$ is a $1$-periodic orbit of the Hamiltonian $f^{(\lambda)}_0\circ Q_\eta$
if and only if $z$ is a $1$-periodic orbit of the Hamiltonian $H_\eta$ with action $\mathscr{A}_{H_\eta}(z)<\lambda^2/2$. Hence, for sufficiently small perturbations $K_4$, all $1$-periodic orbits and their connecting trajectories are located in $D_{\lambda}T^*M$. Moreover, in $D_{\lambda}T^*M$ two Hamiltonians $f_0\circ Q_\eta$ and $f^{(\lambda)}_0\circ Q_\eta$ are the same. So both chain complexes are identical.

The fourth vertical isomorphism is given by Theorem \ref{Thm:AS-isomorphism}. Indeed, for $K_4\in C^\infty_0(S^1\times D_\lambda^F T^*M)$ with sufficiently small $\|K_4\|_{C^2}$, the Hamiltonian $H_\eta+K_4$ are still fiber-wise uniformly convex and quadratic at infinity. Note that the Legendre transform (see~(\ref{e:L-transform})) is involutive, namely,

$$\mathfrak{L}^{-1}:S^1\times T^*M\rightarrow S^1\times TM,\quad (t,x,p)\to\big(t,x,d_pH(t,x,v)\big).$$

For any fiber-wise uniformly convex Lagrangian $L$ and its Fenchel dual Hamiltonian $H$, we have
\begin{equation}\label{e:L-transform identities}
v=\partial_p H(t,q,\partial_v L(t,q,v)),\quad \partial_{pp}H(t,q,\partial_v L(t,q,v))\partial_{vv}L(t,q,v)=Id.
\end{equation}

Denote by $L'$ the Fenchel dual Lagrangian of the Hamiltonian $H_\eta+K_4$.
Set $$V(t,x,v):=L'(t,x,v)-\frac{1}{2}L_\eta(t,x,v)\quad \forall\  (t,x,v)\in S^1\times TM.$$

By (\ref{e:L-transform identities}) and the implicit function theorem, we have $\|V\|_{C^2}$ is also small enough.
Since $K_4$ has a compactly supported set in $D^F_\lambda T^*M$, by the Legendre transform we find that $V$ is compactly supported in the disk tangent bundle $D^F_\lambda TM:=\{(x,v)\in TM|F(x,v)\leq \lambda\} $.

For $\lambda\in (0,+\infty)\setminus \Lambda_\alpha$, if $\|V\|_{C^2}$ is sufficiently small, $\lambda^2/2$ is a regular value of $\mathscr{L}^{\eta}_{V}$ on $\Lambda_\alpha M$. Therefore, by Lemma~\ref{lem:deformlemma},  we have the isomorphism
\begin{equation}\label{e:iso5a}
{\rm H}_*\big(\big\{x\in\Lambda_\alpha M\big|\mathscr{L}^{\eta}_{V}(x)<\lambda^2/2\big\}\big) \cong
{\rm H}_*\big(\big\{x\in\Lambda_\alpha M\big|\mathscr{L}^{\eta}_{0}(x)<\lambda^2/2\big\}\big).
\end{equation}

Furthermore, if the perturbation $K_4$ is chosen such that all $z\in \mathscr{P}_{\alpha}^a(H_\eta+K_4)$ are non-degenerate,  then Corollary~\ref{coro:AS-isomorphism} implies
\begin{equation}\label{e:iso5b}
{\rm HF}^{(-\infty,\lambda^2/2)}_*(H_{\eta}+K_4;\alpha) \cong {\rm H}_*\big(\big\{x\in\Lambda_\alpha M\big|\mathscr{L}^{\eta}_{V}(x)<\lambda^2/2\big\}\big).
\end{equation}
Combining (\ref{e:iso5a}) and (\ref{e:iso5b}) we get the fourth vertical isomorphism.

\noindent \textbf{The proof of (ii).}\quad It suffices to show that for $0<\eta\leq \hat{\eta}$, the following diagram commutes:
\begin{eqnarray}
\begin{CD}\label{DC:diag1}
{\rm HF}^{(-\infty,c_{f,\mu})}_*(f\circ Q_{\eta};\alpha)  @>{\iota^F}>> {\rm HF}^{(-\infty,c_{f,\lambda})}_*(f\circ Q_{\eta};\alpha)\\
@V{\Phi_{f}^{\eta,\mu}}V\simeq V  @V\simeq V{\Phi_{f}^{\eta,\lambda}}V \\
{\rm H}_*(\Lambda_\alpha^{\eta;\mu^2/2} M) @>{[\iota]}>> {\rm H}_*(\Lambda_\alpha^{\eta;\lambda^2/2} M)
\end{CD}
\end{eqnarray}
This reduces to show the commutativity of the diagram~(\ref{CD:diag4}). Let us note that the vertical maps on the left hand side of (\ref{CD:diag4}) are similar to the ones on the right hand side. The perturbations $K_2$ and $K_3$ are the restrictions of $K_1$ and $K_4$ respectively. The first four horizontal maps are induced by inclusion of subcomplexes. The first and third blocks in diagram~(\ref{CD:diag4}) commute on the chain level. The second, third and fourth horizontal maps are induced by continuation map associated to monotone homotopies. Due to the convexity of the unperturbed Hamiltonians, no orbits could enter from infinity during the homotopy, and hence continuation maps satisfy the usual composition rule for concatenations of homotopies, see, e.g., ~\cite{Sa,SZ}. So the second block in~(\ref{CD:diag4}) commute. (\ref{CD:diag3}) implies the commutativity of block four which is already described at the end of (i).

\noindent \textbf{The proof of (iii).}\quad
Let $\tilde{\eta}:=\min\{\bar{\eta},\lambda^2,\epsilon_f^2/2,
\epsilon_g^2/2\}$. It is obvious that $f\circ Q_{\eta_1}=f\circ F^*$ and $g\circ Q_{\eta_2}=g\circ F^*$.
By Remark~\ref{rem0:deformlemma},
we only need to prove that for
$0<\eta_1,\eta_2<\tilde{\eta}$ with $\eta_1\leq\eta_2$, there exists an isomorphism $\Psi^\lambda_{gf}$ such that the following diagram commutes:
\begin{eqnarray}
\begin{CD}\label{CD:diag2}
{\rm HF}^{(-\infty,c_{f,\lambda})}_*(f\circ Q_{\eta_1};\alpha) @>\Psi^\lambda_{gf}>\cong>{\rm HF}^{(-\infty,c_{g,\lambda})}_*(g\circ Q_{\eta_2};\alpha)  \\
@V{\Psi_{f}^{\eta_1,\lambda}}VV  @VV{\Psi_{g}^{\eta_2,\lambda}}V \\ {\rm H}_*(\Lambda_\alpha^{\eta_1;\lambda^2/2} M) @>[I]>\cong>
{\rm H}_*(\Lambda_\alpha^{\eta_2;\lambda^2/2} M)
\end{CD}
\end{eqnarray}

Consider the following diagram:

\begin{equation}\label{CD:diag5}
\begin{split}
\xymatrix{
     {\rm HF}_*^{(-\infty,c_{f,\lambda})}
     (f\circ Q_{\eta_1}+K_1;\alpha)
     \ar[d]_{[{\rm id}]}
     \ar@{-->}[r]^-{\Psi^\lambda_{gf}}
    &
     {\rm HF}_*^{(-\infty,c_{g,\lambda})}
     (g\circ Q_{\eta_2}+K_5;\alpha)
     \ar[d]
     \ar[d]^{[{\rm id}]}
    \\
    {\rm HF}_*^{(-\infty,\infty)}
     (f^{(\lambda)}\circ Q_{\eta_1}+K_1;\alpha)
     \ar[r]^{[\sigma]}
     \ar[d]_{[\sigma]}
    &
      {\rm HF}_*^{(-\infty,\infty)}
     (g^{(\lambda)}\circ Q_{\eta_2}+K_5;\alpha)
     \ar[d]^{[\sigma]}
    \\
     {\rm HF}_*^{(-\infty,\infty)}
     (f_0^{(\lambda)}\circ Q_{\eta_1}+K_4;\alpha)
     \ar[d]_{[{\rm id}]}
     \ar[r]^{[\sigma]}
    &
     {\rm HF}_*^{(-\infty,\infty)}
     (f_0^{(\lambda)}\circ Q_{\eta_2}+K_6;\alpha)
     \ar[d]^{[{\rm id}]}
\\
  {\rm HF}_*^{(-\infty,\lambda^2/2)}
  (H_{\eta_1}+K_4;\alpha)
  \ar[r]^{[\sigma]}
  &
  {\rm HF}_*^{(-\infty,\lambda^2/2)}
  (H_{\eta_2}+K_6;\alpha)
}
\end{split}
\end{equation}

All homomorphisms in ~(\ref{CD:diag5}) induced by
Floer's continuation maps $\sigma$ are isomorphisms which commute on homology because
no orbits could enter from $T^*M\setminus D^F_\lambda T^*M$ during the homotopy. If the compactly supported smooth perturbations $K_1,K_4,K_5$ and $K_6$ are chosen such that $\|K_1\|_{C^2}$, $\|K_4\|_{C^2}$, $\|K_5\|_{C^2}$ and $\|K_6\|_{C^2}$ are sufficiently small, $H_{\eta_2}+K_6\leq H_{\eta_1}+K_4$,  and the associated $1$-periodic Hamiltonian orbits (after perturbing) with action in the given action windows are non-degeneracy, then different perturbations achieve the identical Floer homology groups, and by a homotopy-of-homotopies argument the last two blocks in the diagram~(\ref{CD:diag5}) commute. Using (\ref{CD:diag3}) and (\ref{CD:diag5})  we can define an isomorphism $\Psi^\lambda_{gf}$ such that diagram~(\ref{CD:diag2}) commutes.
\end{proof}
We now give the proof of Lemma \ref{lem:loophomology}.
\begin{proof}[Proof of Lemma \ref{lem:loophomology}]
We first claim that there exists $\hat{\eta}>0$ such that for all $0<\eta<\hat{\eta}$  it holds that $\|d\mathscr{L}^{\eta}_0\|>\delta$ on $\{x\in\Lambda_\alpha M|\mathscr{L}^{\eta}_0(x)= a^2/2\big\}$ for some positive number $\delta$. Arguing by contradiction, since the Lagrangian functional $\mathscr{L}^{\eta}_0$ satisfies the Palais-Smale condition (cf.~\cite{Be}) we may assume that for any  $\hat{\eta}>0$, there exists $\eta\in(0,\hat{\eta}]$ and a $x_\eta\in \Lambda_\alpha M$ such that $|d\mathscr{L}^{\eta}_0(x_\eta)|=0$ and $\mathscr{L}^{\eta}_0(x_\eta)= a^2/2$.

Suppose $a<l_\alpha$, then for $\eta$ sufficiently small,
$$\int^1_0 F^2(x(t),\dot{x}(t))dt\leq\int^1_0 L_\eta(x(t),\dot{x}(t))dt +\eta\leq a^2+\eta<l_\alpha^2.$$
This contradicts the fact that $\int^1_0 F^2(x(t),\dot{x}(t))dt\geq l_\alpha^2$ on $\Lambda_\alpha M$. So $a\geq l_\alpha$. From
(\ref{eq:fixed energy level}) and the property of Legendre transform we find that if for some $\tau\in S^1$, $\dot{x}_\eta(\tau)\notin D^F_{\sqrt{\eta}}TM$ then $\dot{x}_\eta(t)\notin D^F_{\sqrt{\eta}}TM$ for all $t\in S^1$, and in this case, $x_\eta$ is also a critical point of the energy functional
$$E_F(x)=\frac{1}{2}\int^1_0 F^2(x(t),\dot{x}(t))dt \quad \forall\  x\in \Lambda_\alpha M,$$
and $E_F(x_\eta)=\mathscr{L}^{\eta}_0(x_\eta)=a^2/2$.
This contradicts $a\in(0,+\infty)\setminus \Lambda_\alpha$. Therefore, $x_\eta(t)\in D^F_{\sqrt{\eta}}T^*M$ for any $t\in S^1$. But in this case,
$$\frac{l_\alpha^2}{2}\leq \frac{a^2}{2}=\mathscr{L}^{\eta}_0(x_\eta)\leq \mathscr{L}(x_\eta)<\frac{\eta}{2}$$
which is obviously a contradiction for $\eta$ sufficiently small.
So the above claim is proved.  Moreover, since $\Lambda_\alpha$ is a closed and nowhere dense subset of $\mathbb{R}$, there exist $\varepsilon>0$ and $\eta_0>0$ such that for all $\eta<\eta_0$ we have $\|d\mathscr{L}^{\eta}_0\|>\delta$ on $\{x\in\Lambda_\alpha M\ |\ a^2/2-\varepsilon \leq\mathscr{L}^{\eta}_0(x)\leq a^2/2\big\}$.

Recall that $L_0-\eta\leq L_\eta\leq L_0$, and for any $\eta_1,\eta_2\in (0,\eta_0]$ satisfying $\eta_1\leq \eta_2$, we have
$$\big\{x\in\Lambda_\alpha M\ \big|\ \mathscr{L}^{\eta_1}_0(x)\leq a^2/2\big\}
\subseteq \big\{x\in\Lambda_\alpha M\ \big|\ \mathscr{L}^{\eta_2}_0(x)\leq a^2/2\big\}.$$
Thus $\mathscr{L}^{\eta}_0$ converges to $\mathscr{L}$ uniformly on
$$\bigcup_{\eta\in(0,\eta_0]}\{x\in\Lambda_\alpha M\ |\ \mathscr{L}^{\eta}_0(x)\leq a^2/2\big\}=\{x\in\Lambda_\alpha M\ |\ \mathscr{L}^{\eta_0}_0(x)\leq a^2/2\big\}.$$
Therefore,by Lemma~\ref{lem:deformlemma},
for any sufficiently small $\eta_1$ and $\eta_2$ with $0\leq\eta_1\leq \eta_2<\eta_0$, the inclusions
$$\big\{x\in\Lambda_\alpha M\big|\mathscr{L}^{\eta_1}_0(x)\leq a^2/2\big\}
\hookrightarrow \big\{x\in\Lambda_\alpha M\big|\mathscr{L}^{\eta_2}_0(x)\leq a^2/2\big\}$$
are homotopy equivalences, and hence
$${\rm H}_*\big(\big\{x\in\Lambda_\alpha M\big|\mathscr{L}^{\eta_1}_0(x)\leq a^2/2\big\}\big)={\rm H}_*\big(\big\{x\in\Lambda_\alpha M\big|\mathscr{L}^{\eta_2}_0(x)\leq a^2/2\big\}\big).$$
The proof of the lemma completes.
\end{proof}

\section{Computing the BPS capacity}\label{SCap}
\setcounter{equation}{0}
Following closely~\cite{We0,BPS}, we will now define certain symplectic capacities. The finiteness of these capacities in various cases will be shown in this section. For $c>0$, denote
$$\mathscr{K}_c:=\{H\in C_0^\infty(S^1\times D^FT^*M)\big|\sup_{S^1\times M}H\leq -c\}.$$

In the following, we use the conventions $\inf\emptyset=\infty$ and $\sup\emptyset=-\infty$.

\begin{definition}
{\rm
Let $\alpha\in[S^1,M]$ be a free homotopy class. For $a\in\mathbb{R}$, define the Biran-Polterovich-Salamon (BPS) capacities of $D^FT^*M$ relative to $M$ as
$$C_{\rm BPS}(D^FT^*M,M;\alpha,a):=\inf\big\{c>0\big|\forall\  H \in \mathscr{K}_c, \;\exists z \in \mathscr{P}_{\alpha}(H)\;\hbox{such that}\;\mathscr{A}_{H}(z)\geq a \big\},$$
$$C_{\rm BPS}(D^FT^*M,M;\alpha):=\inf\big\{c>0\big|\forall\  H \in \mathscr{K}_c, \mathscr{P}_{\alpha}(H)\neq\emptyset\big\}.$$
}
\end{definition}

Indeed, for general open subsets $W\subset T^*M$ containing $M$, one can also define BPS-capacities $C_{\rm BPS}(W,M;\alpha)$ of $W$ relative to $M$. Moreover, the BPS capacity has the following property.
\begin{proposition}[{Monotonicity~\cite[Proposition~3.3.1]{BPS}}]\label{prop:BPS-monotonicity}
If $W_1\subset W_2\subset T^*M$ are open subsets containing $M$ and $\alpha\in[S^1,M]$, then
$C_{\rm BPS}(W_1,M;\alpha)\leq C_{\rm BPS}(W_2,M;\alpha)$.
\end{proposition}

Given the homomorphism $T^{(a,b);c}_\alpha$ as in Proposition~\ref{prop:T-hom}, for $c>0$, we set
$$\Theta_c(D^FT^*M,M;\alpha):=\big\{a\in\mathbb{R}\;\hbox{($a>0$ if $\alpha=0$)}\;\big|
T^{(a,\infty);c}_\alpha\neq0\big\}.$$

To compute BPS-capacities, we introduce the \emph{homological relative capacity}
$$\widehat{C}_{\rm BPS}(D^FT^*M,M;\alpha,a):=\inf\big\{c>0\big|\sup \Theta_c(D^FT^*M,M;\alpha)> a \big\}$$
which bounds BPS-capacity $C_{\rm BPS}$ from above (\cite[Proposition~4.9.1]{BPS}).

The main result about calculating the BPS-capacity is as follows.
\begin{theorem}\label{thm:BPS capacities}
Let $M$ be a closed connected Finsler manifold. Then for every non-trivial free homotopy class $\alpha\in[S^1,M]$, and every $a\in\mathbb{R}$, the BPS capacities are finite and given by
$$C_{\rm BPS}(D^FT^*M,M;\alpha,a)=\max\{l_\alpha,a\},\quad C_{\rm BPS}(D^FT^*M,M;\alpha)=l_\alpha.$$
\end{theorem}
The proof of this theorem follows that of Theorem~3.2.1 and Theorem~3.3.4 in~\cite{BPS} respectively. The main ingredient in the proof is to use the following Theorem \ref{thm:symhomology} to compute $\widehat{C}_{\rm BPS}(D^FT^*M,M;\alpha,a)$.

\begin{figure}[H]
	\centering
	\includegraphics[scale=0.25]{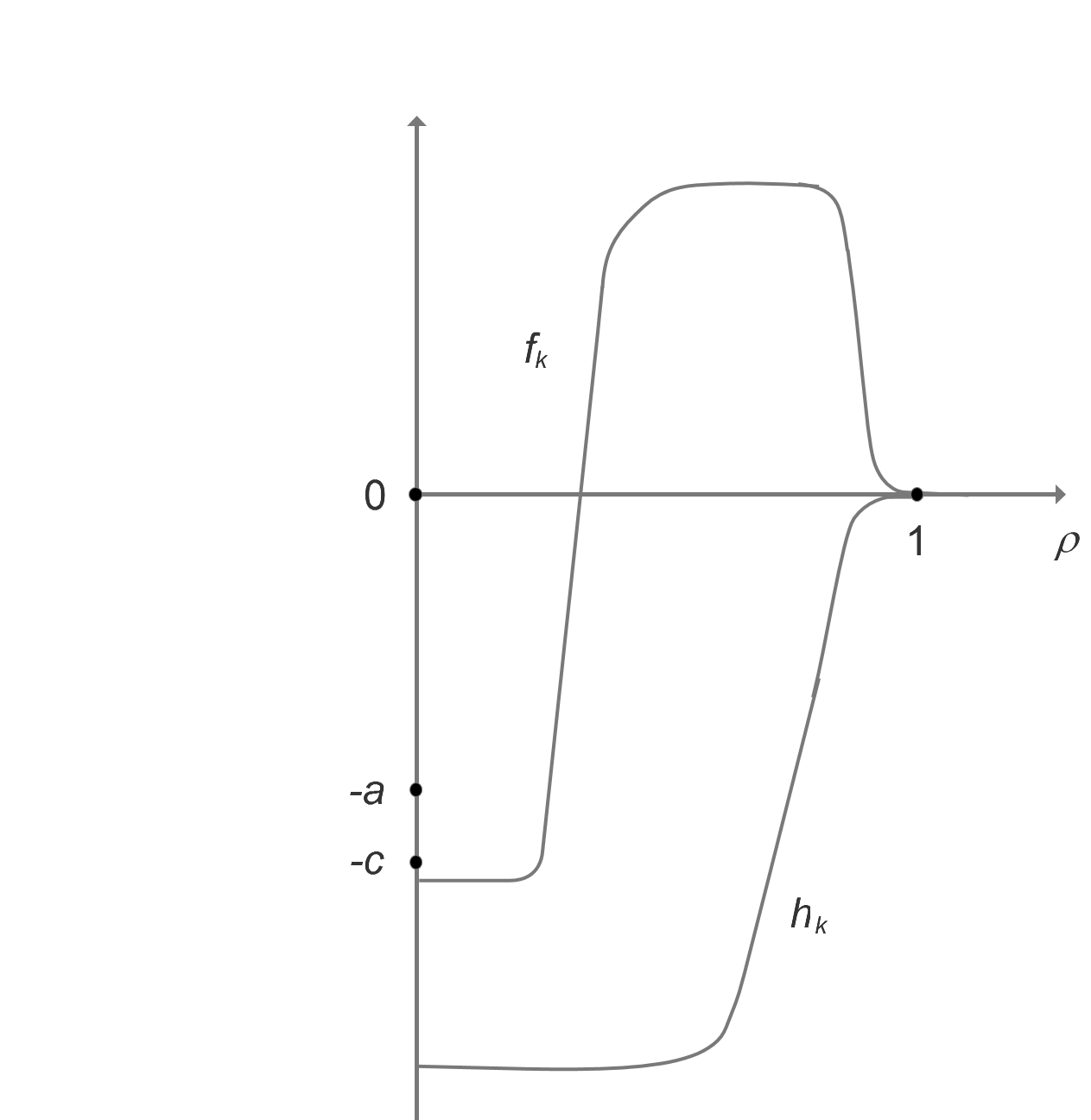}
	\caption{Two sequences of functions}\label{fig:fh}
\end{figure}


\begin{theorem}\label{thm:symhomology}
Assume that $\alpha$ is a homotopy class of free loops in $M$. Then it holds that
\begin{itemize}
  \item [(i)] if $a\in\mathbb{R}\setminus \Lambda_\alpha$, we have a natural isomorphism $\underleftarrow{{\rm SH}}
^{(a,+\infty)}_*(D^FT^*M;\alpha)\cong {\rm H}_*(\Lambda_\alpha^{a^2/2} M)$;

  \item [(ii)] for $a,c>0$, there exists a natural isomorphism
\begin{equation}\notag
\underrightarrow{{\rm SH}}^{(a,+\infty);c}_*(D^FT^*M,M;\alpha)\cong\left\{
             \begin{array}{ll}
            {\rm H}_*(\Lambda_\alpha M)&\hbox{if}\;a\in(0,c],  \\
             0 & \hbox{if}\;a>c;
             \end{array}
\right.
\end{equation}
  \item [(iii)] for any $a\in (0,c]\setminus\Lambda_\alpha$, the following diagram commutes:
\begin{equation}\label{CD:diag7}
\xymatrix{
    \\
     \underleftarrow{{\rm SH}}^{(a,+\infty)}_*(D^FT^*M;\alpha)
     \ar[r]^{\qquad\cong}
     \ar[d]_{T^{(a,\infty);c}_\alpha}
    &
     {\rm H}_*(\Lambda_\alpha^{a^2/2} M)
    \ar[d]^{[I_{a^2/2}]}
    \\
     \underrightarrow{{\rm SH}}^{(a,+\infty);c}_*(D^FT^*M,M;\alpha)
     \ar[r]^{\qquad\quad\cong}
    &
     {\rm H}_*(\Lambda_\alpha M)
}
\end{equation}
where $I_{a^2/2}$ denotes the natural inclusion $\Lambda_\alpha^{a^2/2} M\hookrightarrow \Lambda_\alpha M$.

\end{itemize}

\end{theorem}

The proof is similar to that of \cite{BPS} and \cite{We0}.
We will construct two sequences of profile functions $\{f_k\}_{k\in \mathbb{N}}$ and $\{h_k\}_{k\in \mathbb{N}}$, one of which is upward exhausting and the other one is downward exhausting. The shape of the graphs of these functions are shown in Figure~\ref{fig:fh}. 
However, there is a main difference between our profile functions and those used in \cite{We0}.
Indeed, to apply our Theorem \ref{thm:convexradial}
 to compute the Floer homology groups, we require all these functions to be constant near $\rho=0$, which is dictated by the fact that the Hamiltonian $H_{F^*}={F^*}^2/2$ is not smooth on $T^*M$ in general. Therefore, we use a different construction of the upward exhausting profile functions from \cite{We0} and use some new arguments to prove Theorem~\ref{thm:symhomology}.

The proof of this theorem is postponed to Section \ref{SSProfile}.

\section{Proof of the main theorem and its applications}\label{SProofs}

\subsection{Proofs of the main Theorem~\ref{thm:mainthm} and Theorem~\ref{coro:inv.F-length}}
\begin{proof}[Proof of Theorem~\ref{thm:mainthm}]
	Consider the Hamiltonian function defined by
	$$\overline{H}(t,z):=-H(-t,z)\quad \forall\ \;(t,z)\in S^1\times D^FT^*M.$$
	
Obviously, $x(t)$ is a periodic orbit of $H$ representing $-\alpha$ if and only if $x(-t)$ is a periodic orbit of $\overline{H}$ representing $\alpha$, and it holds that
	$$\sup_{S^1\times M}\overline{H}\leq -l_\alpha.$$
	
By Theorem~\ref{thm:BPS capacities}, we have
$C_{\rm BPS}(D^FT^*M,M;\alpha,l_\alpha)=l_\alpha$. This imlies that the set
$$\big\{c>0\big|\hbox{For any}\; H\in \mathscr{K}_c,\; \hbox{there exists}\;
z \in \mathscr{P}_{\alpha}(H)\;\hbox{such that}\;\mathscr{A}_{H}(z)\geq c
\big\}$$
is nonempty. From~\cite[Proposition~3.3.4]{BPS} we know that this set is either empty or has a minimum. Therefore, $\overline{H}$ has a Hamiltonian periodic orbits whose projection on $M$ belongs to $\alpha$. This completes the proof.

\end{proof}

\begin{proof}[Proof of Theorem~\ref{coro:inv.F-length}]
Since $\psi(M\times \{0\})$ is the graph of an exact one-form on $M$ in $D^{F_2}T^*M$, there is a $C^\infty$ function $S$ on $M$ such that
$$\psi(M\times \{0\})={\rm graph}(dS)=:\Sigma_\psi.$$
Note that $D^{F_2}T^*M-\Sigma_\psi$ is a fiberwise strictly convex subset of $T^*M$ containing $M\times \{0\}$. There exists a Finsler metric on $M$, denoted by $F_\psi$, such that the unit open disk bundle $D^{F_\psi}T^*M$ equals $D^{F_2}T^*M-\Sigma_\psi$. Now we define a vertical diffeomorphism $\nu_S$ associated to the exact $1$-form $dS$
\begin{equation}\label{eq:vertDiff}
\nu_S:T^*M\to T^*M,\quad \nu_S(x,p)=(x,p-dS(x)).
\end{equation}
It follows from (\ref{eq:vertDiff}) that
$$\nu_S^*\lambda_0-\lambda_0=\pi^*dS.$$

Denote by $\alpha_2$ and $\alpha_\psi$  the restriction of the canonical $1$-form $\lambda_0$ to the unit co-sphere bundles $S^{F_2}T^*M$ and $S^{F_\psi}T^*M$ respectively. Then (\ref{eq:vertDiff}) shows that
$\nu_S(S^{F_2}T^*M)=S^{F_\psi}T^*M$, and that the contact forms $\alpha_2$ and $\alpha_\psi$ satisfy
\begin{equation}\label{eq:contactf}
\tilde{\nu}_S^*\alpha_\psi-\alpha_2=\iota^*\pi^*dS=df,
\end{equation}
where $\tilde{\nu}_S:S^{F_2}T^*M\to S^{F_\psi}T^*M$ is induced by $\nu_S$,  $\iota:S^{F_2}T^*M\hookrightarrow T^*M$ is the natural inclusion map, and $f:=S\circ\pi\circ\iota$ is a $C^\infty$ function on $S^{F_2}T^*M$. Observe that if $\gamma$ is any smooth curve with unit speed on a Finsler manifold $(M,F)$, then we have
$${\rm len}_F(\gamma)=\int_{\ell^F\circ\dot{\gamma}}\alpha_F$$
where $\ell^F$ is the Legendre transform associated to $F$, and $\alpha_F$ is the corresponding contact form on $S^FT^*M$. This, together with (\ref{eq:contactf}), implies that $\tilde{\nu}_S$ maps closed orbits of the geodesic flow on $S^{F_2}T^*M$ to closed orbits of the geodesic flow on $S^{F_\psi}T^*M$ with the same length.
Therefore, the length spectra (the set of the lengths of closed geodesics) with respect to $F_2$ and $F_\psi$ on $M$ are the same. Moreover, since $\nu_S$ is isotopic to the identity map by the isotopy $t\mapsto \nu_{tS}$, we have $l_\alpha^{F_2}=l_\alpha^{F_\psi}$ for any nontrivial free homotopy class $\alpha\in[S^1,M]$.

On the other hand, since BPS capacities are invariant under symplectomorphisms $\psi$ and $\nu_S$, we have
$$C_{\rm BPS}(D^{F_1}T^*M,M;\alpha)=C_{\rm BPS}(D^{F_2}T^*M,\Sigma_\psi;\psi_*\alpha)=C_{\rm BPS}(D^{F_\psi}T^*M,M;{\nu_S}_*\psi_*\alpha).$$
Note that $\psi$ and $\nu_S$ are isotopic to ${\rm Id}$, we have ${\nu_S}_*\psi_*\alpha=\psi_*\alpha=\alpha$. This, together with Theorem~\ref{thm:BPS capacities} and $l_\alpha^{F_2}=l_\alpha^{F_\psi}$, shows that
$l_\alpha^{F_1}=l_\alpha^{F_2}$.

\end{proof}

\subsection{The proof of Theorem~\ref{thm:Xue-Lie}}
\begin{proof}
Since ${\rm supp}\;H$ is compact and contained in $S^1\times G\times {\rm int}~\mathcal{C}^*$, there exists a compact strictly convex domain with $C^\infty$-boundary which is denoted by $K^*$ such that
$$p^*\in {\rm int}\;K^*,\quad K^*\subseteq \mathcal{C}^*,\quad {\rm supp}\;H\subseteq S^1\times G\times {\rm int}\;K^*.$$
This implies
\begin{equation}\label{eq:zerosection}
c\geq \max\{\langle v,-X \rangle |v\in K^*-p^* \}
\end{equation}
because $X\in \mathcal{C}$, and  for every $v\in K^*-p$, one has $v+p^*\in \mathcal{C}^*$ and
$$\langle v,-X \rangle =- \langle v+p^*,X \rangle +\langle p^*,X \rangle \leq \langle p^*,X \rangle \leq c.$$

Now let us take a Minkowskian metric $F$ on $\mathfrak{g}$ which induces a
left-invariant Finsler metric on $G$ (for brevity, we simply denote it by $F$)   such that for every $x\in G$
$$(D^FT^*G)_x=x\times (K^*-p^*).$$

Define the Hamiltonian $H_{p^*}:S^1\times T^*G\to \mathbb{R}$ as
$$H_{p^*}(t,x,p):=H(t,x,p^*+p)\quad \forall\  t\in S^1,\;\forall\  x\in G,\;\forall\  p\in \mathfrak{g}^*.$$
Here we have used the identification $T^*G=G\times \mathfrak{g}^*$.

Note that since $X\in\Gamma$, $\gamma_0(t)=\exp (-tX)$ with $t\in S^1$ is a loop  in $G$ representing the free homotopy class $-\alpha$.
Then (\ref{eq:zerosection}) shows
$$c\geq F(-X)={\rm len }_F(\gamma_0)\geq \inf \big\{{\rm len }_F(\gamma)\big|[\gamma]=-\alpha\big\}.$$
Then by Theorem~\ref{thm:mainthm}, there exists $S^1\to T^* G$ such that $\dot{z}(t)=X_{H_{p^*}}(t,z(t))$ and $[z]=\alpha$.
Since $p^*\in [\mathfrak{g},\mathfrak{g}]^0$,
the map
$$T^* G\longrightarrow T^*G,\quad (x,p)\mapsto (x,p+p^*)
$$
preserves the canonical symplectic form $\omega$. Therefore, $\gamma+p^*$ is a $1$-periodic orbit of $X_{H_t}$ representing $\alpha$.

\end{proof}

\subsection{Noncompact domains, proof of Theorem \ref{thm:Ncompact}}

\begin{proof}[Proof of Theorem \ref{thm:Ncompact}]
By our assumption, for every $K_i$, $i=0,1,\ldots$ the Finsler metric $F_i:TM\to\mathbb{R}$ is given by
$$F_i(x,v):=\sup_{p\in K_i\cap T_x^*M} \langle p,v\rangle.$$

Since $K_0\subseteq K_1\subseteq K_2\subseteq\ldots$,
it is obvious that $F_0\leq F_1\leq F_2,\ldots$
Let $\gamma_i\in C^\infty(S^1,M)$ be the length minimizing $F_i$-geodesic loop representing $\alpha$. Denote by $D^{F_i}T^*M$ the unit open Finslerian disk cotangent bundle. Then we find $D^{F_i}T^*M={\rm int} K_i$.
Set
$$l_\alpha^i:={\rm len}_{F_i}(\gamma_i).$$

By Theorem~\ref{thm:BPS capacities} we have
$$C_{\rm BPS}(D^{F_i}T^*M,M;\alpha)=l_\alpha^i.$$

We claim that the sequence of numbers $\{l_\alpha^i\}_{i=0}^\infty$
are bounded. In fact, by our assumption, we have
$$\ell^{F_i}(\dot{\gamma}_i)\subseteq A$$
where $\ell^{F_i}$ denotes the Legendre transform associated to $F_i$ (see
Definition~\ref{def:Ltransform}).
This implies that for every $t\in S^1$,
\begin{equation}\label{eq:Fmetric}
F_i(\dot{\gamma}_i(t))=F_i^*\circ\ell^{F_i}(\dot{\gamma}_i(t))\leq
\sup_{(x,p)\in A}F_i^*(x,p).
\end{equation}
Note that $A$ is a compact set in $T^*M$, and $F_0^*\geq F_1^*\geq F_2^*,\ldots$  Thus integrating~(\ref{eq:Fmetric}) over $S^1$ yields
$$c:=\sup_{i}l_\alpha^i\leq \sup_{(x,p)\in A}F_0^*(x,p)<\infty.$$

Since the Hamiltonian $H\in C_0^\infty(S^1\times T^*M)$ is compactly supported in $K=\lim_iK_i$, for $i$ large enough, we have ${\rm supp}~H\subset S^1\times {\rm int}K_i=S^1\times D^{F_i}T^*M$.
If  $\min_{t,q}H(t,q,0)\geq c$,
then the finiteness of the BPS capacity $C_{\rm BPS}(D^{F_i}T^*M,M;\alpha)$
shows that there there is a $1$-periodic orbit that represents $-\alpha$.

\end{proof}

\subsection{Lorentzian Hamiltonian, proof of Theorem \ref{ThmLorentz}}

To prove this theorem, we first formulate the following result.

\begin{proposition}\label{PropCone} Let $H:S^1 \times T^*\T^n\to \R$ be a $C^\infty$ Hamiltonian compactly supported in the interior of $\mathcal C^*$. Given $p^*\in \mathcal C^*$, $\al\in \mathcal C\cap H_1(\T^n,\Z)$ and $c>0$ satisfying $$c\geq \langle p^*,\al\rangle,$$ we assume
$$\min_{q,t}H(t,q,p^*)\geq c.$$
Then $H$ admits a 1-periodic orbit in the homology class $\al$.
\end{proposition}
This proposition is proved in the same way as the above Theorem~\ref{thm:Xue'theorem}, so we skip the proof.
\begin{proof}[Proof of Theorem \ref{ThmLorentz}]
We introduce a nondecreasing function $\phi:\ \R\to \R_+$ such that $\phi(x)=0$ for $x\leq 0$ and $\phi(x)=1$ for $x\geq 1$, and a function $\varphi:\ \R_+\to [0,1]$ which is nonincreasing and satisfies $\varphi(x)=1$ for $x\in [0,0.9]$ and $\varphi(x)=0$ for $x\geq 1$. We next introduce an auxiliary Hamiltonian for any given $0<a<b$
$$
G(q,p)=\begin{cases}c\phi (\frac{H(q,p)}{b-a})\varphi(\frac{\|p\|}{R}),\quad &p\in \mathcal C,\\
0,\quad &p\notin \mathcal C^*,
\end{cases}
$$
where $c$ and $R$ are to be determined later. We fix choose $p^*\in \mathcal C^*$ such that $\min_q H(q,p^*)>b$. For a given homology class $\al\in \mathcal C$, we choose $c\geq \langle p^*,\al\rangle$. Finally, we choose $R\gg \|p^*\|$ to be further determined later.

Since we have normalized $\max V=0$, we get that $\{H>a\}\subset \T^n\times \mathcal C^*$, hence supp~$G\subset \mathcal C^*$ and $G\in C^\infty$.
Applying Proposition \ref{PropCone}, we get that $G$ admits a periodic orbit in the homology class $\al$. We assume for a moment that the periodic orbit is not created by $\varphi$, i.e. the periodic orbit does not intersect supp$\varphi'(\frac{\|p\|}{R})\frac{p}{R\|p\|}$. So the periodic orbit is also a periodic orbit of the Hamiltonian $c\phi (\frac{H(q,p)}{b-a})$. Then by the energy conservation and the injectivity of $\phi$, we get that the periodic orbit is a periodic orbit of $H$ on some energy level in $(a,b)$.

It remains to prove that the periodic orbit is not created by $\varphi$. Suppose there exists such a 1-periodic orbit $\gamma$. We write down the Hamiltonian equation
$$\begin{cases}
\dot q&=\dfrac{c}{b-a} \phi' \left(\dfrac{H(q,p)}{b-a}\right)\varphi\left(\dfrac{\|p\|}{R}\right) \dfrac{\partial H}{\partial p}+c \phi\left(\dfrac{H(q,p)}{b-a}\right)\varphi' \left(\dfrac{\|p\|}{R}\right)\dfrac{p}{R\|p\|}\\
\dot p&=-\dfrac{c}{b-a} \phi' \left(\dfrac{H(q,p)}{b-a}\right)\varphi\left(\dfrac{\|p\|}{R}\right) \dfrac{\partial V}{\partial q}
\end{cases}.$$
We consider two cases depending on whether $\gamma$ intersects the region $$D:=\mathcal C^*\cap \{p_1^2-(p_2^2+\ldots+p_n^2)\leq -\min V+b\}$$ or not.

Case 1, suppose that $\gamma\cap D=\emptyset$. This implies $H>b$ hence $\phi'(\frac{H(q,p)}{b-a})=0$. We get that $$\|\dot q\|=\left\|c \phi\varphi' \frac{p}{R\|p\|}\right\|\leq \frac{c}{R}|\varphi'|,\quad \dot p=0.$$
For large $R$, the 1-periodic orbit $\gamma$ cannot have homology class $\al\neq0$.

We remark that in this case, once $\gamma(t)\cap D=\emptyset$ for some $t\in S^1$, due to the openness of $D$ and the fact $\dot p=0$, we get that $\gamma(t) \cap D=\emptyset$ for all $t\in S^1$. This implies that if $\gamma\cap D\neq\emptyset$, then $\gamma\subset D$.

Case 2, suppose $\gamma\subset D$. Since we have assumed that $\gamma\cap \mathrm{supp}\varphi'(\frac{\|p\|}{R})\neq\emptyset$, we get that for $0.9R\leq\|p(t^*)\|\leq R$ some $t^*\in S^1$. Applying the Hamiltonian equation again we have
$$\begin{cases}
\dot q&=\frac{c\phi'\varphi}{b-a} (p_1,-q_2,\ldots,-q_n)+O(1/R)\\
\dot p&=\frac{c\phi'\varphi}{b-a}O(1)
\end{cases}.$$
as $R\to \infty$. Since $\dot p$ is uniformly bounded, within time 1, we have $0.8R\leq\|p\|\leq 1.1R$ along the periodic orbit $\gamma$. Since we also know $\gamma\subset D$, denoting $r(p)=\sqrt{p_2^2+\ldots+p_n^2}, $ we get that $|p_1-r(p)|<\frac{C}{R}$ and $|p_1+r(p)|\geq R/C$ for some constant $C$ independent of $R$.
For a fixed homology class $\al\in \mathcal C$, we have $\al_1^2-r(\al)^2>0.$ hence $0<\frac{1}{C_\al}<|\al_1\pm r(\al)|\leq C_\al$ for some constant $C_\al.$
Choosing $R$ large such that $C/R\ll 1/C_\al<C_\al\ll R/C$, we see that the 1-periodic orbit $\gamma$ cannot have homology class $\al$. Indeed, since we have $|p_1+r(p)|\geq R/C$, to attain homology class $|\al_1+ r(\al)|\leq C_\al$, we must have $\phi'\varphi=O(1/R)$, hence $\dot p=O(1/R)$. Combined with $|p_1-r(p)|<\frac{C}{R}$, this will imply that $|\al_1- r(\al)|<C'/R$ contradicting $\frac{1}{C_\al}<|\al_1- r(\al)|$.
\end{proof}

\subsection{Kawasaki's conjecture, proof of Theorem \ref{ThmKawasaki}}
\begin{proof}[Proof of Theorem \ref{ThmKawasaki}]
We take a sequence of Finsler metrics $\{F_n\}$ to approximate the degenerate Finsler metric $F(q,p )=\sum_{i=1}^n R_i|p_i|: T^*\T^n\to \R$ in the $C^0$ norm (we say that it is degenerate since the disk unit tangent bundle is not strictly convex). After Legendre transform, the disk unit cotangent bundles associated to $F_n$ approximate in the $C^0$ norm the following set
$$D^FT^*\T^n=\T^n\times\left\{\max_i \frac{1}{R_i}|p_i|\leq 1\right\}=\T^n\times\bigg(\prod_{i=1}^n[-R_i,R_i]\bigg).$$

We next pick a closed geodesic $\gamma_0$ on $\T^n$ in the homology class $-\al$ in the Euclidean metric. The length of the geodesic in the metric $F$ is len$_F(\gamma_0)=\sum_i R_i|\al_i|$. Then, for any $\eps>0$, there exists $N$ such that for all $n>N$, it holds that len$_{F_n}(\gamma_0)\leq \sum_i R_i|\al_i|+\eps.$
So we get for all $n>N$
$$\sum_i R_i|\al_i|+\eps\geq \inf\{\mathrm{len}_{F_n}(\gamma):\ \ [\gamma]=\al\}.$$

For any Hamiltonian $H$ compactly supported in the interior of $D^FT^*\T^n$, there exists $N'$ such that for all $n>N'$ we have that $H$ is also compactly supported in the interior of $D^{F_n}T^*\T^n$.
Then the theorem follows directly from the main theorem.
\end{proof}

\subsection{Symplectic nonsqueezing, proof of Theorem~\ref{thm:nonsqueezing} and ~\ref{thm:nonsqueezing'}}

\begin{proof}[Proof of Theorem~\ref{thm:nonsqueezing}]
If $s\leq r$, the inclusion $P^{2n}(s)\to Y^{2n}(r)$ is a symplectic embedding. Conversely, suppose that there exists a symplectic embedding $\phi:P^{2n}(s)\to Y^{2n}(r)$.
For sufficiently small $\epsilon>0$, let $p^*\in \Delta^n(s)$ so that $|p^*|<\epsilon$. By our assumption, the image $\phi (\mathbb{T}^n\times \{p^*\})$ is a smooth section in $T^*\mathbb{T}^n$, equivalently,
$$\Sigma_{p^*}:=\phi (\mathbb{T}^n\times \{p^*\})=\{(x,\sigma(x))|x\in \mathbb{T}^n\}$$
where $\sigma$ is a $C^\infty$ closed $1$-form on $\mathbb{T}^n$ (since $\phi(\mathbb{T}^n\times \{p^*\})$ is a Lagrangian submanifold in~$T^*\mathbb{T}^n$).
Obviously, $\Sigma_{p^*}\subseteq Y^{2n}(r)$.
Now we choose a strictly convex (closed) subset $K\subseteq \Delta^n$ with $C^\infty$-boundary sufficiently close to $\partial\Delta^n$ such that $p^*\in {\rm int} K$. Since $\phi$ is an embedding from $P^{2n}(s)$ into $Y^{2n}(r)$, $\phi(\mathbb{T}^n\times K)$ is a compact set in $Y^{2n}(r)$ containing $\Sigma_{p^*}$. Hence,  one can find a strictly fiberwise convex compact subset $T\subseteq Y^{2n}(r)$ such that $\phi(\mathbb{T}^n\times K)\subset T$. Let $F_K$ and $F_T$ be the Finsler metrics on $T^*\mathbb{T}^n$ associated to $\mathbb{T}^n\times (K-p^*)$ and $T-\Sigma_{p^*}$ respectively, namely,

$$F_K(x,v):=\sup_{p\in (K-p^*)} \langle p,v\rangle,$$
$$F_T(x,v):=\sup_{p\in (T-\Sigma_{p^*})\cap T^*_x\mathbb{T}^n} \langle p,v\rangle.$$
It is obvious from the above definitions that
$D^{F_K}T^*\mathbb{T}^n=\mathbb{T}^n\times (K-p^*)$
and $D^{F_T}T^*\mathbb{T}^n=T-\Sigma_{p^*}$.
Observe that for any closed $1$-form $\sigma$ on $\mathbb{T}^n$, the map
$$\tau_{\sigma}:TT^* \mathbb{T}^n\longrightarrow TT^* \mathbb{T}^n,\quad (x,p)\mapsto (x,p-\sigma(x))
$$
preserves the canonical symplectic form $\omega_0$.
The map
$$\Phi:\mathbb{T}^n\times (K-p^*)\longrightarrow T-\Sigma_{p^*},\quad \Phi=\tau_{\sigma}\circ\phi\circ\tau_{-p^*}$$
is a symplectic embedding satisfying $\Phi(\mathbb{T}^n\times \{0\})=\mathbb{T}^n\times \{0\}$.  
Since the BPS-capacity is invariant under $\tilde{\pi}_1$-trivial symplectic diffeomorphism,
Proposition~\ref{prop:BPS-monotonicity} shows that for every non-trivial $\alpha\in [S^1,\mathbb{T}^n]\cong \mathbb{Z}^n$,
\begin{equation}\label{eq:BPScomparison}
C_{\rm BPS}(D^{F_K}T^*\mathbb{T}^n,\mathbb{T}^n;\alpha)=C_{\rm BPS}(\Phi (D^{F_K}T^*\mathbb{T}^n),\Phi(\mathbb{T}^n);\alpha)\leq C_{\rm BPS}(D^{F_T}T^*\mathbb{T}^n,\mathbb{T}^n;\alpha)
\end{equation}
Set $e_1:=(1,0,\ldots,0)\in \mathbb{Z}^n$
 and let $\gamma(t)=[e_1t]$ be a closed curve in $\mathbb{T}^n=\mathbb{R}^n/\mathbb{Z}^n$ representing $e_1$. To calculate the length of $\gamma$ with respect to the Finsler metrics $F_K$ and $F_T$ respectively, we find
$$F_K(x,e_1)=\sup_{p\in (K-p^*)} \langle p,e_1\rangle=\sup_{p\in K} \langle p,e_1\rangle-\langle p^*,e_1\rangle\geq s-2\epsilon,$$
$$F_T(x,e_1)=\sup_{p\in (T-\Sigma_{p^*})\cap T^*_x\mathbb{T}^n} \langle p,e_1\rangle\leq r.$$
So we have ${\rm len}_{F_K}(\gamma)\geq s-2\epsilon$ and ${\rm len}_{F_T}(\gamma)\leq r$.
Note that $F_K$ is invariant under translations
$$\tau_a:\mathbb{T}^n\longrightarrow \mathbb{T}^n,\quad \tau_a(y)=y+a,$$ $\gamma(t)$ is the length minimizing closed $F_K$-geodesics in class $e_1$. Theorem~\ref{thm:BPS capacities}, together with (\ref{eq:BPScomparison}), implies
$$s-2\epsilon\leq {\rm len}^{F_K}(\gamma)=l_{e_1}^{F_K}\leq l_{e_1}^{F_T}\leq {\rm len}^{F_T}(\gamma)\leq r.$$
Since $\epsilon$ is arbitrary, the proof of the theorem completes.

\end{proof}

The proof of Theorem~\ref{thm:nonsqueezing'} is essentially analogous to that of Theorem~\ref{thm:nonsqueezing}.

\begin{proof}[Proof of Theorem~\ref{thm:nonsqueezing'}] It is obvious that for $s\leq r$ the inclusion $\T^n \times B^{n}(s)\to \T^n \times Z^{n}(r)$ is a symplectic embedding. Conversely, if there is a symplectic embedding $\phi:\T^n \times B^{n}(s)\to \T^n \times Z^{n}(r)$. By the assumption, every image $\phi (\mathbb{T}^n\times \{u_i\})$ is a smooth section in $T^*\mathbb{T}^n$, meaning that
$$\Sigma_{i}:=\phi (\mathbb{T}^n\times \{u_i\})=\{(x,\sigma_i(x))|x\in \mathbb{T}^n\}$$
with closed $1$-forms $\sigma_i\in \Omega^1(\mathbb{T}^n)$. One can choose strictly convex compact subsets $K_i\subseteq B^{n}(s)$ containing $u_i$ with $C^\infty$ boundaries approximating $\partial B^{n}(s)$ and fibewise strictly convex compact subsets $ T_i\subseteq \T^n \times Z^{n}(r)$ containing $\phi (\mathbb{T}^n\times K_i)(\supseteq\Sigma_{i})$.
Pick the homotopy class $-e_{1}=(-1,0\ldots,0)\in\mathbb Z^n$. Then $\gamma(t)=[-e_1t]$ is a closed curve in $\mathbb{T}^n=\mathbb{R}^n/\mathbb{Z}^n$ representing $-e_1$. Let $F_{K_i}$ and $F_{T_i}$ be the Finsler metrics on $T^*\mathbb{T}^n$ associated to $\mathbb{T}^n\times (K_i-u_i)$ and $T_i-\Sigma_{i}$ respectively.  Finally, arguing as in the proof of Theorem~\ref{thm:nonsqueezing} by making use of the monotonicity property of BPS-capacity and the fact that $|u_i-(s,0\ldots,0)|$ is sufficiently small for large $i>N$
leads to $s\leq r$.

\end{proof}

\subsection{Symplectic squeezing, proof of Theorem \ref{ThmCBPSY}}
\begin{proof}[Proof of Theorem \ref{ThmCBPSY}]

Let $U$ be an open convex set in $\R^{n}$, not necessarily strictly convex or compact, and $\{U_i\}$ be a sequence of strictly convex sets with $C^\infty$ boundaries and satisfying $U_i\subset U_{i+1}$ and $\lim U_i=U$. Let $F_i$ be the sequence of Finsler metrics associated to $U_i$ via $F_i(q,\dot q)=\sup_{p\in U_i}\langle p,\dot q\rangle$ for any $(q,\dot q)\in T^*\T^n$.

For each $H$ compactly supported in $\T^n\times U$, we have that there exists some $I$ such that supp$(H)\subset \T^n\times U_i$ for all $i>I$. This shows that $C_{\rm BPS}(\T^n\times U,\T^n, \alpha)=\lim C_{\rm BPS}(\T^n\times U_i,\T^n, \alpha)$.

Suppose now $v$ is proportional to an integer vector $\alpha$.

We define $U_i$ to be a sequence of ellipsoids centered at the origin with minor axis being the vector $rv$ so that they are all tangent to $Y^{2n}(r,v)$ at the two points $\pm rv$.

A closed geodesic with homology class $\alpha$ has constant velocity $\dot q=\alpha$. For each $i$, the sup $ r|\alpha|$ in the definition of $F_i$ is always attained at the point $rv$ or $-rv$. In fact, for all $p$ on one boundary of $Y^{2n}(r,v)$, the inner product $\langle p,\dot q\rangle$ is constant $\pm r|\alpha|$. So in this case, $C_{\rm BPS}(Y^{2n}(r, v),\T^n,\pm\alpha)=r|\alpha|$.

Suppose $v$ is not proportional to an integer vector. In this case, for any $\alpha\in H_1(\T^n,\Z)\setminus\{0\}$, we have that $\alpha$ is not perpendicular to $v^\perp$. So we can find $p$ on the boundary of $Y^{2n}(r,v)$ such that $\langle \alpha,p\rangle$ is as large as we wish. For any $N>0$, we can find $i$ and $p_i\in \partial U_i$ such that $\langle \alpha,p_i\rangle>N$.  This shows that $C_{\rm BPS}(Y^{2n}(r, v),\T^n,\pm\alpha)=\infty$.

\end{proof}

\section{Proof of the main technical results}\label{STechnical}
\subsection{Convexity results, proof of Lemma \ref{lem:convexity}}\label{SSMP}
\begin{proof}[Proof of Lemma \ref{lem:convexity}]
Consider the function $r:\mathbb{R}\times S^1\to \mathbb{R}$ defined by
$$r(s,t):=\rho(u(s,t)).$$

Arguing by contradiction, we assume that the open subset of $\mathbb{R}\times S^1$
$$\Sigma:=r^{-1}((\rho_1,\infty))=u^{-1}((\rho_1,\infty)\times\partial X) $$
is not empty. By our assumption, $\Sigma$ is bounded, so $r$ achieves its maximum on $\Sigma$, and
\begin{equation}\label{e:convexity0}
r_0=\max_{(s,t)\in\Sigma}r(s,t)>\rho_1.
\end{equation}

It is easy to verify that
$$X_{H_{s,t}}=\frac{\partial H_{s,t}}{\partial \rho}R\quad \hbox{on}\; [\rho_0,\infty)\times\partial X.$$
Then, by~(\ref{e:contactJ}) and our assumptions concerning $H_{s,t}$ and $J_{s,t}$,
\begin{eqnarray}\label{e:convexity1}
\frac{\partial r}{\partial s}&=&d\rho(\partial_su)\notag\\
&=&d\rho\circ J_{s,t}(u)(-\partial_tu+X_{H_{s,t}}(u))
\notag\\
&=&\widehat{\lambda}(-\partial_tu+\partial_\rho H_{s,t}(u)R)
\notag\\
&=&-\widehat{\lambda}(\partial_tu)+\mu f(s)r^\mu(s,r).
\end{eqnarray}
Similarly,
\begin{equation}\label{e:convexity2}
\frac{\partial r}{\partial t}=d\rho(\partial_tu)=d\rho\circ J_{s,t}(\partial_su-J_{s,t}X_{H_{s,t}})=\widehat{\lambda}
(\partial_su-\partial_\rho H_{s,t}(u)Z(u))=\widehat{\lambda}(\partial_su).
\end{equation}

Write (\ref{e:convexity1}) and (\ref{e:convexity2}) in a less coordinate-bound way as
\begin{equation}\label{e:convexity3}
d^c r=\frac{\partial r}{\partial t}ds-\frac{\partial r}{\partial s}dt
=u^*\widehat{\lambda}-\mu f(s)r^\mu dt.
\end{equation}
Note that $dd^c r=-\triangle rds\wedge dt$ and
$$|\partial_s u|_{J_{s,t}}^2=d\widehat{\lambda}( J_{s,t}\partial_s u,\partial_s u)
=d\widehat{\lambda}( \partial_t u-X_{H_{s,t}}(u),\partial_s u)=-d\widehat{\lambda}(\partial_s u, \partial_t u)+dH_{s,t}(\partial_su).$$
Then by differentiating (\ref{e:convexity3}) we arrive at
\begin{equation}\label{e:convexity4}
\triangle r-\mu(\mu-1)f(s)r^{\mu-1}\partial_sr(s,t) =|\partial_s u|_{J_{s,t}}^2+\mu f'(s)r^\mu.
\end{equation}
Keeping in mind that $r>0$ on $\Sigma$ by definition, the left hand side of
(\ref{e:convexity4}) is nonnegative by our assumptions, then the maximum principle implies that $r$ achieves its maximum on the boundary $\partial \Sigma$ of $\Sigma$. But $r|_{\partial \Sigma}=\rho_1$,
this contradicts with (\ref{e:convexity0}). So $\Sigma$ must be empty. This completes the proof of Lemma~\ref{lem:convexity}.
\end{proof}

\subsection{The radial Hamiltonian and its action, proof of Lemma \ref{lem:radialHamsyst}}\label{SSRadial}

\begin{proof}[Proof of Lemma \ref{lem:radialHamsyst}]
For each $(x,y)\in T^*M\setminus \{0\}$, locally, we write
$$F^{*2}(x,y)=g^{ij}(x,y)y_iy_j.$$
Then,
$$dH(x,y)=h'(F^{*2}(x,y))\bigg[\frac{\partial g^{ij}}{\partial x_k}y_iy_jdx^k
+\frac{\partial g^{ij}}{\partial y_l}y_iy_jdy^l+2g^{ij}y_idy_j\bigg]$$

Using $-dH=\omega_0(X_H,\cdot)$, we get
$$X_H=h'(F^{*2}(x,y))\bigg[-\frac{\partial g^{ij}}{\partial x_k}y_iy_j\frac{\partial}{\partial y_k}
+\bigg(\frac{\partial g^{ij}}{\partial y_l}y_iy_j+2g^{il}y_i\bigg)\frac{\partial}{\partial x_l}\bigg]$$
which gives the Hamiltonian equation
\begin{equation}\label{eq:hameq}
\left\{
             \begin{array}{ll}
            \dot{x}_l=h'(F^{*2}(x,y))\big(\frac{\partial g^{ij}}{\partial y_l}y_iy_j+2g^{il}y_i\big)=2h'(F^{*2}(x,y))g^{il}y_i,  \\
            \dot{y}_k=-h'(F^{*2}(x,y))\frac{\partial g^{ij}}{\partial x_k}y_iy_j
             \end{array}
\right.
\end{equation}
Here we have use the homogeneity property of the Finsler metric $F$:
$$\frac{\partial g^{ij}}{\partial y_l}y_i=\frac{\partial^3(F^2)}{2\partial y_i\partial y_j\partial y_l}y_i=0$$

For a solution $z(t)$ of the Hamiltonian equation~(\ref{eq:hameq}) sitting in $T^*M\setminus \{0\}$, we compute

\begin{eqnarray}\label{eq:fixed energy level}
\frac{d}{dt}\big[F^{*2}(x(t),y(t))\big]&=&\frac{\partial g^{ij}}{\partial x_k}\dot{x}_ky_iy_j+\frac{\partial g^{ij}}{\partial y_l}\dot{y}_ly_iy_j+2g^{ij}(x,y)\dot{y}_iy_j\nonumber\\
&=&\frac{\partial g^{ij}}{\partial x_k}h'(F^{*2}(x,y))2g^{lk}y_ly_iy_j
-2g^{ij}(x,y)h'(F^{*2}(x,y))\frac{\partial g^{kl}}{\partial x^i}y_ky_ly_j\notag\\
&=&0
\end{eqnarray}
which implies
$$F^{*2}(x(t),y(t))\equiv constant\quad\forall\  t\in\mathbb{R}.$$

Denote $C=h'(F^{*2}(x(t),y(t)))$. The Legendre transform,
$$\tau_x:T_x^*M\rightarrow T_xM\quad (x,y)\mapsto (x,v),$$
where $v=\sum v^i\frac{\partial}{\partial x_i}$ and $v^l=\sum g^{il}y_i$, implies
\begin{eqnarray}\notag
\dot{v}^l&=&\frac{\partial g^{il}}{\partial x_k}\dot{x}_ky_i+\frac{\partial g^{il}}{\partial y_k}\dot{y}_ky_i+g^{il}(x,y)\dot{y}_i\nonumber\\
&=&\frac{\partial g^{il}}{\partial x_k}\dot{x}_ky_i+g^{il}(x,y)\dot{y}_i\nonumber\\
&=&\frac{\partial g^{il}}{\partial x_k}2Cg^{jk}y_jy_i
-Cg^{il}(x,y)g^{il}\frac{\partial g^{jk}}{\partial x_i}y_jy_k\notag\\
&=&\frac{\partial g^{il}}{\partial x_k}2Cv^kg_{it}v^t-Cg^{il}(x,y)g^{il}\frac{\partial g^{jk}}{\partial x_i}y_jy_k\notag\\
&=&-2Cg^{il}\frac{\partial g_{it}}{\partial x_k}v^kv^t+Cg^{il}\frac{\partial g_{jt}}{\partial x_i} g^{jk}g_{ks}v^s\notag\\
&=&-Cg^{il}\bigg(2\frac{\partial g_{it}}{\partial x_k}-\frac{\partial g_{kt}}{\partial x_i}\bigg)v^kv^t.
\end{eqnarray}
Hence,
\begin{eqnarray}\notag
\nabla^F_{\dot{x}}v&=&\big[ \dot{v}^l(t) +v^j(t)\dot{x}_k(t)\Gamma^l_{jk}(x(t),v(t))\big]\frac{\partial}{\partial x_i}\bigg|_{x(t)}\notag\\
&=&\big[ -Cg^{il}\bigg(2\frac{\partial g_{it}}{\partial x_k}-\frac{\partial g_{kt}}{\partial x_i}\bigg)v^kv^t +v^j(t)2Cv^k(t)\Gamma^l_{jk}(x(t),v(t))\big]\frac{\partial}{\partial x_i}\bigg|_{x(t)}.\notag
\end{eqnarray}
Here
$$\Gamma^i_{jk}=\gamma^i_{jk}-\frac{g^{il}}{F}\bigg(
A_{ljs}N_k^s-A_{jks}N_i^s+A_{kls}N_j^s\bigg)$$
are the components of the Chern connection with the Christoffel symbols
$$\gamma^l_{jk}:=\frac{1}{2}g^{ls}\bigg(\frac{\partial g_{sj}}{\partial x_k}-\frac{\partial g_{jk}}{\partial x_k}+\frac{\partial g_{ks}}{\partial x_j}\bigg),$$
the Cartan tensor
$$A_{ijk}(x,y):=\frac{F}{2}\frac{\partial g_{ij}}{\partial y_k}=\frac{F}{4}\frac{\partial^3(F^2)}{\partial y_i\partial y_j\partial y_k}$$
and the coefficients of the nonlinear connection on $TM\setminus \{0\}$
$$N^i_j(x,y):=\gamma^i_{jk}y_k-\frac{1}{F}A^i_{jk}\gamma^k_{rs}y_ry_s$$

Using the homogeneity property of Finsler metric again, we have
$$\Gamma^i_{jk}(x,v)v^jv^k=\frac{1}{2}g^{ls}\bigg(2\frac{\partial g_{sj}}{\partial x_k}-\frac{\partial g_{jk}}{\partial x_s}\bigg)v^j v^k$$
Therefore,
$$\nabla^F_{\dot{x}}v=0.$$
So, under the Legendre transform  $\ell:T^*M\rightarrow TM$ induced by the Finsler metric, the Hamiltonian equation associated to $H$ can be written as
\begin{equation}\label{e:HamEq}
\left\{
             \begin{array}{ll}
            \dot{x}=2Cv,  \\
            \nabla^F_{\dot{x}}v=0
             \end{array}
\right.
\end{equation}
Then
$$\nabla^F_{\dot{x}}\dot{x}=2C\nabla^F_{\dot{x}}v=0$$
which means that $x(t)$ is a Finsler geodesic loop on $M$. By (\ref{e:HamEq}), we have
\begin{eqnarray}\notag
F^2(x(t),\dot{x}(t))&=&g_{ij}\dot{x}_i\dot{x}_j\notag\\
&=&4C^2g_{ij}v^iv^j\notag\\
&=&4C^2g_{ij}g^{ik}y_kg^{jl}y_l\notag\\
&=&4C^2F^{*2}(x(t),y(t))\notag\\
&=&4h'(F^{*2}(x(t),y(t)))^2F^{*2}(x(t),y(t))\notag\\
&=&\big(f'(F^*(x(t),y(t)))\big)^2
\end{eqnarray}
which implies that the Finsler length of the geodesic loop $x(t)$ equals to $\pm f'(r_z)$ for some constant $r_z=F^*(x(t),y(t))\geq 0$. Then
$$r_z\ell(\dot{x})=2r_zC\ell(v)=2r_zCy= 2r_zh'(r_z^2)y= f'(r_z)y.$$
Moreover,  the action of $\mathscr{A}_{H}$ at $z(t)$ can be computed as follow:
\begin{eqnarray}\notag
\mathscr{A}_{H}(z(t))&=&\int_{S^1}\langle y_i(t),\dot{x}_i(t)\rangle-\int_{S^1}h(F^{*2}(x,y))\notag\\
&=&\int_{S^1}2Cg^{il}y_ly_i-\int_{S^1}h(F^{*2}(x,y)) \notag\\
&=&2h'(F^{*2}(z(t)))F^{*2}(z(t))-h(F^{*2}(z(t)))\notag\\
&=&f'(F^*(z(t)))F^*(z(t))-f(F^*(z(t))).
\end{eqnarray}

Conversely, if $x$ is a $F$-geodesic loop satisfying~(\ref{e:F-geo&Hamorbit}), it is easy to verify that the loop $z=(x,y)$ is a critical point of the action functional $\mathscr{A}_{H}$, we omit it here.
\end{proof}

\subsection{Quadratic modifications, proof of Lemma \ref{lem:modification}}\label{sec:Pfmodification}
\begin{proof}[Proof of Lemma \ref{lem:modification}]Set
$$a_0:=\inf\limits_{x\in M,|v|_x=1}\inf_{u\neq 0}\frac{g^F(x,v)[u,u]}{g_x(u,u)}\quad \hbox{and}\quad
a_1:=\sup\limits_{x\in M,|v|_x=1}\sup_{u\neq 0}\frac{g^F(x,v)[u,u]}{g_x(u,u)}.$$
Due to the compactness of $M$, we have $0<a_0\leq a_1$. Since $g^F$ is positively homogeneous of degree $0$,  it holds that for any $(x,v)\in TM\setminus\{0\}$ and $(x,u)\in TM$,
$$a_0g_x(u,u)\leq g^F(x,v)[u,u]\leq a_1 g_x(u,u).$$
By rescaling the metric $g$ (for instance, by choosing $\tilde{g}:=a_0g$), we may assume that
\begin{equation}\label{e:F-metric vs.g-metric}
g_x(u,u)\leq g^F(x,v)[u,u]\leq A g_x(u,u)\quad \forall\  (x,v)\in TM\setminus\{0\},\; \forall\  (x,u)\in TM
\end{equation}
for some constant $A>1$. In particular, we have
\begin{equation}\label{e:F&g}
|v|_x^2\leq L_0(x,v)\leq A |v|_x^2\quad \forall\  (x,v)\in TM.
\end{equation}

In order to construct convex quadratic Lagrangians, we follow closely the line of \cite[Section~2]{Lu} by modifying $L_0$ near the zero section of $TM$. In the following we first construct two auxiliary functions, see Figure~\ref{fig:2}.

Choose positive parameters $0<\epsilon<\delta<\frac{\eta}{A}$. Let $\lambda^\mu_{\epsilon,\delta}:(0,\infty)\to \mathbb{R}$ be a smooth function such that $\lambda^\mu_{\epsilon,\delta}(s)=0$ for $s\in[0,\epsilon)$, $\lambda^\mu_{\epsilon,\delta}(s)=\mu s+\sigma$ for $s\in(\delta,\infty)$, $\lambda^\mu_{\epsilon,\delta}$ is convex and $(\lambda_{\epsilon,\delta}^\mu)'(s)>0$ on $(\epsilon,\infty)$, where $\mu>0$ and $\sigma<0$ are suitable constants. Let $\chi^\kappa_{\delta,\rho}$ be another smooth function such that $\chi^\kappa_{\delta,\rho}(s)=\kappa (s-\delta)$ for $s\in[0,\delta]$, $\chi^\kappa_{\delta,\rho}(s)=\rho$  for $s\in[\eta/A,\infty)$, and $\chi^\kappa_{\delta,\rho}$ is concave and nondecreasing, here $\kappa>0$ and $\rho>0$ are suitable constants.

\begin{figure}[H]
  \centering
  \includegraphics[scale=0.5]{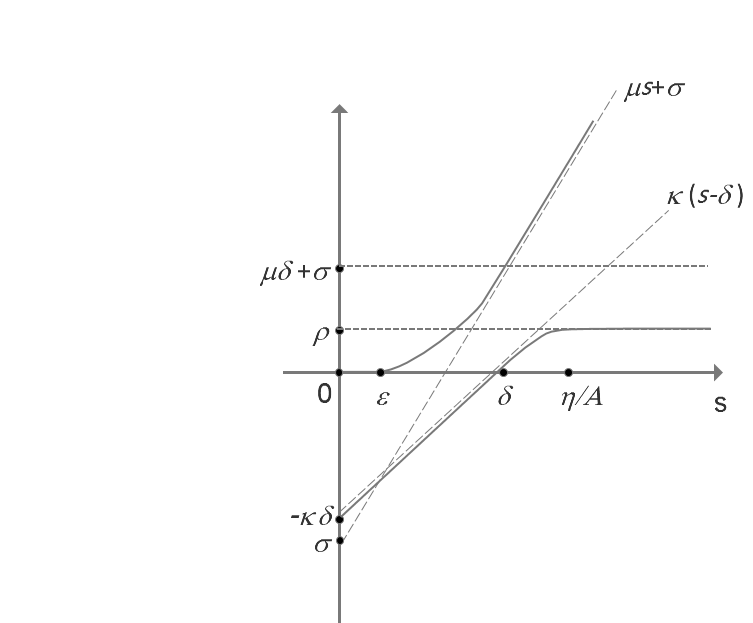}
  \caption{The auxiliary functions}\label{fig:2}
\end{figure}


Define the new Lagrangians $L_\eta$ by
$$L_\eta(x,v):=\frac{1}{\mu}\bigg\{\lambda^\mu_{\epsilon,\delta}(L_0(x,v))+  \chi^\kappa_{\delta,\rho}(|v|^2_x)-\sigma-\rho\bigg\}.$$
We check properties (a) and (b) in two steps.

\noindent \textbf{Step 1.}\quad Obviously, for every $\eta>0$, $L_\eta$ is smooth on the whole $TM$.
Suppose that $\mu\geq\kappa$ and $\kappa$ is sufficiently large such that
\begin{equation}\label{e:constrLagr1}
\delta\kappa+\sigma+\rho>0\quad\hbox{and}\quad \kappa(\delta-\epsilon)+\sigma>0.
\end{equation}

If $0\leq L_0(x,v)\leq\epsilon$ (hence $|v|^2_x\leq\epsilon$ by~(\ref{e:F&g})), by the first inequality in~(\ref{e:constrLagr1}),
\begin{equation}\label{e:constrLagr2}
L_\eta(x,v)=\frac{1}{\mu}\big\{\kappa(|v|^2_x-\delta)-\sigma-\rho\big\}=
\frac{\kappa}{\mu}|v|^2_x-\frac{\kappa\delta+\sigma+\rho}{\mu}
\leq L_0(x,v).
\end{equation}

If $\epsilon\leq L_0(x,v)\leq\delta$, observe that $\lambda^\mu_{\epsilon,\delta}(s)\leq\frac{\mu\delta+\sigma}
{\delta-\epsilon}(s-\epsilon)$ for every $s\in(\epsilon,\delta)$, we have
\begin{eqnarray}\label{e:constrLagr3}
L_\eta(x,v)&\leq&\frac{1}{\mu}\bigg\{\frac{\mu\delta+\sigma}
{\delta-\epsilon}(L_0(x,v)-\epsilon)+\kappa(|v|^2_x-\delta)-\sigma-\rho\bigg\}
\notag\\
&\leq&\frac{1}{\mu}\bigg\{\frac{\mu\delta+\sigma}
{\delta-\epsilon}(L_0(x,v)-\epsilon)+\kappa(L_0(x,v)-\delta)-\sigma-\rho\bigg\}
\notag\\
&\leq&\frac{1}{\mu}\bigg\{\bigg(\frac{\mu\delta+\sigma}
{\delta-\epsilon}+\kappa\bigg)L_0(x,v)-\frac{\mu\delta+\sigma}
{\delta-\epsilon}\epsilon-\kappa\delta-\sigma-\rho\bigg\}.
\end{eqnarray}
Then
\begin{equation}\label{e:constrLagr4}
L_\eta(x,v)-L_0\leq\frac{1}{\mu}\bigg\{\bigg(\frac{\mu\epsilon+\sigma}
{\delta-\epsilon}+\kappa\bigg)L_0(x,v)-\frac{\mu\delta+\sigma}
{\delta-\epsilon}\epsilon-\kappa\delta-\sigma-\rho\bigg\}
\end{equation}

By the second inequality in~(\ref{e:constrLagr1}),
$$\frac{\mu\epsilon+\sigma}
{\delta-\epsilon}+\kappa>\frac{\mu\epsilon}
{\delta-\epsilon}>0.$$
So, by~(\ref{e:constrLagr3}), if $\epsilon\leq L_0(x,v)\leq\delta$, then
\begin{eqnarray}\label{e:constrLagr5}
L_\eta(x,v)-L_0&\leq&\frac{1}{\mu}\bigg\{\bigg(\frac{\mu\epsilon+\sigma}
{\delta-\epsilon}+\kappa\bigg)\delta-\frac{\mu\delta+\sigma}
{\delta-\epsilon}\epsilon-\kappa\delta-\sigma-\rho\bigg\}
\notag\\
&=&-\frac{\rho}{\mu}<0.
\end{eqnarray}
Therefore, if $\epsilon\leq L_0(x,v)\leq\delta$, $L_\eta(x,v)\leq L_0$.

If $L_0(x,v)\geq\delta$,
\begin{eqnarray}\label{e:constrLagr6}
L_\eta(x,v)=\frac{1}{\mu}\big\{\mu L_0(x,v)+\sigma+\chi^\kappa_{\delta,\rho}(|v|^2_x)-\sigma-\rho\big\}
\leq L_0(x,v),
\end{eqnarray}
and in this case, if $L_0(x,v)\geq\eta$ (thus $|v|_x^2>\eta/A$ by~(\ref{e:F&g})), we have
\begin{equation}\label{e:constrLagr7}
L_\eta(x,v)=L_0(x,v).
\end{equation}

\noindent\textbf{Step 2.}\quad By our assumption, $\lambda^\mu_{\epsilon,\delta}$ is convex, so for any $s\in[0,\infty)$ we have
\begin{eqnarray}\label{e:constrLagr8}
\lambda^\mu_{\epsilon,\delta}(s)&\geq& \frac{d\lambda^\mu_{\epsilon,\delta}}{ds}\bigg|_{s=\delta}
(s-\delta)
+\lambda^\mu_{\epsilon,\delta}(\delta)\notag\\
&=&\mu(s-\delta)+\mu \delta+\sigma=\mu s+\sigma.
\end{eqnarray}
$\chi^\kappa_{\delta,\rho}(s)\geq \chi^\kappa_{\delta,\rho}(0)=-\kappa\delta$ because $\chi^\kappa_{\delta,\rho}$ is nondecreasing on $[0,\infty)$.
Combing this with (\ref{e:constrLagr8}) shows
$$L_\eta\geq\frac{1}{\mu}\bigg\{\mu L_0+\sigma-\kappa\delta-\sigma-\rho\bigg\}=L_0-\frac{\kappa\delta+\rho}{\mu}.$$

Since $0<\delta<\eta$ and $0<\kappa\leq \mu$ by our assumption, taking $\mu>0$ sufficiently large yields
$$L_\eta(x,v)\geq L_0-\eta.$$

To complete the proof, it suffices to prove that $L_\eta$ is fiberwise convex and quadratic at infinity. Due to the compactness of $M$ and  $L_0(x,v)=g^F(x,v)[v,v]$ for any $(x,v)\in TM\setminus\{0\}$,
the condition (\textbf{L2}) holds obviously. To show (\textbf{L1}), for every $v\in TM\setminus\{0\}$ and $u\in T_xM$ we compute
\begin{eqnarray}\label{e:convexL1}
\partial_{vv}L_\eta(x,v)[u,u]&=&\frac{\partial^2}{\partial s\partial t}
L_\eta(x,v+su+tu)\bigg|_{s=t=0}\nonumber\\
&=&\frac{1}{\mu}\bigg\{(\lambda^\mu_{\epsilon,\delta}){''}(L_0(x,v))
\big(\partial_vL_0(x,v)[u]\big)^2+4(\chi^\kappa_{\delta,\rho}){''}(|v|_x^2)
g_x(v,u)^2\nonumber\\
&&+(\lambda^\mu_{\epsilon,\delta})'(L_0(x,v))\partial_{vv}L_0(x,v)[u,u]
+2(\chi^\kappa_{\delta,\rho})'(|v|_x^2)|u|_x^2\bigg\}.
\end{eqnarray}

We prove (\textbf{L1}) in two cases.

\noindent \textbf{Case 1. }If $L_0(x,v)<\delta$,   $|v|_x<\delta$ by (\ref{e:F&g}), and thus
\begin{equation}\label{e:convexL2}
\partial_{vv}L_\eta(x,v)[u,u]\geq \frac{2\kappa}{\mu}|u|^2_x.
\end{equation}
Here we have use the properties that $\lambda^\mu_{\epsilon,\delta}$ is convex, $(\lambda^\mu_{\epsilon,\delta})'\geq 0$ and $(\chi^\kappa_{\delta,\rho}){''}=0$ on $[0,\delta)$, and the fact that $g^F=\frac{1}{2}\partial_{vv}L_0$ is positive definite.

\noindent \textbf{Case 2.} If $L_0(x,v)\geq\delta$, $|v|_x^2\geq\delta/A$ by (\ref{e:F&g}).  $\lambda^\mu_{\epsilon,\delta}$  equals to the affine function $\mu s+\sigma$ and $\chi^\kappa_{\delta,\rho}$ is non-decreasing on $[0,\infty)$, thus
\begin{eqnarray}\label{e:convexL3}
\partial_{vv}L_\eta(x,v)[u,u]&=&\mu\partial_{vv}L_0(x,v)[u,u]+
4(\chi^\kappa_{\delta,\rho}){''}(|v|_x^2)
g_x(v,u)^2+2(\chi^\kappa_{\delta,\rho})'(|v|_x^2)|u|_x^2\nonumber\\
&\geq&\mu\partial_{vv}L_0(x,v)[u,u]+
4(\chi^\kappa_{\delta,\rho}){''}(|v|_x^2)
|v|_x^2|u|_x^2\nonumber\\
&\geq&\mu |u|_x^2+\frac{4\delta}{A}(\chi^\kappa_{\delta,\rho}){''}(|v|_x^2)
|u|_x^2.
\end{eqnarray}

Since  $(\chi^\kappa_{\delta,\rho}){''}(s)=0$ for $s\in[\eta/A,\infty)$, and $(\chi^\kappa_{\delta,\rho}){''}(|v|_x^2)$ is bounded for $|v|_x^2\in[\delta/A,\eta/A]$, we may choose $\mu$ so large that
$$\mu+\frac{4\delta}{A}(\chi^\kappa_{\delta,\rho}){''}(|v|_x^2)>0.$$
This and (\ref{e:convexL3}) imply that $L_\eta$ satisfy (\textbf{L1}) for $L_0(x,v)\geq\delta$.

\end{proof}
\subsection{Construction of the profile functions, proof of Theorem ~\ref{thm:symhomology}}\label{SSProfile}
\begin{figure}[H]
  \centering
  \includegraphics[scale=0.25]{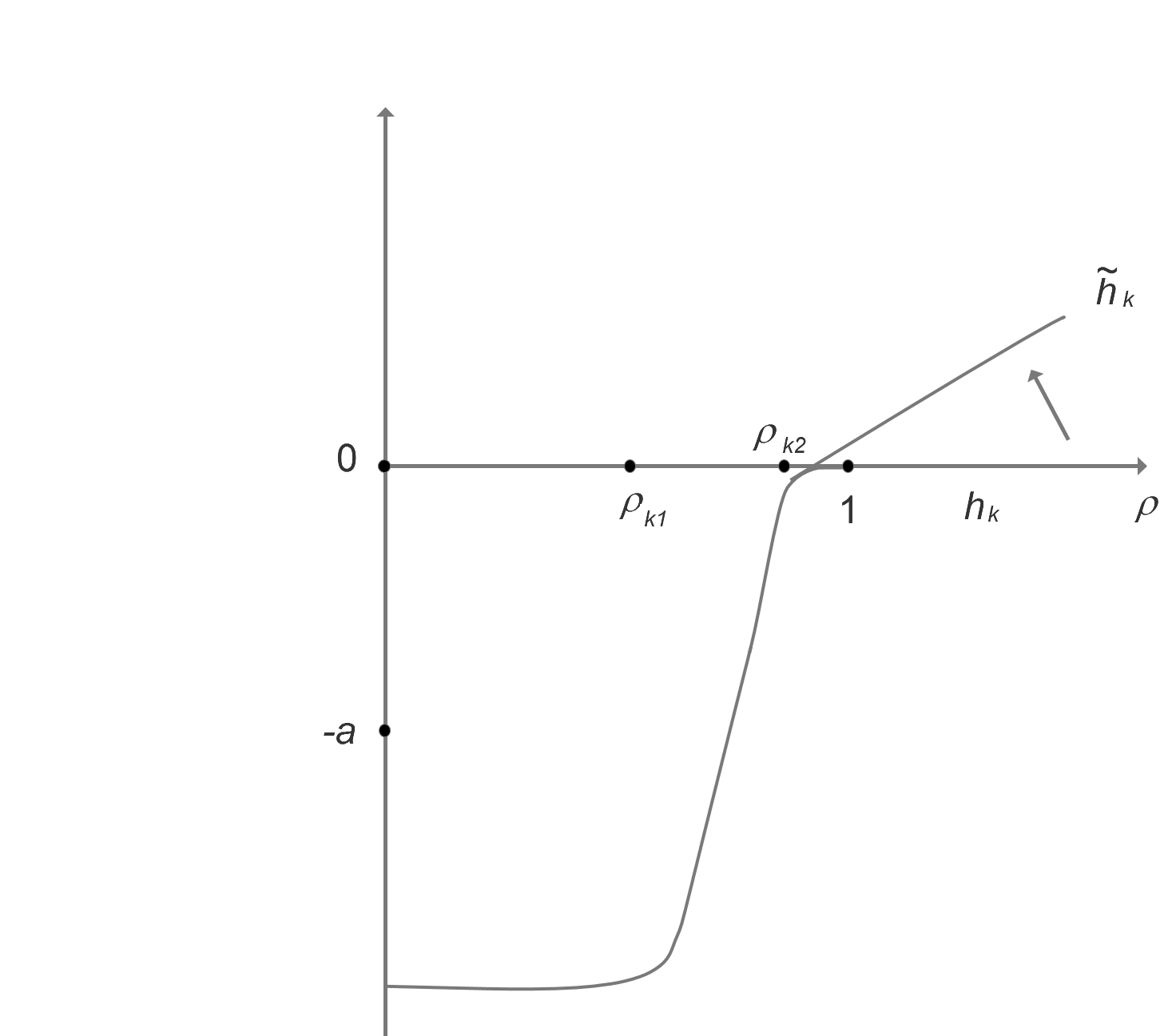}
  \caption{The monotone homotopy between $h_k$ and $\widetilde{h}_k$}\label{fig:3}
\end{figure}

\begin{proof}[Proof of Theorem~\ref{thm:symhomology}]
\noindent We prove Theorem~\ref{thm:symhomology} in three steps.

\noindent\textbf{Step 1.}\quad The idea of proving (i)  is to construct a downward exhausting sequence of profile functions $\{h_k\}_{k\in \mathbb{N}}$ and use Proposition~\ref{prop:T-hom} to compute symplectic homology. 

Fix $a\in\mathbb{R}\setminus \Lambda_\alpha$. The functions $h_k:[0,\infty)\to +\infty$ are smooth functions such that $h_k|_{[0,\rho_{k1}]}=h_k(0)$ and
$h_k|_{[1,+\infty)}=0$ with $h_k(0)\to-\infty$ and $\rho_{k1}\to 1$ as $k\to \infty$. Moreover, the derivatives are required to satisfy $h'_k\geq0$ everywhere, $h_k''\geq 0$ near $\rho_{k1}$, $h_k''\leq 0$ near $\rho_{k2}$ and $h''_k=0$ elsewhere, see Figure~\ref{fig:3}.

We prove that for each $k\in\mathbb{N}$, there is a natural isomorphism
\begin{equation}\label{e:isohk5}
\Psi_{h_k}^{a}:{\rm HF}^{(a,+\infty)}_*(h_k\circ F^*;\alpha)\longrightarrow
{\rm H}_*(\Lambda_\alpha^{a^2/2} M).
\end{equation}

We first consider a monotone homotopy, indicated in Figure~\ref{fig:3}, from $h_k\circ F^*$ to $\widetilde{h}_k\circ F^*$, where $\widetilde{h}_k$ is obtained by making the graph of $h_k$ linear with slope $a$ near $\rho_{k2}$.
By Lemma~\ref{lem:radialHamsyst}, no Hamiltonian $1$-periodic orbit with action in the action window $[a,\infty)$ appears during the homotopy, this gives the isomorphism
\begin{equation}\label{e:hmhk1}
\sigma_{\tilde{h}_kh_k}:{\rm HF}^{(a,+\infty)}_*(h_k\circ F^*;\alpha)\to
{\rm HF}^{(a,+\infty)}_*(\tilde{h}_k\circ F^*;\alpha).
\end{equation}

\begin{figure}[H]
  \centering
  \includegraphics[scale=0.25]{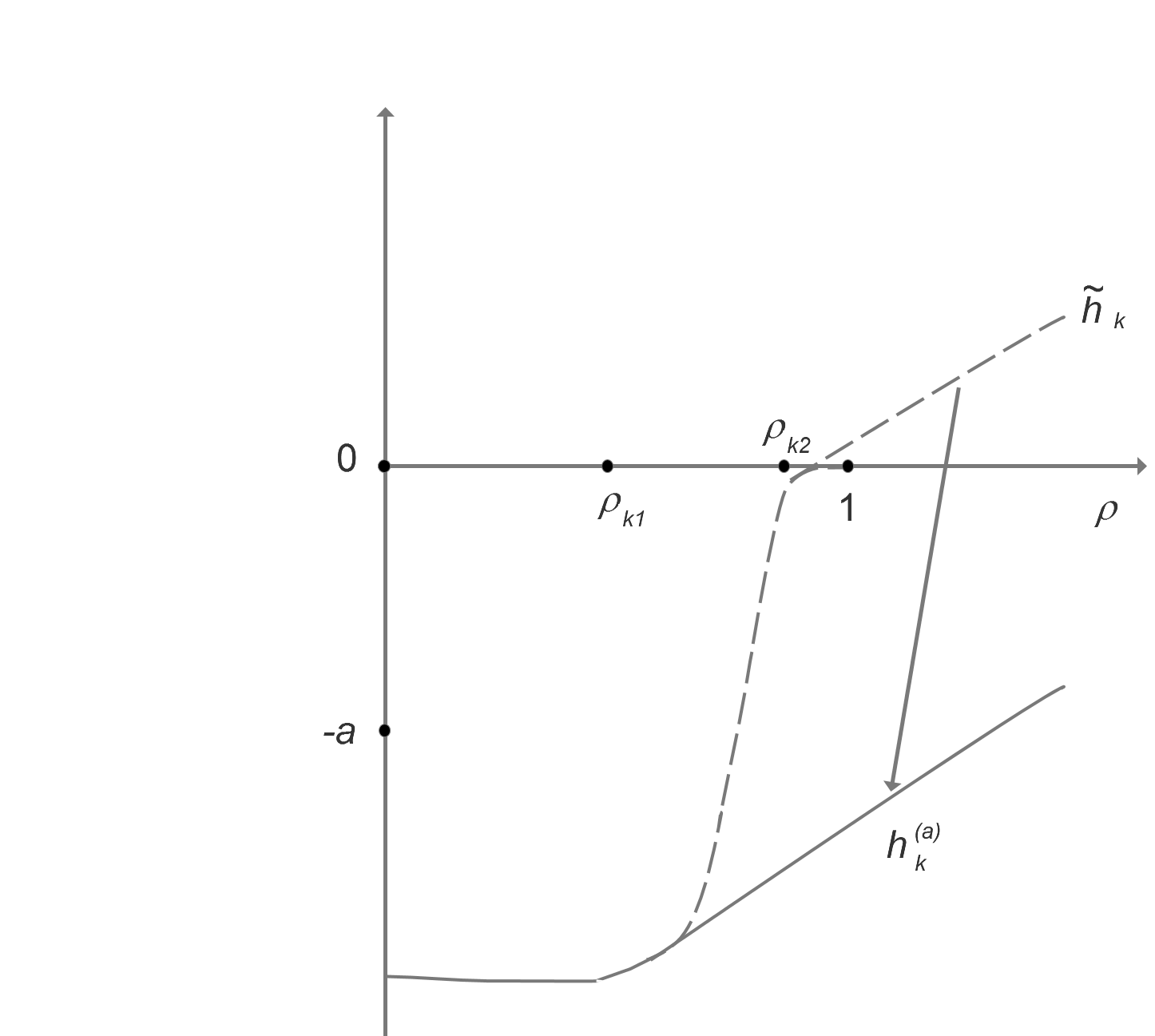}
  \caption{The monotone homotopy between $\widetilde{h}_k$ and $h_k^{(a)}$}\label{fig:4}
\end{figure}


Next, we consider another monotone homotopy from $h_k^{(a)}\circ F^*$ to $\widetilde{h}_k\circ F^*$ as shown in Figure~\ref{fig:4}. Here $h_k^{(a)}$ is obtained by making the graph of $h_k$ linear with slope $a$ near $\rho_{k1}$. This homotopy induces the monotone homomorphism

\begin{equation}\label{e:hmhk2}
\sigma_{\widetilde{h}_kh_k^{(a)}}:{\rm HF}^{(a,+\infty)}_*(h_k^{(a)}\circ F^*;\alpha)\longrightarrow {\rm HF}^{(a,+\infty)}_*(\widetilde{h}_k\circ F^*;\alpha),
\end{equation}
which is actually an isomorphism because no Hamiltonian $1$-periodic orbit has action $a$ during the homotopy. Observe that the minimal action of the Hamiltonian $1$-periodic orbit of $h_k^{(a)}\circ F^*$ is larger than $a$, and
the maximal action of the Hamiltonian $1$-periodic orbit of $h_k^{(a)}\circ F^*$ is less than $C_{h_k,a}:=c_{h_k^{(a)},a}$ (here $c_{f,\lambda}$ is  defined as in Theorem~\ref{thm:convexradial}). So we obtain the following isomorphisms

\begin{equation}\label{e:hmhk3}
{\rm HF}^{(a,+\infty)}_*(h_k^{(a)}\circ F^*;\alpha)\xrightarrow[{[j^F]^{-1}}]
{\cong}
{\rm HF}^{(-\infty,+\infty)}_*(h_k^{(a)}\circ F^*;\alpha)
\xrightarrow[{[i^F]^{-1}}]
{\cong} {\rm HF}^{(-\infty,C_{h_k,a})}_*(h_k^{(a)}\circ F^*;\alpha).
\end{equation}
For $a\notin\Lambda_\alpha$, Theorem~\ref{thm:convexradial} yields the isomorphism
\begin{equation}\label{e:isohk4}
\Psi_{h_k^{(a)}}^{a}: {\rm HF}^{(-\infty,C_{h_k,a})}_*(h_k^{(a)}\circ F^*;\alpha)\rightarrow {\rm H}_*(\Lambda_\alpha^{a^2/2} M).
\end{equation}
By (\ref{e:hmhk1}) -- (\ref{e:isohk4}), we arrive at the isomorphism~(\ref{e:isohk5}).

To conclude the proof of (i) in Theorem~\ref{thm:symhomology}, we check that the sequence $\{h_k\circ F^*\}_{k\in\mathbb{N}}$ is downward exhausting for the inverse limiting defining symplectic homology. Firstly, for any $H\in\mathscr{H}_\alpha^{a,+\infty}$, by choosing $k\in\mathbb{R}$ sufficiently large, one has $H(t,z)\geq h_k\circ F^*(z)$ for all $(t,z)\in S^1\times D^FT^*M$. Secondly, for every $k\in\mathbb{R}$, the monotone homomorphism
\begin{equation}\label{e:mhhk6}
\sigma_{h_kh_{k+1}}:{\rm HF}^{(a,+\infty)}_*(h_{k+1}\circ F^*;\alpha)\to
{\rm HF}^{(a,+\infty)}_*(h_k\circ F^*;\alpha)
\end{equation}
induced by a monotone homotopy between $h_{k+1}\circ F^*$ and $h_k\circ F^*$ is an isomorphism, since no Hamiltonian $1$-periodic orbit has action $a$ during the homotopy, for more details about this fact see~\cite[Section~3.2]{We0}. Taking the inverse limit in (\ref{e:isohk5}),
by Lemma~\ref{lem:inv/dirlimit}, we obtain the isomorphism $\underleftarrow{{\rm SH}}
^{(a,+\infty)}_*(D^FT^*M;\alpha)\cong {\rm H}_*(\Lambda_\alpha^{a^2/2} M)$.

\noindent\textbf{Step 2.}\quad To prove (ii), we construct an upward exhausting sequence of
functions $\{f_k\}_{k\in\mathbb{N}}$ to compute relative symplectic homology. Here let us emphasize that the functions $f_k$ are required to be constant near $\rho=0$, and thus are different from those functions $h_\delta$ constructed in~\cite{We0}. So we need to define $f_k$ and compute the filtered Floer homology of $f_k\circ F^*$ carefully. Fix $a,c>0$. Now we diskuss in two cases.

\begin{figure}[H]
	\centering
	\includegraphics[scale=0.25]{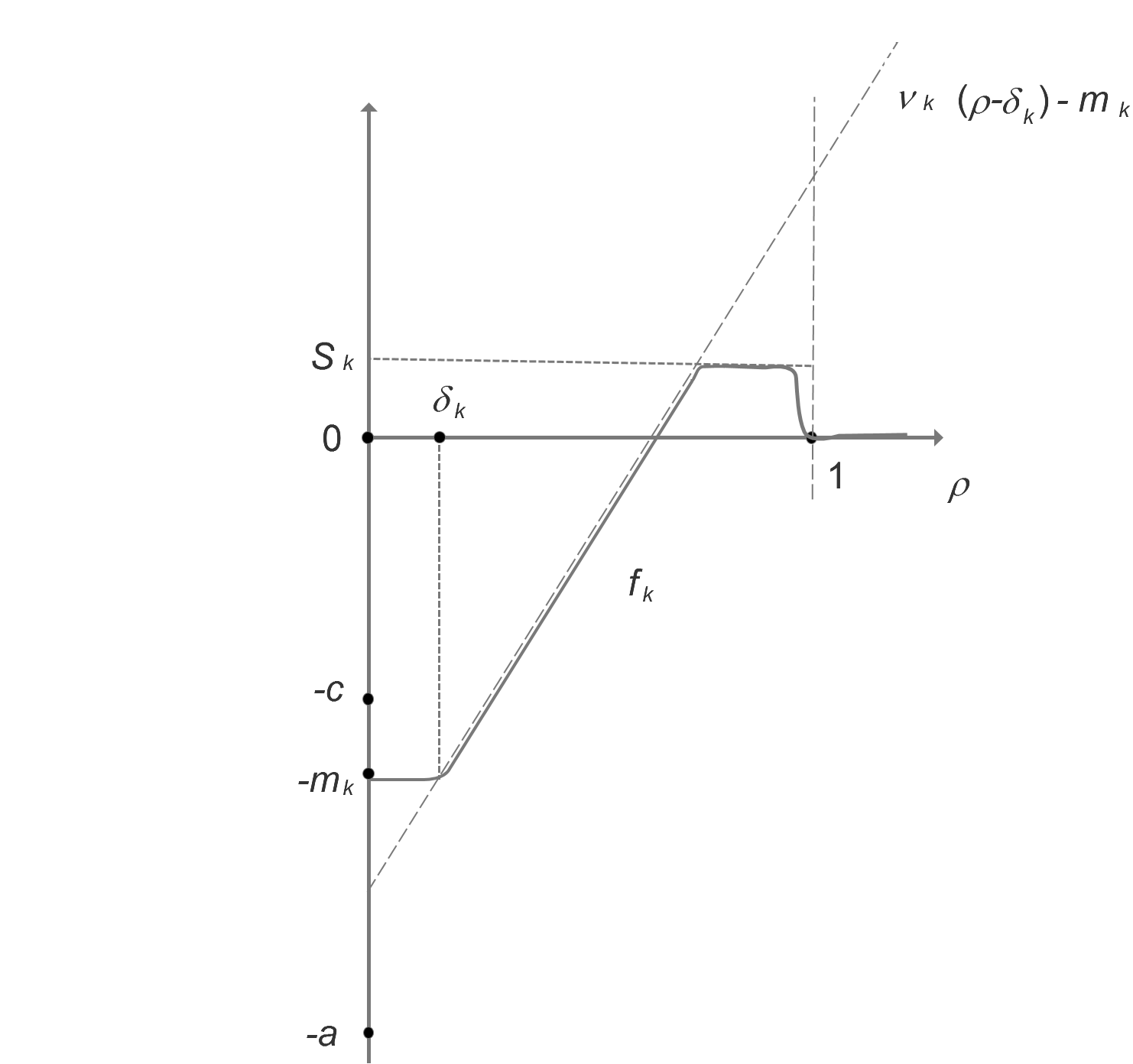}
	\caption{The function $f_k$ for $a>c>0$}\label{fig:5}
\end{figure}


\noindent \textbf{Case 1}. If $a>c>0$, we show $$\underrightarrow{{\rm SH}}^{(a,+\infty);c}_*(D^FT^*M,M;\alpha)=0.$$

Let $\{m_k\}_{k\in\mathbb{N}}$ be a monotone decreasing sequence of numbers such that $c<m_k<\frac{a+c}{2}$ and $m_k\to c$ as $k\to\infty$. Let $\{\delta_k\}_{k\in\mathbb{N}}$ be another monotone decreasing sequence of numbers such that $0<\delta_k<1/4$ and
$\delta_k\to 0$ as $k\to\infty$. For every $k\in\mathbb{N}$, set
$$\nu_k:=\frac{a-c}{2\delta_k},\quad S_k:=\max\bigg\{\frac{a-m_k}{2\sqrt{\delta_k}}-m_k,0\bigg\}.$$
Here we require in addition that the sequence $\{S_k\}_{k\in\mathbb{N}}$ is increasing and tends to $\infty$ as $k\to\infty$, which is possible by choosing a rapidly decreasing sequence $\{\delta_k\}_{k\in\mathbb{N}}$.

Next, we consider the piecewise linear curve in $\mathbb{R}^2$ obtained by beginning with a line segment with the starting point $(0,-m_k)$. Upon reaching the point $(\delta_k,-m_k)$, follow the line with slope $\nu_k$
until meeting the horizontal line through $(0,S_k)$, then follow this horizontal line to the right until it has a point of intersection with the vertical line through $(1,0)$, then go straight down to the point $(1,0)$, finally follow the horizontal coordinate axis to $+\infty$. Smooth out this piecewise linear curve near its corners so that it becomes the graph of a smooth function $f_k$, see Figure~\ref{fig:5}.

We claim that
\begin{equation}\label{e:homologfk1}
{\rm HF}^{(a,+\infty)}_*(f_k\circ F^*;\alpha)=0.
\end{equation}
In fact, since $f_k\circ F^*$ is a radial function with respect to $\rho$, and all tangents to the graph of $f_k$ intersects the vertical coordinate axis strictly above $-a$, the action of $1$-periodic orbits is strictly less than $a$. By~(\ref{e:homologfk1}), the monotone homomorphism
\begin{equation}\notag
\sigma_{f_{k+1}f_k}:{\rm HF}^{(a,+\infty)}_*(f_k\circ F^*;\alpha)\longrightarrow {\rm HF}^{(a,+\infty)}_*(f_{k+1}\circ F^*;\alpha).
\end{equation}
is obviously an isomorphism for every $k\in\mathbb{N}$. To show that $\{f_k\circ F^*\}$ is an upward exhausting sequence in
$\mathscr{H}_\alpha^{a,b;c}$, it suffices to check that for any $H\in \mathscr{H}_\alpha^{a,b;c}$, there exists $f_k$ such that $f_k\circ F^*(z)\geq H(t,z)$ for all $(t,z)\in S^1\times T^*M$. But this is clearly true by our construction. This, combining with (\ref{e:homologfk1}), completes the proof of Case 1.

\begin{figure}[H]
	\centering
	\includegraphics[scale=0.25]{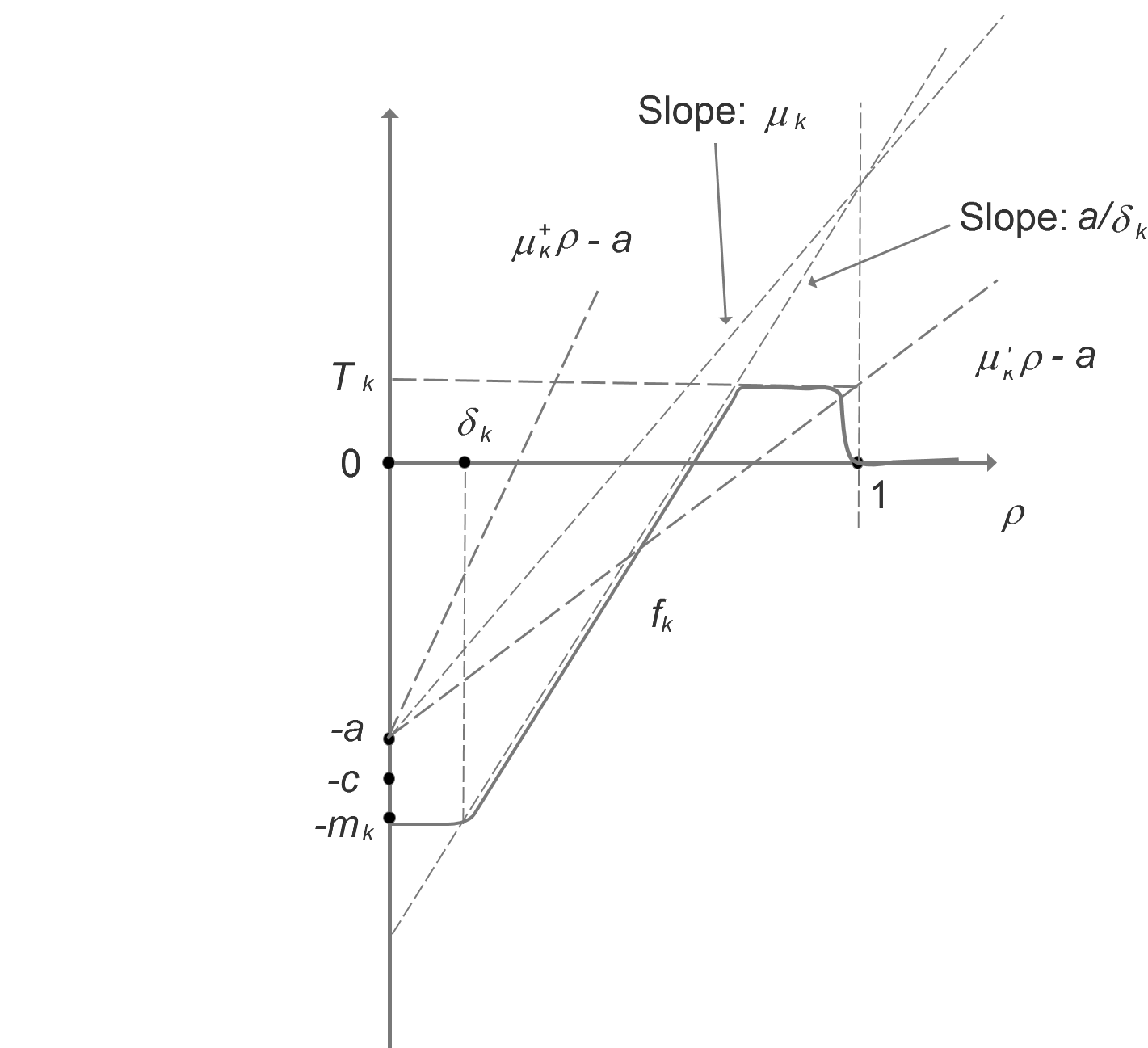}
	\caption{The function $f_k$ for $c\geq a$}\label{fig:6}
\end{figure}


\noindent \textbf{Case 2}. If $c\geq a>0$, we show
\begin{equation}\label{e:relsymplhom}
\underrightarrow{{\rm SH}}
^{(a,+\infty)}_*(D^FT^*M;\alpha)\cong {\rm H}_*(\Lambda_\alpha M).
\end{equation}

In the similar fashion to the proof of case 1, we construct an upward directed sequence to compute the relative symplectic homology.  Let $\{\delta_k\}_{k\in\mathbb{N}}$ be a decreasing sequence of numbers such that $0<\delta_k<a/c$,
$$\mu_k:=\frac{a}{\delta_k}-c\notin\Lambda_\alpha\quad\hbox{for all }k\in\mathbb{N}$$
and
$\delta_k\to 0$ as $k\to\infty$. Set
$$\mu^+_k:=\inf\big((\mu_k,\infty)\cap\Lambda_\alpha\big)
\quad \mu^-_k:=\sup\big((0,\mu_k)\cap\Lambda_\alpha\big), \quad\mu'_k:=\frac{\mu^-_k+\mu_k}{2}.$$

Let $T_k:=\min\{\mu'_k-a,0\}$ for every $k\in\mathbb{N}$, and $\{m_k\}_{k\in\mathbb{N}}$ be a decreasing sequence of numbers such that $m_k>c$ and $m_k\to c$ as $k\to\infty$. In the case of $c\geq a$, the shape of these functions $f_k$ are similar to those in case 1. Namely, $f_k$ are obtained by smoothing out a piecewise linear curve in $\mathbb{R}^2$. The difference here is to replace $S_k$ and $\nu_k$ by $T_k$ and $a/\delta_k$ respectively, see Figure~\ref{fig:6}.

Now we prove that there is a natural isomorphism
\begin{equation}\label{e:homologfk2}
\Psi_{f_k}^{\mu_k}:{\rm HF}^{(a,+\infty)}_*(f_k\circ F^*;\alpha)\longrightarrow {\rm H}_*(\Lambda_\alpha^{\mu_k^2/2} M).
\end{equation}
\begin{figure}[H]
	\centering
	\includegraphics[scale=0.25]{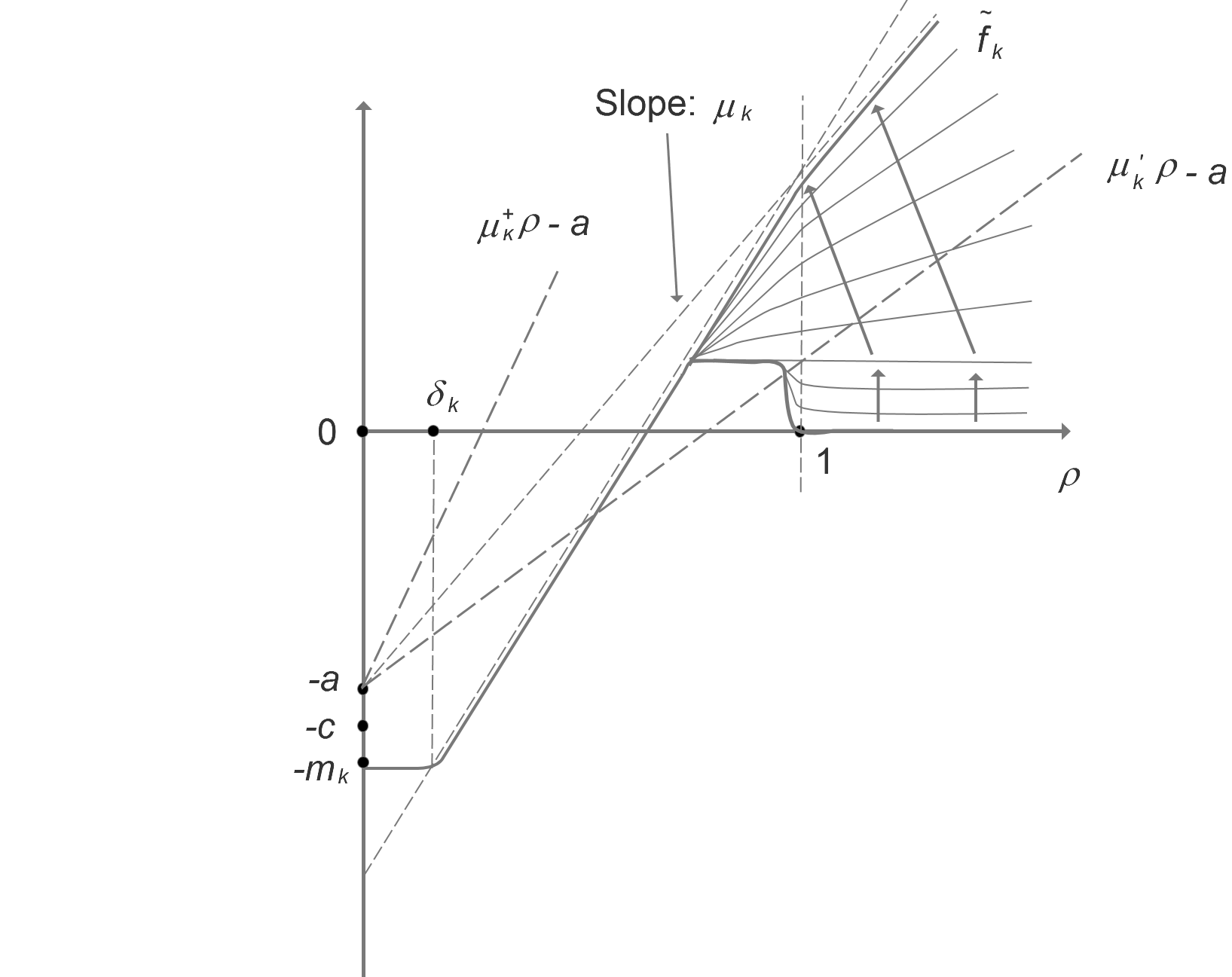}
	\caption{The function $f_k$ for $c\geq a$}\label{fig:7}
\end{figure}


To prove~(\ref{e:homologfk2}),  we initially deform $f_k$ by the monotone homotopy as showed
in Figure~\ref{fig:7} to a smooth function $\tilde{f}_k$, which is obtained by following the graph of $f_k$ until its first right turn. Then keep going linearly with  constant slope $a/\delta_k$. Upon meeting the line $\mu_k\rho-a$, turn right and follow that line closely and linearly as shown in Figure~\ref{fig:7}. Note that all points on the graph of the function during the homotopy, at which  tangential lines pass through the point $(0,-a)$, lie strictly between the lines $\mu_k\rho-a$ and $\mu_k'\rho-a$. Since $(\mu_k^-,\mu_k^+)\cap\Lambda_\alpha=\emptyset$, there are no $1$-periodic orbits of action $a$ during the monotone homotopy, and therefore, we obtain the monotone isomorphism
\begin{equation}\label{e:moiso1}
\sigma_{\tilde{f}_kf_k}:{\rm HF}^{(a,+\infty)}_*(f_k\circ F^*;\alpha)\longrightarrow {\rm HF}^{(a,+\infty)}_*(\tilde{f}_k\circ F^*;\alpha).
\end{equation}

\begin{figure}[H]
	\centering
	\includegraphics[scale=0.25]{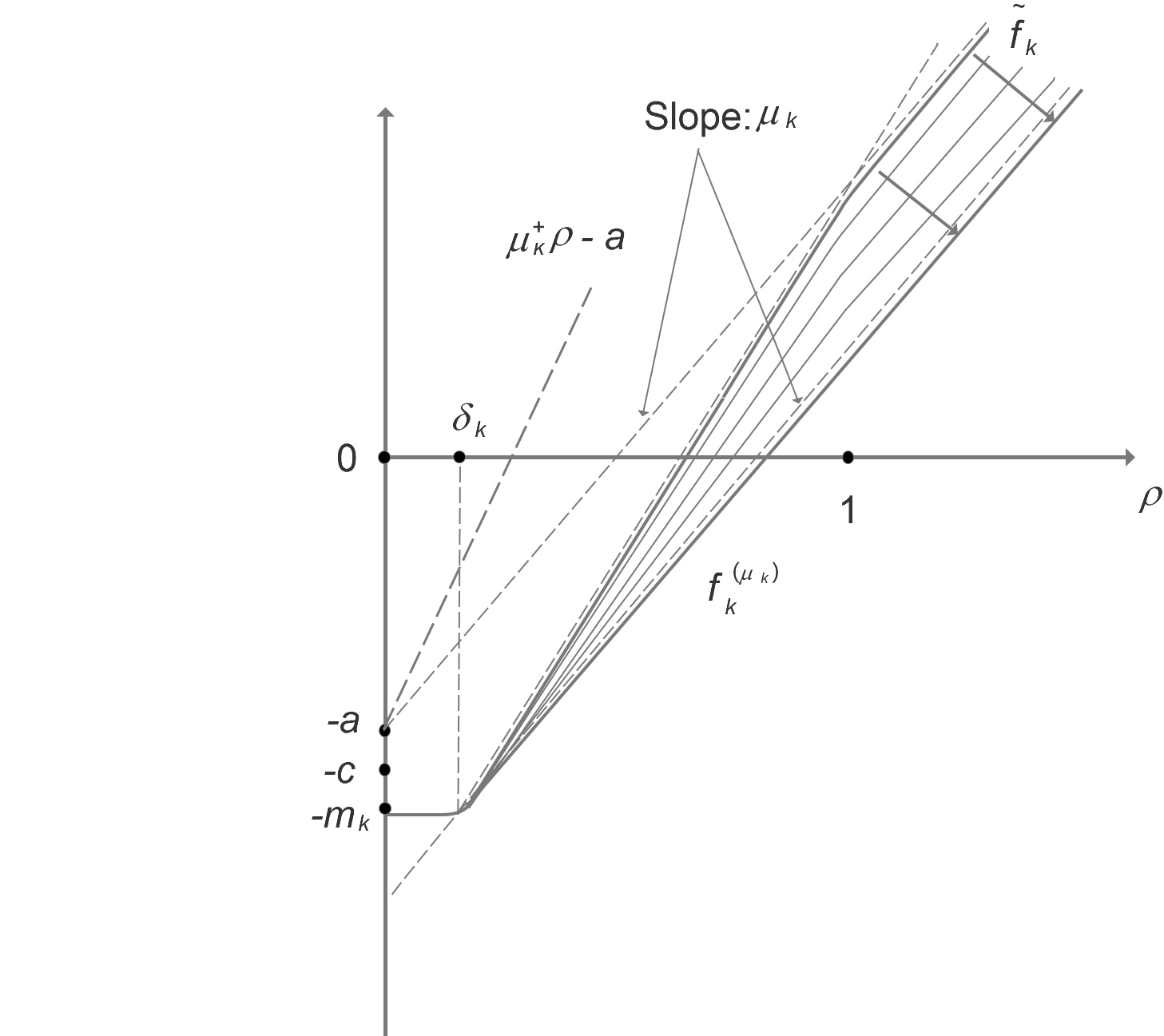}
	\caption{The function $f_k$ for $c\geq a$}\label{fig:8}
\end{figure}


Next, we deform $\tilde{f}_k$ by the monotone homotopy as indicated in Figure~\ref{fig:8} to a convex function $f^{(\mu_k)}_k$. The graph of this function is obtained by following $\tilde{f}_k$ until it takes on slope $\mu_k$ for the first time (near the point $(\delta_k,-m_k)$), and then making a smooth turn and continuing  linearly with slope $\mu_k$. Since the vertical axis intercepts of the tangential lines of the graphs during the
homotopy as showed in Figure~\ref{fig:8} are less than $-a$, we obtain the monotone isomorphism with inverse

\begin{equation}\label{e:moiso2}
\sigma_{f^{(\mu_k)}_k\tilde{f}_k}^{-1}:{\rm HF}^{(a,+\infty)}_*(\tilde{f}_k\circ F^*;\alpha)\longrightarrow {\rm HF}^{(a,+\infty)}_*(f^{(\mu_k)}_k\circ F^*;\alpha).
\end{equation}

The fact that $f^{(\mu_k)}_k\circ Q_{\eta_k}$ has neither $1$-periodic orbits of action less than $a$ nor  $1$-periodic orbits of action larger than $C_{f_k,\mu_k} :=c_{f_k^{(\mu_k)},\mu_k}$ implies the following isomorphisms

\begin{equation}\label{e:hmfk3}
{\rm HF}^{(a,+\infty)}_*(f_k^{(\mu_k)}\circ F^*;\alpha)\xrightarrow[{[j^F]^{-1}}]
{\cong}
{\rm HF}^{(-\infty,+\infty)}_*(f_k^{(\mu_k)}\circ F^*;\alpha)
\xrightarrow[{[i^F]^{-1}}]
{\cong} {\rm HF}^{(-\infty,C_{f_k,\mu_k})}_*(f_k^{(\mu_k)}\circ F^*;\alpha).
\end{equation}

Since $f_k^{(\mu_k)}\circ F^*$ is convex and radial with respect to $\rho$, Theorem~\ref{thm:convexradial} implies the isomorphism
\begin{equation}\label{e:isofk4}
\Psi_{f_k}^{\mu_k}: {\rm HF}^{(-\infty,C_{f_k,\mu_k})}_*(f_k^{(\mu_k)}\circ F^*;\alpha)\rightarrow {\rm H}_*(\Lambda_\alpha^{\mu_k^2/2} M).
\end{equation}
Composing (\ref{e:moiso1})--(\ref{e:isofk4}) yields the desired isomorphism~(\ref{e:homologfk2}).

By our construction, for any $H\in \mathscr{H}_\alpha^{a,b;c}$, we can find some $f_k\circ F^*\in \mathscr{H}_\alpha^{a,b;c}$ such that $H\preceq f_k\circ F^*$, and for any $k\in\mathbb{N}$ it holds that $f_k\circ F^*\preceq f_{k+1}\circ F^*$. Thus, $f_k\circ F^*$ is an upward directed sequence in $\mathscr{H}_\alpha^{a,b;c}$. Taking the direct limit of both side of~(\ref{e:homologfk2}) yields the desired isomorphism~(\ref{e:relsymplhom}).

\noindent\textbf{Step 3.}\quad
Given $a\in (0,c]\setminus\Lambda_\alpha$, we adopt the notation used in Step 1 and Step 2, and choose $k\in\mathbb{N}$ sufficiently large so that $\mu_k>a$, $h_k\in\mathscr{H}_\alpha^{a,b;c}$ and $f_k\geq h_k$.
Following closely~\cite[Section~3.2]{We0}, the proof of the commutativity of the diagram
\begin{equation}\label{CD:diag8}
\begin{split}
\xymatrix{
    \\
     \underleftarrow{{\rm SH}}^{(a,+\infty)}_*(D^FT^*M;\alpha)
     \ar[r]^{\qquad\cong}
     \ar[d]_{T^{(a,\infty);c}_\alpha}
    &
     {\rm H}_*(\Lambda_\alpha^{a^2/2} M)
    \ar[d]^{[I_{a^2/2}]}
    \\
     \underrightarrow{{\rm SH}}^{(a,+\infty);c}_*(D^FT^*M,M;\alpha)
     \ar[r]^{\qquad\quad\cong}
    &
     {\rm H}_*(\Lambda_\alpha M)
}
\end{split}
\end{equation}
\noindent reduces to show the commutativity of the diagram
\begin{equation}\label{CD:diag9}
\begin{split}
\xymatrix{
	{\rm HF}^{(a,\infty)}_*(h_k\circ F^*;\alpha)
	\ar[r]^{\sigma_{f_k h_k}}
	\ar[d]_{\Psi_{h_k}^{a}}^\cong
	&
	{\rm HF}^{(a,\infty)}_*
	(f_k\circ F^*;\alpha)
	\ar[d]^{\Psi_{f_k}^{\mu_k}}_\cong
	\\
	{\rm H}_*(\Lambda_\alpha^{a^2/2} M)
	\ar[r]^{[I]}
	&
	{\rm H}_*(\Lambda_\alpha^{\mu_k^2/2} M)
}
\end{split}
\end{equation}
\noindent This can be deduced from the following commutative diagram
\begin{equation*}\label{CD:diag10}
\begin{split}
\xymatrix{
	{\rm HF}^{(a,\infty)}_*(h_k\circ F^*;\alpha)
	\ar[r]^\sigma
	\ar[d]_\sigma^{(\ref{e:hmhk1})}
	&
	{\rm HF}^{(a,\infty)}_*
	(f_k\circ F^*;\alpha)
	\ar[d]^\sigma_{(\ref{e:moiso1})}
	\\
	{\rm HF}^{(a,\infty)}_*(\tilde{h}_k\circ F^*;\alpha)
	\ar[r]^\sigma
	&
	{\rm HF}^{(a,\infty)}_*(\tilde{f}_k\circ F^*;\alpha)
	\\
	{\rm HF}^{(a,\infty)}_*(h_k^{(a)}\circ F^*;\alpha)
	\ar[u]^\sigma_{(\ref{e:hmhk2})}
	\ar[r]^\sigma
	&
	{\rm HF}^{(a,\infty)}_*(f^{(\mu_k)}_k\circ F^*;\alpha)
	\ar[u]_\sigma^{(\ref{e:moiso2})}
	\\
	{\rm HF}^{(-\infty,\infty)}_*(h_k^{(a)}\circ F^*;\alpha)
	\ar[u]^{[j^F]}_{(\ref{e:hmhk3})}
	\ar[r]^\sigma
	&
	{\rm HF}^{(-\infty,\infty)}_*(f^{(\mu_k)}_k\circ F^*;\alpha)
	\ar[u]_{[j^F]}^{(\ref{e:hmfk3})}
	\\
	{\rm HF}^{(-\infty,C_{h_k,a})}_*(h_k^{(a)}\circ F^*;\alpha)
	\ar[u]^{[i^F]}_{(\ref{e:hmhk3})}
	\ar[r]^{\Psi_{fh}^a}_{(\ref{CD:diag2})}
	\ar[dr]_{\Psi_{h_k}^{a},(\ref{e:isohk4})}
	&
	{\rm HF}^{(-\infty,C_{f_k,a})}_*(f^{(\mu_k)}_k\circ F^*;\alpha)
	\ar[u]_{[i^F]}
	\ar[r]^{[i^F]}
	\ar[d]^{\Psi_{f_k}^{a}}_\cong
	&
	{\rm HF}^{(-\infty,C_{f_k,\mu_k})}_*(f^{(\mu_k)}_k\circ F^*;\alpha)
	\ar[ul]_{[i^F],(\ref{e:hmfk3})}
	\ar[d]^{\Psi_{f_k}^{\mu_k}}_{(\ref{e:isofk4})}
	\\
	&
	{\rm H}_*(\Lambda_\alpha^{a^2/2} M)
	\ar[r]^{[I]}
	&
	{\rm H}_*(\Lambda_\alpha^{\mu_k^2/2} M)
}
\end{split}.
\end{equation*}
where $\Psi_{fh}^a$ is an isomorphism by Theorem~\ref{thm:convexradial} (iii). In fact,
Lemma~\ref{lem:mh} implies the commutativity of the first two rectangular blocks. The third and fourth block commutes by the naturality of long exact sequences concerning Floer homology. The lowest rectangular block commutes by~(\ref{DC:diag0.1}), and the triangle to its left commutes by~(\ref{trangle:diag0.2}). The remaining part of the proof proceeds exactly like~\cite[Section~3.2]{We0} by taking direct and inverse limits in the diagram~(\ref{CD:diag9}).
The proof of Theorem~\ref{thm:symhomology} (iii) completes.

\end{proof}

\end{document}